\documentclass[11pt,reqno]{amsart}
\usepackage{amsthm, amsfonts, amssymb, color}
\usepackage{mathrsfs}
\usepackage{enumerate}
\usepackage{amsmath}
\usepackage{amstext, amsxtra}
\usepackage{txfonts}
\usepackage[colorlinks, linkcolor=black, citecolor=blue, pagebackref, hypertexnames=false]{hyperref}
\usepackage{graphicx}
\usepackage[symbol]{footmisc}
\usepackage{latexsym}
\usepackage{tikz}
\usepackage{cite}
\allowdisplaybreaks

\newtheorem{theorem}{Theorem}[section]
\newtheorem{proposition}[theorem]{Proposition}
\newtheorem{corollary}[theorem]{Corollary}
\newtheorem{lemma}[theorem]{Lemma}

\newtheorem{remark}[theorem]{Remark}

\newtheorem{assumption}{Assumption}

\newcommand\N{\mathbb{N}}
\newcommand\R{\mathbb{R}}

\newcommand\Z{\mathbb{Z}}

\newcommand{\supp}{{\rm supp}{\hspace{.05cm}}}

\def\R {\mathbb{R}}
\def\N {\mathbb{N}}

\def\d{{\rm d}}

\def\i{{\rm i}}
\def\k{{ \textbf{k}}}

\def \and {{\qquad\text{and}\qquad}}

\numberwithin{equation}{section}
\theoremstyle{definition}

\title[Pointwise estimates for higher order schr\"{o}dinger equations]
{POINTWISE ESTIMATES FOR THE FUNDAMENTAL solutions OF HIGHER ORDER Schr\"{o}dinger equations with finite rank perturbations}

\author{Xinyi Chen,\, Han Cheng,\, Shanlin Huang}

\address{Xinyi Chen, School of Mathematics and Statistics, Huazhong University of Science and Technology, Wuhan 430074, Hubei, PR China }
\email{xinyi\_c@hust.edu.cn}

\address{Han Cheng,  Institute of Applied Physics and Computational Mathematics, Beijing, 100088, China.}
\email{chh@hust.edu.cn}

\address{Shanlin Huang, 
	School of Mathematics (Zhuhai), Sun Yat-sen University, Zhuhai 519082, Guangdong, China}
\email{shanlin\_huang@hust.edu.cn}


\begin{document}
	
	\begin{abstract}
		This  paper is dedicated to studying  pointwise estimates of the fundamental solution for the  higher order Schr\"{o}dinger equation:
		$$\i{\partial}_{t}u(x,t)=Hu(x,t),\ \ \ t\in \mathbb{R},\ x\in {\mathbb{R}}^{n},$$
		where the Hamiltonian $H$ is defined as
		$$H={(-\Delta)}^{m}+\displaystyle\sum_{j=1}^{N}  \langle\cdotp ,{\varphi }_{j}  \rangle{\varphi }_{j},$$
		with each $\varphi_j$  ($1\le j\le N$) satisfying certain smoothness and decay conditions.
		We show that for any positive integer $m>1$ and  spatial dimension $n\ge 1$, 
		${e}^{-\i tH}P_{ac}(H)$ 
		has an integral kernel $K(t,x,y)$ satisfying the following pointwise estimate:
		$$\left |K(t,x,y)\right |\lesssim |t|^{-\frac{n}{2m}}(1+|t|^{-\frac{1}{2m}}\left | x-y\right |)^{-\frac{n(m-1)}{2m-1}} ,\ \ t\ne 0,\ x,y\in {\mathbb{R}}^{n}.$$
		This estimate is consistent with the upper bounds in the free case. As an application, we derive 
		$L^p-L^q$ decay estimates for the propagator ${e}^{-\i tH}P_{ac}(H)$, where the  pairs $(1/p, 1/q)$ lie within  a quadrilateral region in the plane.
	\end{abstract}
	\maketitle
	
	
	\tableofcontents

	
	
	\section{Introduction and main results}\label{sec1}

	\subsection{Background and motivation}\label{sec1.0}\
	
	In this paper, we investigate pointwise estimates of fundamental solutions for higher order Schr\"{o}dinger equations under finite rank perturbations. The equations are given by
	\begin{equation}\label{equ1.1.1}
		\i\partial_tu(x, t)=H_N u(x, t), \quad t\in \mathbb{R},\,\,x\in \mathbb{R}^n,
	\end{equation}
	where
	\begin{equation}\label{eq-high-per}
		H_N=(-\Delta)^{m}+\sum_{j=1}^NP_j, \,\,\,\,\,\,\, P_j=\langle\cdot\,, \varphi_j\rangle \varphi_j,
	\end{equation}
	and  $\varphi_j\in L^2$, $1\le j\le N$. Here and in the rest of the paper, $\langle\cdot\,, \cdot\rangle$ denotes the usual inner product in $L^2(\mathbb{R}^n)$.
	
	The fundamental solution is a basic tool in partial differential equations and it plays a central role in developing  potential and regularity theory
	for solutions of related PDEs.
	We mention in particular the recent paper \cite{MM} where fundamental solution is  used in the   regularity problem of poly-harmonic  boundary value problems and \cite{BHM} for applications to  the  theory of layer potentials.
	
	In the second order case, it is well known that the fundamental solution of the free Schr\"{o}dinger equation is given by
	\begin{equation}\label{eq1.1.2}
		e^{\i t \Delta}(x, y):=(4\pi \i t)^{-\frac{n}{2}}e^{\frac{\i|x-y|^2}{4t}},\qquad t\ne 0,\,\,x, y\in \mathbb{R}^n.
	\end{equation}
	It implies immediately the following  dispersive estimate
	\begin{equation}\label{eq-disp}
		\|e^{\i t\Delta}\|_{L^1-L^{\infty}}\leq C|t|^{-\frac{n}{2}},\qquad\,\,t\ne 0.
	\end{equation}
	Motivated by nonlinear problems, the extension of \eqref{eq-disp} to the Schr\"{o}dinger operators $-\Delta+V(x)$ has evoked considerable interest in the past three decades,  and we refer to \cite{JSS,Ya,RS} for pioneering  works on this theme as well as the surveys \cite{Sch07,Sch21}.
	%
	
	Higher order Schr\"{o}dinger equations are important in modeling spinless, weakly relativistic particles and quantum mechanical systems. Notably, the bi-Laplacian case was introduced by Karpman-Shagalov \cite{KS} to investigate the stabilizing role of higher-order dispersive effects on soliton instabilities in light propagation, For further related research, we refer to \cite{CLM,CM}.
	
	From the  perspective  of fundamental solutions,   the higher-order case does not admit a simple closed-form expression as \eqref{eq1.1.2}. 
	A further  distinction between  higher-order and second-order
	Schr\"{o}dinger equation is the presence of additional  spatial decay.
	Indeed, it is known that (see e.g. in \cite{Mi}):
	\begin{equation}\label{eq-high-free-p}
		|e^{-\i t (-\Delta)^m}(x, y)|\le C|t|^{-\frac{n}{2 m}}\left(1+|t|^{-\frac{1}{2 m}}|x-y|\right)^{-\frac{n(m-1)}{2 m-1}},\qquad t\ne 0,\,\,x, y\in \mathbb{R}^n.
	\end{equation}
	We refer to \cite{HHZ,KPV,Mi} for fundamental solution estimates of $e^{\i t P(D)}$  with general elliptic operator $P(D)$.

	On the other hand, the theory of finite rank perturbations originates from Weyl's seminal work in
	1910  \cite{Wey}, where these perturbations were first utilized to determine the spectrum of Sturm-Liouville operators.
	They also  arise in  various  mathematical physics  problems,
	for example, Simon and Wolff \cite{SW} found a relationship between rank one perturbation and discrete random Schr\"{o}dinger operators, and applied it to the Anderson localization problem.
	It deserves to mention that despite the formal simplicity, the study of this restrictive perturbation is very rich and interesting, we refer to \cite{Simon} for the spectral analysis and the survey paper \cite{Liaw} for other problems.
	
	Nier and Soffer \cite[Theorem 1.1]{NS} first investigated 
	$L^p$  decay estimates in the second order case.  Specifically, they showed that when  $m=1$ in \eqref{eq-high-per} and $n\ge 3$,
	\begin{equation}\label{eq1.3.1}
		\|e^{-itH}\|_{L^p-L^{p'}}\leq C(N, n, \varphi_1,\cdots, \varphi_N)\cdot t^{-n(\frac{1}{p}-\frac{1}{2})},\quad\,\,\,\,t>0,\quad \,1<p\leq 2
	\end{equation}
	holds under suitable conditions on each $\varphi_j$. This was  later extended 
	in \cite{CHZ}   to include the endpoint case $p=1$ and lower spatial dimensions $n=1, 2$.
	
	{\bf Aim and motivation.}  In light of the existing results for the second order case mentioned above,  the following question arises naturally:  Is the pointwise estimate \eqref{eq-high-free-p} still valid in the presence of finite rank perturbations?
	Our goal is to  establish the pointwise estimates of the  fundamental solution of \eqref{equ1.1.1}-\eqref{eq-high-per} for all spatial dimensions $n\ge 1$ and for  every pair of positive integers $m>1$ and $N\ge 1$.

	\subsection{Rank one perturbations}\label{sec1.1}\
	
	We begin by considering rank one perturbations ($N=1$ in \eqref{eq-high-per}). More precisely, let
	\begin{equation}\label{eq1.1}
		{H}={H}_{0}+\alpha \left \langle \cdotp ,\varphi \right \rangle\varphi ,\ \ {H}_{0}={(-\Delta)}^{m},\ \ \alpha >0,
	\end{equation}
	where $\varphi\in L^2$.
	Before stating our assumptions on $\varphi$, we highlight the crucial role played by the following  function
	\begin{equation}\label{eq1.1.1-F(z)}
		F(z):=\langle((-\Delta)^m-z)^{-1}\varphi,\varphi\rangle=\int{\frac{d\mu_{H_0}^{\varphi}(\lambda)}{\lambda-z}},
	\end{equation}
	where $z\in \mathbb{C}\setminus[0, \infty)$. This can be viewed as the Borel transform of the spectral measure $d\mu_{H_0}^{\varphi}$ for $\varphi$ (here $\mu_{H_0}^{\varphi}((a, b))=\langle E_{(a,b)}(H_0)\varphi,\varphi\rangle$). Moreover, if $\varphi\in L^{2}_{\sigma}$  with $\sigma>\frac12$, where $L^{2}_{\sigma}:=\{(1+|x|)^{-\sigma}f:\, f\in L^2\}$, then by the well-know limiting absorption principle (see e.g. in \cite{KK}), the boundary values
	\begin{equation}\label{eq1.11.1}
		F^\pm(\lambda):=\langle((-\Delta)^m-\lambda\mp \i0)^{-1}\varphi,\varphi\rangle,
	\end{equation}
	are defined when $\lambda\ne 0$ and coincide on $(-\infty, 0)$.
	
	\begin{assumption}\label{assumption-1}\,\, Let  $\alpha>0$ and let $\varphi$ be a real-valued, $L^2$-normalized  function on $\mathbb{R}^n$ such that the following conditions hold:
		
		\noindent\emph{(\romannumeral1) } (Decay and smoothness) \,\,There exists a constant $M>0$ such that
		\begin{equation}\label{eq1.2}
			|\varphi(x)|\leq M\langle x \rangle^{-\delta},
		\end{equation}
		where
		\begin{equation}\label{eq2}
			\delta >\left\{
			\begin{aligned}
				&\ n+3/2,   &n \ge 2m,\\
				&\ 2m+1,  &1\le n<2m.
			\end{aligned}
			\right.
		\end{equation}
		Moreover, when  $n>4m-1$, we assume that $\varphi(x)\in C^{\beta_0}(\mathbb{R}^n)$,  where $\beta_0=[\frac{n+5}{4}]-m$, and
		\begin{equation}\label{eq1.3}
			|\partial^{\beta}_x\varphi(x)|\leq M\langle x \rangle^{-\delta}, \,\,\,\,\,|\beta|\leq \beta_0,\,\,\,\,\,  \delta>n+\frac{3}{2}.
		\end{equation}
		
		\noindent\emph{(\romannumeral2) } (Spectral assumption) \,\, There is an absolute constant $c_0>0$ such that
		\begin{equation}\label{eq1.4}
			|1+\alpha F^\pm(\lambda^{2m})|\ge c_0,\,\,\,\,\,\,\text{for all}\,\,\,\lambda>0,
		\end{equation}
		where $F^\pm(\lambda^{2m})$ is given by \eqref{eq1.11.1}.
	\end{assumption}
	
	Nier and Soffer \cite[Lemma 2.6]{NS} showed that the assumption \eqref{eq1.4}, together with some decay condition on $\varphi$, yields that the spectrum of $H$ is absolutely continuous for all $n\ge 3$. Furthermore, it is proved in \cite[Lemma B.1]{CHZ} that this is true for all $n\ge1$.
	
	We state our first  result as  follows.
	
	\begin{theorem}\label{thm-1.1}
		Let $n\ge 1$ and let $H$ be given by \eqref{eq1.1}. Suppose that Assumption \ref{assumption-1} is satisfied, then there exists an absolute constant $C=C(\alpha ,\varphi)>0$ such that ${e}^{-\i tH}$ has integral kernel satisfying
		\begin{equation}\label{eq1.5}
			\left |K(t,x,y)\right |\le C|t|^{-\frac{n}{2m}}(1+|t|^{-\frac{1}{2m}}|x-y|)^{-\frac{n(m-1)}{2m-1}} ,\ \ t\ne 0,\,\,x, y\in \mathbb{R}^n.
		\end{equation}
	\end{theorem}
	We make the following remarks related to Theorem \ref{thm-1.1}.
	
	As far as we are aware, Theorem  \ref{thm-1.1} seems to provide the first pointwise estimates for the fundamental solution  of ${e}^{-\i tH}$ ($H$ is given by \eqref{eq1.1}) in any dimension $n\ge1$. Moreover, the upper bound in \eqref{eq1.5} is also sharp compared to the free case \eqref{eq-high-free-p}.
	
	
	As an immediate application, Theorem \ref{thm-1.1}  implies the following dispersive estimate
	\begin{equation}\label{equ-disp-h}
		\|e^{-\i tH}\|_{L^1-L^{\infty}}\lesssim|t|^{-\frac{n}{2m}},\qquad\,\,t\ne 0,
	\end{equation}
	as well as the  Strichartz-type estimate (see e.g. in \cite{KT}): For $n>2m$, and let $\frac{2m}{q}+\frac{n}{r}=\frac{n}{2}$ and $2\le q\le \infty$, then
	\begin{equation}\label{equ-stri}
		\|e^{-\i tH}u_0\|_{L_t^qL_x^r(\mathbb{R}\times \R^n)}\lesssim \|u_0\|_{L^2(\R^n)}.
	\end{equation}
	
	In recent years, there has been a surge of interest in the dispersive estimates for  higher order Schr\"{o}dinger operators $H=(-\Delta)^{m}+V$.
	Local dispersive estimates have been established by  Feng, Soffer, and Yao  \cite{FSY}, as well as Feng, Soffer, Wu, and Yao  \cite{FSWY}.
	For  fourth-order Schr\"{o}dinger equations, Erdo\v{g}an, Green and Toprak \cite{EGT} derived global dispersive estimates for three-dimensional cases, with Green and Toprak \cite{GT} later extending these results to four-dimensional scenarios. Subsequent work by Soffer, Wu, and Yao \cite{SWY}, along with Li, Soffer, and Yao \cite{LSY}, demonstrated dispersive bounds in 1-D and 2-D settings, respectively. Moreover, Erdo\v{g}an, Goldberg, and Green \cite{EGG822}  established dispersive estimates for scaling-critical potentials under the condition   $2m<n<4m$ and $m>1$. For $n>2m$, Erdo\v{g}an and Green \cite{EG22,EG23}   obtained dispersive estimates by examining  the $L^p$ boundedness of wave operators. Most recently,  pointwise estimates for the fundamental solution of $e^{-\i tH}$ in odd dimensions were proved in \cite{CHHZ}.

	We briefly outline the strategy employed in the proof of Theorem \ref{thm-1.1}.
	Inspired by the work on Schr\"{o}dinger operators, our starting point is to represent the propagator  via the Stone's formula. The peculiarity of the rank one perturbation is the so-called Aronszajn-Krein formula (see \eqref{A-K-for}).  As usual, we treat the high and low  energy parts respectively. One of the key ingredients throughout the proof is to  take suitable  space-time decompositions when estimating  various oscillatory integrals  appeared in both the high and low  energy regimes. This strategy ultimately leads to a simplification of the problem to the analysis  of certain 1-dim oscillatory integrals, as detailed in Lemmas \ref{lm2.3} and \ref{lm2.4}. 
	
	In the high energy part, we  write the kernel of $e^{-\i tH}-e^{\i tH_{0}}$ in the form
	\begin{equation}\label{eq1.6.1}
		\int_{\mathbb{R}^{2n}}{I_n(t, |x-x_1|, |x_2-y|)G_1(x, x_1)G_2(x_2, y)\,\d x_1\d x_2},
	\end{equation}
	where $I_n(t,|x-x_1|,|x_2-y|)$  stands for the following type of oscillatory integrals
	\begin{equation}\label{eq1.6.2}
		\int_{r_0}^{\infty} e^{-{\rm i}t{\lambda }^{2m}\pm {\rm i}\lambda \left (| x-{x}_{1}|+|{x}_{2}-y|\right )}\frac{\psi(\lambda,|x-x_1|, |x_2-y|)}{1+\alpha F^{\pm}(\lambda^{2m})}\d \lambda,
	\end{equation}
	where $r_0>0$ is some fixed constant. 
	The explicit expression of $G_1, G_2,$ and $\psi$ depend on the dimension.
	Since the behavior of the free resolvent differs in different dimensions, we shall divide the analysis into two different cases: $n<4m$ and $n\ge 4m$. When $n\ge 4m$,   the  magnitude of  $\lambda$  in \eqref{eq1.6.2} may become excessively large. To overcome this difficulty,  we borrow ideas from \cite{EG10, CHZ} and employ integration by parts arguments with respect to the oscillating term $e^{\i \lambda(|x-x_1|+ |x_2-y|)}$.

	In the low-energy regime, the symbol $|\xi|^{2m}$ ($m>1$) is degenerate at $\xi=0$  when compared to the the second order case, and the asymptotic behavior of the resolvent $R_0^{\pm}(\lambda^{2m})$ at zero energy is significantly more intricate. Further, the function $\frac{1}{1+\alpha F^{\pm}(\lambda^{2m})}$ exhibits distinct asymptotic behaviors near zero,
	depending on whether $n>2m$ and $n\le 2m$. It deserves to mention that when $n\le 2m$, the behavior is also determined by the vanishing moment conditions of $\varphi$, particularly whether  $\int_{\mathbb{R}^d}{x^{\gamma}\varphi}=0$ for some $\gamma\in \N_0^{n}$. 
	To establish the asymptotic properties in this case, we employ appropriate expansions of the free resolvent (see Lemma \ref{lm3.4}-\ref{lm3.7}). Moreover, in the process of  estimating the fundamental solution, we shall rewrite $R_0^{\pm}(\lambda^{2m})\varphi$ in the following form  
	\begin{equation*}
		(R_0^{\pm}(\lambda^{2m}) \varphi)(x)=e^{ \pm \i\lambda|x|}W^\pm(\lambda, x),
	\end{equation*}
	where  we isolate the oscillatory component $e^{ \pm \i\lambda|x|}$. This separation allows us to derive pointwise estimates for $W^\pm$ and its derivatives (see Proposition \ref{lm-res-low-nl2m-1}, \ref{lm-n<2m-odd-vanish} and \ref{lm-n=2m-odd-vanish}). Additionally, it reduces to  estimate certain one dimensional oscillatory integrals, which can be treated in a unified manner. 

	\subsection{Finite rank perturbations and $L^p-L^q$ estimates}\label{sec1.2}\
	
	Now we address the same problem for finite rank perturbations \eqref{eq-high-per}. Denote by 
	$$
	P=\sum_{j=1}^NP_j, \,\,\,\,\,\,\, P_j=\langle\cdot\,, \varphi_j\rangle \varphi_j.
	$$
	We assume that
	$\varphi_1,\cdots, \varphi_N$ are $L^2$-normalized and mutually orthogonal, i.e.,
	\begin{equation}\label{eq1.5.1}
		\langle\varphi_i, \varphi_j\rangle=\delta_{ij},\,\,\,\qquad 1\leq i, j\leq N,
	\end{equation}
	where $\delta_{ij}$ denotes the Kronecker delta function.  In addition to the decay and smoothness condition on each $\varphi_j$, we  also require a spectral assumption analogous to that of \eqref{eq1.4}. Instead of the scalar function defined in \eqref{eq1.11.1},  we define the following $N\times N$ matrix
	\begin{equation}\label{eq1.11}
		F_{N\times N}(z)=(f_{ij}(z))_{N\times N},\qquad f_{ij}(z)=\langle((-\Delta)^m-z)^{-1}\varphi_i, \varphi_j\rangle,\,\,\, z\in \mathbb{C}\setminus[0, \infty).
	\end{equation}
	Note that if  $\varphi_i\in L^{2}_{\sigma}$ ( $1\le i\le N$) for some  $\sigma>\frac12$, then it is meaningful to define the boundary value
	\begin{equation}\label{eq1.12}
		F_{N\times N}^{\pm}(\lambda^{2m}):=F_{N\times N}(\lambda^{2m}\pm \i 0),\qquad \lambda>0.
	\end{equation}
	In view of \eqref{eq1.4}, it is natural to make the following assumption:
	For $n\ge1$, there exists some absolute constant ${c}_{0}>0$ such that
	\begin{equation}\label{eq-low-spec}
		\left | \det(I+{F}_{N\times N}^{\pm }({\lambda }^{m}))\right |\ge {c}_{0},\ \ \ \text{for  any} \,\, \lambda>0.
	\end{equation}
	
	Based on the above assumptions, we have
	\begin{theorem}\label{cor-finite-rank}
		Let $n\ge 1$ and let $H_N$ be given by \eqref{eq-high-per}.  Assume that \eqref{eq1.5.1} holds and each $\varphi_i$ ($1\le i\le N$) satisfies the smoothness and decay condition \emph{(\romannumeral1)} in Assumption \ref{assumption-1}, as well as   the spectral assumption  \eqref{eq-low-spec}.
		Then there exists a  positive constant $C=C(N,\varphi_1,\cdots,\varphi_N)$ such that  ${e}^{-\i tH_N}$ has integral kernel satisfying
		\begin{align}\label{equ-point-fin}
			\left |K(t,x,y)\right |\le  C|t|^{-\frac{n}{2m}}(1+|t|^{-\frac{1}{2m}}\left | x-y\right |)^{-\frac{n(m-1)}{2m-1}} ,\ \ t\ne 0,\,\,x, y\in \mathbb{R}^n.
		\end{align}
	\end{theorem}
	The basic strategy in our proof of Theorem \ref{cor-finite-rank} is to simplify the problem to the rank one case by examining the resolvent $R^{\pm}(\lambda^{2m})=(H-(\lambda^{2m} \pm \i 0))^{-1}$, However, unlike Theorem \ref{thm-1.1}, the main challenge here lies in  dealing with the inverse of the  $N\times N$ matrix $I+{F}_{N\times N}^{\pm }({\lambda }^{m})$, particularly when $n\le 2m$. Observe that  this is equivalent to examining the inverse of the operator $I+P{R}_{0}^{\pm}({\lambda }^{2m})P$ on the finite dimensional space  spanned by $\{\varphi_1,\cdots, \varphi_N\}$ near zero energy.  Our approach introduces a novel decomposition, expressing the space as a direct sum: 
	\begin{align*}
		PL^2={\bigoplus }_{j\in J}Q_jL^2, \,\,\,\,\  J=
		\begin{cases}
			\ \{0,\, 1,...,m-\frac{n+1}{2},m-\frac{n}{2}\},   &\mbox{if $n$ is odd},\\
			\ \{0,\, 1,...,m-\frac{n}{2},m-\frac{n}{2}+1\}, &\mbox{if $n$ is even}, 
		\end{cases}  
	\end{align*}
	where $Q_j$ denotes certain  orthogonal projection (see \eqref{equ-5.27}-\eqref{eq-projection-Qj}) which is designed to capture the vanishing conditions of $\varphi_j$, $1\le j\le N$.
	It turns out that the problem is ultimately transformed into considering the Gram matrix associated with the coefficients of ${R}_{0}^{\pm}({-1})(x-y)$ (see Propositions \ref{Thm-I+PRP} and \ref{Thm-I+PRP-even}).
	
	We mention that analogous  difficulties  are encountered in the context of  higher order Schr\"{o}dinger operators $H=(-\Delta)^{m}+V$. Notably, for the case where $m=2$ and $1\le n\le 4$,  the asymptotic behavior of the perturbed resolvent has been  individually established in \cite{EGT,GT,SWY,LSY}. We adopt the idea in \cite{CHHZ}, where the matrix-based approach is introduced and  a unified treatment for odd dimensions is presented.

	Given that the  pointwise bounds in \eqref{eq1.5} exhibit additional  decay with respect to   $|x-y|$ when $m>1$, we are able to further establish the $L^p-L^q$ type estimate for $e^{-\i tH}P_{ac}(H)$, where $p, q$ can be exponents that do not form a dual pair. More precisely, set
	\begin{align}\label{eq3.7}
		\square_{ABCD}=\{\mbox{$(p, q)$}\in\mathbb{R}_+^2;\ \mbox{$({\frac 1p},{\frac 1q})$}\ {\rm lies\
			in\ the\ closed\ quadrilateral \ ABCD}\}.
	\end{align}
	Here $A=(\frac12,\frac12)$, $B=(1,\frac{1}{\tau_m})$, $C=(1,0)$ and $D=(\frac{1}{\tau'_m},0)$, where $\tau_m=\frac{2m-1}{m-1}$, see Figure \ref{fig1} below.
	Moreover, we denote by $H^1$ the Hardy space on $\R^n$ and by BMO the space of functions with bounded mean oscillation on $\R^n$. Let $(p, q)\in \square_{ABCD}$ and denote by
	\begin{align}\label{equ3.6}
		L^p_*-L^q_*=
		\begin{cases}
			L^1-L^{\tau_m,\infty}\, \text{or}\, H^1-L^{\tau_m}&\text{if}\,\, (p,q)=(1, \tau_m),\\[4pt]
			L^{\tau'_m,1}-L^\infty\,  \text{or}\, L^{\tau_m'}-\text{BMO}&\text{if}\,\,(p,q)=(\tau'_m,\infty),\\[4pt]
			L^p-L^q&\text{otherwise}.
		\end{cases}
	\end{align}
	\begin{figure}[h]
		\centering
		\begin{tikzpicture}[scale=5.0]
			\draw [<->,] (0,1.14) node (yaxis) [right] {$\frac{1}{q}$}
			|- (1.24,0) node (xaxis) [right] {$\frac{1}{p}$};


			\draw (0.80,0) coordinate (a_1)node[below]{\color{blue}{\small{$D=(\frac{1}{\tau_m'},0)$}}};
			\draw[dashed](1,0) coordinate (a_3)node[below=3.5mm] [right=0.1mm] {\small{\color{blue}{C=(1,0)}}} -- (1,1) coordinate (a_4);
			\draw[dashed](1,0) coordinate (a_3) -- (0,1) coordinate (a_5)node[left=0.9mm] {$1$};
			\draw[dashed](1,1) coordinate (a_3) -- (0,1) coordinate (a_5);
			\draw (1,0.2) coordinate (a_5)node[right]{\color{blue}{$B=(1,\frac{1}{\tau_m})$}} -- (0.56,0.444) coordinate (a_2);
			
			\fill[blue]  (0.5,0.5)coordinate (a_6) circle (0.2pt) node [above=3.3mm] [right=0.1mm] {$(\frac12, \frac12)$} node[above=0.3mm]{{A}} ;
			\fill[blue]  (0.80,0) circle (0.2pt);
			\fill[blue] (1,0) circle (0.2pt);
			\fill[blue]  (0.56,0.444) circle (0.2pt);
			\fill[blue]  (1,0.2) circle (0.2pt);
			
			\draw[fill=gray!20]  (0.5,0.5)--(0.80,0) -- (1,0) -- (1,0.2) -- cycle;
			\draw [densely dotted] (0.5,0.5) -- (1,0.2);
			\draw [densely dotted] (0.5,0.5) -- (0.8,0);
			\end{tikzpicture}
			\caption{The  $L^p-L^q$ estimates}\label{fig1}
		\end{figure}
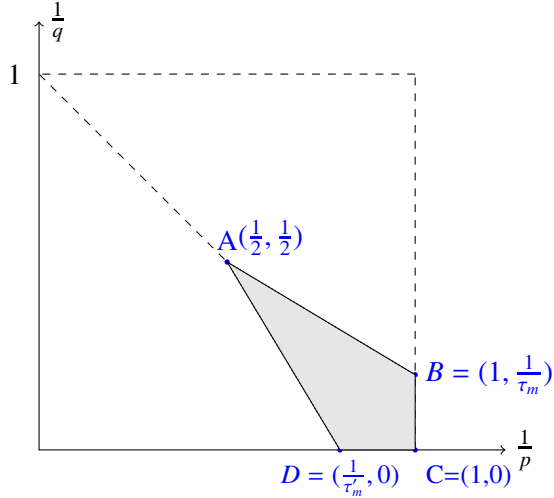
		
		\begin{corollary}\label{cor-Lp}
			Let the assumption of Theorem  \ref{cor-finite-rank} be satisfied and  $(p, q)\in \square_{ABCD}$, then
			\begin{equation}\label{equ-Lp-Lq}
				\|e^{-\i tH_N}P_{ac}(H)\|_{L^p_*-L^q_*}\lesssim|t|^{-\frac{n}{2m}(\frac{1}{p}-\frac{1}{q})},\qquad\,\,t\ne 0,
			\end{equation}
			where $(p, q)\in \square_{ABCD}$ and $L^p_*-L^q_*$ is given by \eqref{equ3.6}.
		\end{corollary}
		
		
		\subsection{Notation}\
		
		Throughout the paper, $C(\cdots)$ and $C_{\{\cdot\}}$ will denote positive constants depending on what are enclosed in the brackets and they may differ on different occasions. And we use $a\pm:= a\pm \epsilon$ for some small but fixed $\epsilon >0$. For $a,b \in \mathbb{R}^+$, $a \lesssim b$ (or $a \gtrsim b$) means that there exists some constant $C$ such that $a\le Cb$ (or $a \ge Cb$). For $l>0$, we denote by $[l]$ the greatest integer at most $l$.
		The rest of the paper is organized as follows:
		In section \ref{sec2}, we present the spectral representation of the propagator $e^{-\i tH}$, expansions of the free resolvent, as well as the oscillatory integrals which will be frequently used later.
		Section \ref{sec3} and \ref{sec4} are devoted to the proof of  Theorem \ref{thm-1.1}. In Section  \ref{sec5} and \ref{sec6}, we give the proof of Theorem \ref{cor-finite-rank} and Corollary \ref{cor-Lp} respectively.
		
		\section{Preliminaries}\label{sec2}
		\subsection{Spectral representation}\label{sec2.1}\
		
		Inspired by the dispersive estimates for Schr\"{o}dinger operators $-\Delta+V$ (see e.g. \cite{RS,Sch}), we begin with the  spectral representation of the propagator as given  by Stone's formula:
		\begin{align}\label{eq3.1}
			\left  \langle e^{-\i tH}P_{ac}f, g\right\rangle=\frac{1}{2\pi \i}\int_0^{\infty}{e^{-\i t\lambda}\left\langle [R^{+}_{\alpha}(\lambda)-R^{-}_{\alpha}(\lambda)]f,\, g\right\rangle\,\d\lambda},
		\end{align}
		Here, $f, g$ are Schwartz functions, $P_{ac}$ denotes the projection onto the absolutely continuous spectrum of $H$, and ${R}^{\pm}_{\alpha}({\lambda })={(H-\lambda\mp\i0)}^{-1}$.
		
		Compared  with the case of Schr\"{o}dinger operators $(-\Delta)^m+V$, the main difference here is the following Aronszajn-Krein formula (see \cite{Simon}): For
		$z\in \mathbb{C}\backslash[0,\infty )$, one has
		\begin{equation}\label{A-K-for}
			R_{\alpha}(z)=R_{0}(z)-\frac{\alpha}{1+\alpha F(z)}R_{0}(z)\varphi\big\langle R_{0}(z)\cdot, \varphi\big\rangle,
		\end{equation}
		where $R_0(z)=((-\Delta)^m-z)^{-1}$ is the free resolvent and $F(z)$ is given by \eqref{eq1.1.1-F(z)}.
		Plugging the Aronszajn-Krein formula \eqref{A-K-for} into \eqref{eq3.1} and  noting that
		\begin{equation*}
			{e}^{-{\rm i}t{H}_{0}}=\frac{m}{\pi {\rm i}}\int_{0}^{\infty }{e}^{-{\rm i}t{\lambda }^{2m}}( {R}_{0}^{+}({\lambda }^{2m})-{R}_{0}^{-}({\lambda }^{2m}) ){\lambda }^{2m-1}\rm d \lambda,
		\end{equation*}
		then the change of variables $\lambda\rightarrow\lambda^{2m}$ in \eqref{eq3.1} leads to
		\begin{equation}\label{equ-rank1-st}
			\begin{aligned}
				{e}^{-{\rm i}tH}-{e}^{-{\rm i}t{H}_{0}}
				&=-\frac{m}{\pi {\rm i}}\int_{0}^{\infty }{e}^{-{\rm i}t{\lambda }^{2m}}\frac{\alpha }{1+\alpha{F}^{+}({\lambda }^{2m})}{R}_{0}^{+}({\lambda }^{2m})\varphi \left\langle {R}_{0}^{+}({\lambda }^{2m})\cdotp ,\varphi\right\rangle {\lambda }^{2m-1}\rm d \lambda \\
				&+\frac{m}{\pi {\rm i}}\int_{0}^{\infty }{e}^{-{\rm i}t{\lambda }^{2m}}\frac{\alpha }{1+\alpha{F}^{-}({\lambda }^{2m})}{R}_{0}^{-}({\lambda }^{2m})\varphi\left \langle {R}_{0}^{-}({\lambda }^{2m})\cdotp ,\varphi\right \rangle {\lambda }^{2m-1}\rm d \lambda\\
				&:={K}_{n}^{+}-{K}_{n}^{-}.
			\end{aligned}
		\end{equation}
		By \eqref{eq-high-free-p}, it suffices to deal with the kernel of ${K}_{n}^{+}-{K}_{n}^{-}$.
		To this end, we first fix a small $0<\lambda_0<1$, and choose a smooth cut-off function $\chi (\lambda )$ such that
		\begin{equation}\label{equ-chi}
			\chi(\lambda)=1\,\,\,\,\,\text{if}\,\,\,\lambda<\frac{\lambda_0}{2}; \qquad \chi(\lambda)=0\,\,\,\,\,\text{if}\,\,\,\lambda>\lambda_0.
		\end{equation}
		We then split ${K}_{n}^{\pm}$ into the high energy part
		\begin{equation}\label{equ-2-high}
			{K}_{n}^{\pm ,h}=-\frac{m\alpha }{\pi {\rm i}}\int_{0}^{+\infty }{e}^{-{\rm i}t{\lambda }^{2m}}\frac{1-\chi (\lambda )}{1+\alpha {F}^{\pm }({\lambda }^{2m})}{R}_{0}^{\pm }({\lambda }^{2m})\varphi \left\langle {R}_{0}^{\pm }({\lambda }^{2m})\cdotp ,\varphi\right\rangle  {\lambda }^{2m-1}\rm d \lambda,
		\end{equation}
		and the low energy part
		\begin{equation}\label{equ-2-low}
			{K}_{n}^{\pm ,l}=-\frac{m\alpha }{\pi {\rm i}}\int_{0}^{+\infty }{e}^{-{\rm i}t{\lambda }^{2m}}\frac{\chi (\lambda )}{1+\alpha {F}^{\pm }({\lambda }^{2m})}{R}_{0}^{\pm }({\lambda }^{2m})\varphi\left \langle {R}_{0}^{\pm }({\lambda }^{2m})\cdotp ,\varphi\right\rangle  {\lambda }^{2m-1}\rm d \lambda.
		\end{equation}
		We shall establish pointwise estimates for the kernel of ${K}_{n}^{\pm ,h}$ and ${K}_{n}^{+, l}-{K}_{n}^{-, l}$ in Section \ref{sec3} and \ref{sec4} respectively.
		In particular, Theorem \ref{thm-1.1} follows from Theorem \ref{thm3.1} and \ref{thm3.2} immediately.
		
		\subsection{Resolvent expansion}\label{sec2.2}\
		
		We first collect some basic results about the free resolvent $R_0(z)$, which is used to establish estimates of  $(1+\alpha {F}^{\pm }({\lambda }^{2m}))^{-1}$ that appeared both in \eqref{equ-2-high} and \eqref{equ-2-low}.
		
		When $\lambda>0$, the well known limiting absorption principle implies that the weak* limits $R_0^{\pm}(\lambda^{2m}):= R_0(\lambda^{2m}\pm \i 0)=w*-\lim_{\epsilon\downarrow0}R_0(\lambda^{2m}\pm\i\epsilon)$ exist as bounded operators from weighted space $L_{\sigma}^2$ to $L_{-\sigma}^2$ with $\sigma>\frac12$.
		Further, we have the splitting identity  (see \cite{FSWY,H-Y-Z})
		\begin{equation}\label{equ2.1.1.1}
			{R}_{0}^{\pm }(\lambda^{2m})=\frac{1}{m{\lambda }^{2m}}\displaystyle\sum_{k\in {I}^{\pm }}^{}{{\lambda }_{k}}^{2}\mathfrak{R}_0^{\pm}({{\lambda }_{k}}^{2})=\frac{1}{m{\lambda }^{2m-2}}\displaystyle\sum_{k\in {I}^{\pm}}
			\omega_{k}\mathfrak{R}_0^{\pm}(\omega_{k}{\lambda }^{2}),
		\end{equation}
		where $\mathfrak{R}_0(z)=(-\Delta-z)^{-1}$, ${\lambda }_{k}=\lambda{e}^{\frac{{\rm i}k\pi}{m}}$, $\omega_{k}=\mbox{exp}({\rm i}2\pi k/m)$ are the $\rm {m}^{th}$ roots of unity, and
		\begin{equation*}
			{I}^{+}=\{0,1,2,\cdots,m-1\}, \qquad  {I}^{-}=\{1,2,3,\cdots,m\}.
		\end{equation*}
		From \eqref{equ2.1.1.1}, we see  that the kernel of ${R}_{0}^{\pm }(\lambda^{2m})$ can be expressed as linear combinations of the resolvent kernel of the Laplacian.
		Let $R_0(z)(x-y)$ (resp. $\mathfrak{R}_0(z)(x-y)$) be the integral kernel of $R_0(z)$ (resp. $\mathfrak{R}_0(z)$),
		recall that
		\begin{equation}\label{eq2.1}
			\mathfrak{R}_0^{\pm }(\lambda_k^2)(x-y)=\frac{\pm \rm i}{2\lambda_k }{e}^{\pm {\rm i}\lambda_k |x-y|},\ \ \ n = 1,
		\end{equation}
		and
		\begin{equation}\label{eq2.2}
			\mathfrak{R}_0^{\pm }(\lambda_k^2)(x-y)=\frac{\pm \rm i}{4}(\frac{\lambda_k }{2\pi \left | x-y\right |})^{\frac{n}{2}-1}{H}_{\frac{n}{2}-1}^{\pm }(\lambda_k \left | x-y\right |),\ \ \ n\ge 2,
		\end{equation}
		where ${H}_{\frac{n}{2}-1}^{\pm }(z)={J}_{\frac{n}{2}-1}^{\pm }(z)\pm {\rm i}{Y}_{\frac{n}{2}-1}^{\pm }(z)$ are the Hankel functions of order $\frac{n}{2}-1$.
		In particular, when $n \ge 3$ is odd, we can rewrite \eqref{eq2.2} as
		\begin{equation}\label{eq.2.3}
			\mathfrak{R}_0^{\pm }(\lambda_k^2)(x-y)={C}_{n}\frac{{e}^{\pm {\rm i}\lambda_k \left | x-y\right |}}{{\left | x-y\right |}^{n-2}}\displaystyle\sum_{l=0}^{\frac{n-3}{2}}\frac{\left ( n-3-l\right )!}{l!\left ( \frac{n-3}{2}-l\right )!}{\left ( \mp 2 {\rm i}\lambda_k \left | x-y\right |\right )}^{l}.\tag{\ref{eq2.2}'}
		\end{equation}
		
		We note that the asymptotic property of Hankel function ${H}_{\frac{n}{2}-1}^{\pm }(z)$ at infinity shows that ${H}_{\frac{n}{2}-1}^{\pm }(z)$ can be expressed as
		\begin{equation}\label{eq-asmptotic-hankel}
			{H}_{\frac{n}{2}-1}^{\pm }(z)=e^{\pm\i z}\psi_>^\pm (z),\qquad z\gtrsim 1,
		\end{equation}
		where the functions  $\psi_>^\pm(z)$ satisfy for $l\in \N_0$,
		\begin{equation}\label{eq-est-hankel}
			\left| \frac{d^l}{dz^l}\psi_>^\pm (z)\right| \lesssim_l \langle z\rangle^{-\frac 12 -l},\quad \text{for}\   z\gtrsim 1.	
		\end{equation}
		
		
		Next, we present several lemmas that address different types of expansions related to ${R}_{0}^{\pm}({\lambda }^{2m})(x-y)$, ${R}_{0}(-1)(x-y)$ and ${R}_{0}^{+}({\lambda }^{2m})(x-y)-{R}_{0}^{-}({\lambda }^{2m})(x-y)$, as detailed in Lemma \ref{lm2.0}--\ref{lm-expansion-high}. The proofs of these lemmas are given in Appendix \ref{app-01}.
		
		\begin{lemma}\label{lm2.0}
			Let $\lambda>0$ and  $1\le n\le 2m$.
			
			\noindent (\romannumeral1)\, If $n$ is odd, then
			\begin{equation}\label{eq.2-low-odd}
				{R}_{0}^{\pm }({\lambda }^{2m})(x-y)=\displaystyle\sum_{j=0}^{m-\frac{n+1}{2}}{a}_{j}^{\pm} {\lambda }^{n-2m+2j} {\left | x-y\right |}^{2j} +{b}_{0}{\left | x-y\right |}^{2m-n}+\lambda^{n-2m}{r}_{2m-n+1}^{\pm}(\lambda|x-y|),
			\end{equation}
			where ${a}_{j}^{\pm}\in \mathbb{C}\setminus\mathbb{R},~ b_0\in \R$. Moreover, ${r}_{2m-n+1}^{\pm}(z)\in C^\infty(\R)$ satisfies 
			\begin{equation}\label{eq3.5.0-rem}
				\left|\frac{{d}^{l}}{d{z}^{l}}{r}_{2m-n+1}^{\pm }(z) \right|\lesssim 
				\begin{cases}
					{\left| z\right |}^{2m-n+1-l}, &\quad \text{for} \ 0\le l\le 2m-\frac{n-1}{2},\\  
					{\left| z\right |}^{2m-n+1-l}+|z|^{-\frac{n-1}{2}}, &\quad \text{for} \ l> 2m-\frac{n-1}{2}.
				\end{cases}
			\end{equation}
			
			\noindent (\romannumeral2)\, If $n$ is even, then
			\begin{align}\label{eq.2-low-even}
				&{R}_{0}^{\pm }(\lambda^{2m})(x-y)\nonumber\\ 
				=~&\displaystyle\sum_{j=0}^{m-\frac{n}{2}}{a}_{j}^{\pm }{\lambda }^{n-2m+2j}{\left | x-y\right |}^{2j}+{b}_{0}{\left | x-y\right |}^{2m-n}\log{\left ( \lambda \left | x-y\right |\right )}+\lambda^{n-2m}{r}_{2m-n+1}^{\pm}(\lambda|x-y|), 
			\end{align}
			where ${a}_{j}^{\pm}\in \mathbb{C}\setminus\mathbb{R},~ b_0\in \R$. Moreover, the remainder term ${r}_{2m-n+1}^{\pm }(z)\in C^\infty(\R\setminus\{0\})$ and it satisfies \eqref{eq3.5.0-rem}.
		\end{lemma}
		
		\begin{lemma}\label{lm-finiterank-5.5}
			Let $1\le n\le 2m$. 
			If $n$ is odd, one has
			\begin{equation}
				R_0(-1)(x-y)=\sum\limits_{|\alpha+\beta|\le 2m-n-1} A_{\alpha,\beta}x^\alpha y^\beta+b_0{|x-y|}^{2m-n}+\tilde{r}_{2m-n+1}(|x-y|).
			\end{equation}
			If $n$ is even, one has
			\begin{equation}\label{eq5.47}
				\begin{split}
					&R_0(-1)(x-y)\\
					=&\sum\limits_{|\alpha+\beta|\le 2m-n-2}A_{\alpha,\beta}x^\alpha y^\beta+{a}_{m-\frac{n}{2}}{|x-y|}^{2m-n}+b_0\log(|x-y|){|x-y|}^{2m-n}
					+\tilde{r}_{2m-n+1}(|x-y|).
				\end{split}
			\end{equation}
			In the above expressions,  $A_{\alpha,\beta}$ is defined as 
			\begin{equation}\label{eq-def-A}
				A_{\alpha,\beta}= 
				\begin{cases}
					a_{\frac{|\alpha|+|\beta|}{2}}^+ e^{\frac{\i\pi(n-2m+|\alpha|+|\beta|)}{m}}C_{\alpha,\beta}(-1)^{|\beta|}=  a_{\frac{|\alpha|+|\beta|}{2}}^- e^{\frac{-\i\pi(n-2m+|\alpha|+|\beta|)}{m}} C_{\alpha,\beta}(-1)^{|\beta|},\ &\text{if}\, |\alpha|+|\beta|\, \text{is even}, \\ 
					0,\quad  &\text{else},
				\end{cases}
			\end{equation}
			where the constant  $C_{\alpha,\beta}$ is from the expansion
			\begin{equation}\label{eq-expan-xy}
				|x-y|^{2j}= \sum_{|\alpha|+|\beta|=2j}C_{\alpha,\beta}(-1)^{|\beta|}x^{\alpha}y^{\beta}. 
			\end{equation}
			
			Furthermore, in (\romannumeral1) and (\romannumeral2), $A_{\alpha,\beta}$ has the form
			\begin{align}\label{eq5.48}
				A_{\alpha,\beta}=\frac{(-1)^{|\beta|}\i^{|\alpha|+|\beta|}}{(2\pi)^n\alpha!\beta!}\int_{\R^n}\frac{\xi^{\alpha+\beta}}{1+|\xi|^{2m}}\,\d\xi.
			\end{align}
		\end{lemma}
		
		\begin{lemma}\label{lem-expansion-with-theta}
			Let $\lambda>0$ and $1\le n\le 2m-1$. If $\theta=0,\cdots,2m-n$ when $n$ is odd and $\theta=0,\cdots,2m-n-1$ when $n$ is even, one has
			\begin{equation}\label{eq.2-low-odd-1}
				{R}_{0}^{\pm }({\lambda }^{2m})(x-y)=\displaystyle\sum_{j=0}^{[\frac{\theta-1}{2}]}{a}_{j}^{\pm} {\lambda }^{n-2m+2j} {\left | x-y\right |}^{2j}+\lambda^{n-2m}{r}_{\theta}^{\pm}(\lambda|x-y|),
			\end{equation}	
			where the remainders  ${r}_{\theta}^{\pm}(z)\in C^\infty(\mathbb{R})$  satisfy that
			\begin{equation}\label{eq-est-rthe-1}
				\left|\frac{d^l}{dz^l}{r}_{\theta}^{\pm}(z)\right|
				\lesssim_l
				\begin{cases}
					|z|^{\theta-l}, &\quad l\le \theta+\frac{n-1}{2},\\
					|z|^{\theta-l}+|z|^{-\frac{n-1}{2}}, &\quad l> \theta+\frac{n-1}{2}.
				\end{cases}
			\end{equation}
		\end{lemma}
		
		For the term ${R}_{0}^{+}({\lambda }^{2m})(x-y)-{R}_{0}^{-}({\lambda }^{2m})(x-y)$, it is showed in  \cite[Lemma 2.4]{EG23} that the following holds:
		\begin{lemma}\label{lem-expasion-R+-R-1}
			Let $n\ge 1$, we have
			\begin{equation}\label{eq-R+-R-}
				{R}_{0}^{+}(\lambda^{2m})(x-y)-{R}_{0}^{-}(\lambda^{2m})(x-y)=\lambda^{n-2m}e^{\i \lambda |x-y|}{U}_{0}^{+}(\lambda|x-y|)-\lambda^{n-2m}e^{- \i \lambda |x-y|}{U}_{0}^{- }(\lambda|x-y|),
			\end{equation}
			where ${U}_{0}^{\pm }\in C^\infty(\R)$ and they satisfy the following
			\begin{equation}\label{eq-est-psi0}
				\left|\frac{{d}^{l}}{dz^l}{U}_{0}^{\pm }(z) \right|\lesssim_l |z|^{-l},  \quad \text{for}\  \  z\in \R, \ \  l\in \N_0.
			\end{equation}	
		\end{lemma}
		
		We also recall the following representation of the resolvent kernel ${R}_{0}^{\pm}({\lambda }^{2m})(x-y)$, which will  be needed during the proof (see, e.g., in \cite[Lemma 3.2, Lemma 6.2]{EG22}).
		
		\begin{lemma}\label{lm-expansion-high}
			(\romannumeral 1) Let $1\le n\le 4m-1$, we have
			\begin{equation}\label{eq-expansion-high-1}
				{R}_{0}^{\pm}({\lambda }^{2m})(x-y)=\frac{\lambda ^{\frac{n+1}{2}-2m}}{|x-y|^{\frac{n-1}{2}}}e^{\pm \i \lambda |x-y|}U_{1}^{\pm}(\lambda |x-y|),
			\end{equation}
			where $U_{1}^{\pm}(z)$ satisfy
			\begin{equation}\label{eq-expansion-high-4}
				\left | \frac{d^l}{dz^l}U_{1}^{\pm}(z)\right | \lesssim_l\  |z|^{-l},\qquad l\in \mathbb{N}_0.
			\end{equation}
			
			(\romannumeral2)  Let $n>2m$, we have
			\begin{equation} \label{eq-decom-r-nl2m}
				{R}_{0}^{\pm}({\lambda }^{2m})(x-y)=\frac{e^{\pm \i \lambda |x-y|}}{|x-y|^{n-2m}}W_{0}^{\pm}(\lambda |x-y|)+\frac{\lambda^{\frac{n+1}{2}-2m}}{|x-y|^{\frac{n-1}{2}}}e^{\pm \i \lambda |x-y|}W_{1}^{\pm}(\lambda |x-y|),   
			\end{equation}
			where $\text{\supp}W_{0}^{\pm}(\cdot) \subset[0, 1]$, $\text{\supp}W_{1}^{\pm}(\cdot) \subset[\frac12, \infty)$. Moreover,  
			\begin{equation}\label{eq-expansion-high-6}
				\left | \frac{d^l}{dz^l}W_{0}^{\pm}(z)\right | \lesssim_l\  |z|^{-l},\,\,\,\,\,\,\,\mbox{and}\,\,\,\,\,\,\,\left | \frac{d^l}{dz^l}W_{1}^{\pm}(z)\right | \lesssim_l\  |z|^{-l}, \qquad l\in \mathbb{N}_0.
			\end{equation}
		\end{lemma}
		
		\subsection{Estimates for some (oscillatory) integrals}\
		
		The purpose of this subsection is to provide the estimates for some (oscillatory) integrals that will be frequently used in what follows. 
		More precisely, in order to obtain the point-wise estimate for the kernel of ${K}_{n}^{\pm ,h}$ and ${K}_{n}^{+, l}-{K}_{n}^{-, l}$ in \eqref{equ-2-high}-\eqref{equ-2-low}, we will reduce the analysis to the estimates of the following one dimensional oscillatory integrals over $\Omega=(0,{\lambda }_{0})$ or  $\Omega= ({\lambda }_{0},+\infty)$, with $\lambda_{0}>0$ being some fixed constant:
		\begin{equation}\label{eq-osci-1213}
			I(t,x)=\int_{\Omega }^{}e^{{\rm i}t{\lambda }^{2m}+{\rm i}\lambda x}\psi (\lambda,x ) \d \lambda .  
		\end{equation}
		We present two  such lemmas, which are not new and can be found in the existing literature, for example, in \cite{CHZ,CHHZ,HHZ}.
		For the sake of completeness, we provide a detailed proof of Lemma \ref{lm2.4} in Appendix \ref{app-4} ,  and the proof of Lemma \ref{lm2.3} can be derived in a similar manner.
		
		\begin{lemma}\label{lm2.4}
			Let  $\Omega=(0,{\lambda }_{0})$ in \eqref{eq-osci-1213}, where  $0 <{\lambda }_{0}< 1$ is a fixed constant. Assume that  $\psi (\cdot,x )\in {C}^{k}(\Omega)$, with $x\in\R^n$ being a parameter, and it satisfies for $\lambda\in \Omega$
			\begin{equation}\label{eq-psi-deri}
				\sup\limits_x\left | \partial_\lambda^l \psi(\lambda,x)\right |\lesssim{\lambda}^{b-l},\,\,\, 0\le l\le k.
			\end{equation}
			Define ${\mu }_{b}$ as
			\begin{equation}\label{eq-mu-b}
				{\mu }_{b}:=\frac{m-1-b}{2m-1}.
			\end{equation}
			Under these conditions, we have
			
			($i$) If $b\in [-\frac{1}{2},2km-1)$ and ${t}^{-\frac{1}{2m}}\left | x\right |> 1$,
			\begin{align}\label{eq-osc-g-1}
				\left | I(t,x)\right |\lesssim {\left | t\right |}^{-\frac{1+b}{2m}}({t}^{-\frac{1}{2m}}\left | x\right |)^{-{\mu }_{b}}.
			\end{align}
			
			($ii$) If $b\in (-1,2km-1)$ and ${t}^{-\frac{1}{2m}}|x|\le 1$,
			\begin{align}\label{eq-osc-l-1}
				\left | I(t,x)\right |\lesssim (1+{t}^{\frac{1}{2m}})^{-(1+b)}.
			\end{align}
		\end{lemma}
		
		\begin{lemma}\label{lm2.3}
			Let  $\Omega= ({\lambda }_{0},+\infty )$ in \eqref{eq-osci-1213}. Assume that  $\psi (\cdot,x )\in {C}^{k}(\Omega)$, with $x\in\R^n$ being a parameter, and it satisfies the estimate \eqref{eq-psi-deri}. Let ${\mu }_{b}$ be given by \eqref{eq-mu-b}. Then
			
			($i$) If $b\in [-\frac{1}{2},2km-1)$ and ${t}^{-\frac{1}{2m}}\left | x\right |> 1$,
			\begin{align}\label{eq-osc-h}
				\left | I(t,x)\right |\lesssim {\left | t\right |}^{-\frac{1+b}{2m}}({t}^{-\frac{1}{2m}}\left | x\right |)^{-{\mu }_{b}}.
			\end{align}
			
			($ii$) If $b\in [-\frac{1}{2},2km-1)$ and ${t}^{-\frac{1}{2m}}\left | x\right |\le 1$,
			\begin{align}\label{eq-osc-l}
				\left | I(t,x)\right |\lesssim {t}^{-\frac{1+b}{2m}}.
			\end{align}
			
		\end{lemma}

		In addition to the aforementioned lemmas, We also need the following estimate, which can be seen in \cite[Lemma 3.8]{GV} or \cite[Lemma 6.3]{EG10}.
		\begin{lemma}\label{lm2.8}
			Let $n\ge 1$. Then there is some absolute constant $C>0$ such that
			\begin{equation}\label{eq2.20}
				\int_{\mathbb{R}^n}|x-y|^{-k}\langle y\rangle^{-l}\,dy\leq C\langle x\rangle^{-\min\{k,\, k+l-n\}},
			\end{equation}
			provided  $l\ge 0$, $0\le k<n$ and $k+l>n$. As a consequence, we also have
			\begin{equation}\label{eq2.21}
				\int_{\mathbb{R}^n}\int_{\mathbb{R}^n}|x-y|^{-\sigma}\langle x-\tau\rangle^{-\beta_1}\langle y\rangle^{-\beta_2}\,dxdy\leq C\langle\tau\rangle^{-\sigma},
			\end{equation}
			provided $\beta_1>n$,  $\beta_2>n$ and $0\le \sigma<n$.
		\end{lemma}
		
		\section{Proof of Theorem \ref{thm-1.1}: High energy estimate}\label{sec3}
		
		The main result in this section is the following theorem.
		
		\begin{theorem}\label{thm3.1}
			Under the assumptions of Theorem \ref{thm-1.1},  the high energy part ${K}_{n}^{\pm ,h}$, defined by \eqref{equ-2-high}, has integral kernel ${K}_{n}^{\pm ,h}(t,x,y)$ satisfying
			\begin{equation}\label{est-for-high-fi-thm3.1}
				\left | {K}_{n}^{\pm ,h}(t,x,y)\right |\le C {t}^{-\frac{n}{2m}}(1+{t}^{-\frac{1}{2m}}\left | x-y\right |)^{-\frac{n(m-1)}{2m-1}},\ \ \ t>0,\,x,y\in\mathbb{R}^n,
			\end{equation}
			for some positive constant $C>0$.
		\end{theorem}
		
		In order to prove Theorem \ref{thm3.1}, we first prove some results for ${F}^{\pm }({\lambda }^{2m})$ (see \eqref{eq1.11.1}) and
		\begin{equation*}
			{F}_{\pm }^{\alpha ,h}(\lambda^{2m}):=\frac{1-\chi (\lambda )}{1+\alpha {F}^{\pm }({\lambda }^{2m})}.
		\end{equation*}
		
		\begin{lemma}\label{lm2.1}
			Let $n\ge 1$ and $\varphi$ satisfy Assumption  \ref{assumption-1}. For $0\le l\le [\frac{n}{2m}]+1$ and ${\lambda}_{0} >0$, we have
			\begin{equation}\label{equ2.1.1}
				\left|\frac{{d}^{l}}{d{\lambda }^{l}}{F}^{\pm }({\lambda }^{2m})\right|\lesssim{\lambda }^{1-2m}, \ \  \lambda > \frac{{\lambda }_{0}}{2}.
			\end{equation}
			Furthermore,
			\begin{equation}\label{equ2.2.1.1}
				\left | \frac{{d}^{l}}{d{\lambda }^{l}}{F}_{\pm }^{\alpha ,h}({\lambda }^{2m})\right |\lesssim 
				\begin{cases}
					C, &\quad \ l=0,\\  
					\lambda^{1-2m}, &\quad \ l\ge 1.
				\end{cases}
			\end{equation}
		\end{lemma}
		
		\begin{proof}
			For $l\ge 0$, the resolvent $\mathfrak{R}_0^{\pm}({\lambda }_k^{2})$ of the Laplacian  satisfies the following limiting absorption principle (see e.g. in \cite[p. 59]{KK}):
			\begin{equation*}
				{ \left \|\frac{{d}^{l}}{d{\lambda }^{l}}\mathfrak{R}_0^{\pm }({\lambda}_k^{2}) \right \|}_{{L}_{\sigma }^{2}-{L}_{-\sigma }^{2}}\le C(l,\,{\lambda }_{0}){\lambda }^{-1},\ \lambda > \frac{{\lambda }_{0}}{2},
			\end{equation*}
			provided $\sigma > l+\frac{1}{2}$. This, together with the splitting identity \eqref{equ2.1.1.1} as well as the decay assumption on $\varphi$ in Assumption  \ref{assumption-1}, yields \eqref{equ2.1.1}.
			
			To prove \eqref{equ2.2.1.1}, we first note that by the spectral assumption \eqref{eq1.4}, it follows that
			$$|{F}_{\pm }^{\alpha ,h}({\lambda }^{2m}) |\le 1/{{c}_{0}}.$$
			This prove \eqref{equ2.2.1.1} when $l=0$. Second, for any $l=1,2,\cdots$, we have
			\begin{equation}\label{equ2.2.1.2-leb}
				\frac{{d}^{l}}{d{\lambda }^{l}}\left(\frac{1}{1+\alpha {F}^{\pm }({\lambda }^{2m})}\right)=\displaystyle\sum_{w=1}^{l}\displaystyle\sum_{{\mu }_1+{\mu }_2+...+{\mu }_w=l}^{}{C}_{{\mu }_1{\mu }_2\cdots{\mu }_w}\frac{{\alpha }^{w}\textstyle\prod_{s=1}^{w}\frac{{d}^{{\mu }_{s}}}{d{\lambda }^{{\mu }_{s}}}{F}^{\pm }({\lambda }^{2m})}{(1+\alpha {F}^{\pm }({\lambda }^{2m}))^{w+1}},
			\end{equation}
			where ${\mu }_{1}, {\mu }_{2},\cdots, {\mu }_{w}\ge 1$ and ${C}_{{\mu }_{1}{\mu }_{2}\cdots{\mu }_{w}}$ are absolute constants. Then by \eqref{equ2.1.1} and the spectral assumption \eqref{eq1.4},
			$$\left| \frac{{d}^l}{d{\lambda }^l}\left(\frac{1}{1+\alpha {F}^{\pm }({\lambda }^{2m})}\right) \right|\lesssim{\lambda }^{1-2m},\quad \, \, \lambda > \frac{{\lambda }_{0}}{2}.$$
			This, together with the fact that $\left| \frac{{d}^l}{d{\lambda }^l}\left(1-\chi(\lambda)\right) \right|\lesssim_N {\lambda }^{-N}$($N\in \mathbb{N}_0$) when $\lambda > \frac{{\lambda }_{0}}{2}$, yields \eqref{equ2.2.1.1} when $l>0$. Therefore, \eqref{equ2.2.1.1} holds and  the proof is complete.
		\end{proof}
		
		Now we are in the position to prove Theorem  \ref{thm3.1}. Since the arguments proceed  differently between $n\le 4m-1$ and $n\ge 4m$,  we present the proof in subsection \ref{sec3.2} and \ref{sec3.3} respectively.
		
		\subsection{The proof of Theorem \ref{thm3.1} for $1\le n \le 4m-1$}\label{sec3.2}\
		
		By \eqref{equ-2-high} and \eqref{eq-expansion-high-1}, we have
		\begin{align}\label{equ3.9.1}
			{K}_{n}^{\pm ,h}(t,x,y)&=\int_{0}^{+ \infty }\int_{{\mathbb{R}}^{2n}}^{}{e}^{-{\rm i}t{\lambda }^{2m}\pm {\rm i}\lambda \left ( \left | x-{x}_{1}\right |+\left | {x}_{2}-y\right |\right )}{F}_{\pm }^{\alpha ,h}({\lambda }^{2m}){\lambda }^{n-2m}U_1^{\pm}(\lambda |x-x_1|)U_1^{\pm}(\lambda |x_2-y|)\nonumber\\
			&\times \frac{\varphi (x_1)}{{|x-x_1|}^{\frac{n-1}{2}}}\frac{\varphi (x_2)}{{|x_2-y|}^{\frac{n-1}{2}}}{\rm d} \lambda {\rm d}{x}_{1}{\rm d}{x}_{2}.
		\end{align}
		\begin{remark}\label{rmk-symm}
			Given the symmetrical properties of  $x$ and $y$ as presented in \eqref{equ3.9.1}, we may, without loss of generality, assume that $|x|\ge |y|$ for the rest of the proof.
		\end{remark}
		To proceed, for each fixed $t,x,y$, we define
		\begin{equation}\label{eq-chi_0}
			\Omega_1=\left \{ (x_1,x_2):{t}^{-\frac{1}{2m}}\big(| x-{x}_{1}|+|{x}_{2}-y|\big)\le 1 \right \},
		\end{equation}
		and
		\begin{equation}\label{eq-chi_0-1119} \Omega_2=\left \{ (x_1,x_2):{t}^{-\frac{1}{2m}}\big(|x-{x}_{1}|+|{x}_{2}-y|\big) > 1 \right \}.
		\end{equation}
		Thus, it follows that $${K}_{n}^{\pm ,h}(t,x,y)={K}_{n,1}^{\pm ,h}(t,x,y)+{K}_{n,2}^{\pm ,h}(t,x,y),$$
		where
		\begin{align}\label{equ3.9-high}
			{K}_{n,j}^{\pm ,h}(t,x,y)=\int_{\Omega_s}^{}{I}_{n}^{\pm ,h}(t,|x-x_1|,|x_2-y|)\frac{\varphi (x_1)}{{|x-x_1|}^{\frac{n-1}{2}}}\frac{\varphi (x_2)}{{|x_2-y|}^{\frac{n-1}{2}}}{\rm d}{x}_{1}{\rm d}{x}_{2}, \quad j=1,2,
		\end{align}
		and
		\begin{equation*}
			{I}_{n}^{\pm ,h}(t,|x-x_1|,|x_2-y|):=\int_{0}^{+ \infty }{e}^{-{\rm i}t{\lambda }^{2m}\pm {\rm i}\lambda \left (| x-{x}_{1}|+|{x}_{2}-y|\right )}{F}_{\pm }^{\alpha ,h}({\lambda }^{2m}){\lambda }^{n-2m}
			U_1^{\pm}(\lambda |x-x_1|)U_1^{\pm}(\lambda |x_2-y|)\d \lambda.
		\end{equation*}
		
		Since $n\le 4m-1$, one has $n-2m\leq \frac{n-1}{2}$, 
		then we derive from \eqref{equ2.2.1.1} and \eqref{eq-expansion-high-4} that for $0\le l\le \left[\frac{n}{2m}\right]+1$,
		\begin{equation}\label{equ3.12-psi}
			\left | \partial_\lambda^l\left({F}_{\pm }^{\alpha ,h}({\lambda }^{2m}){\lambda }^{n-2m}U_1^{\pm}(\lambda |x-x_1|)U_1^{\pm}(\lambda |x_2-y|)\right)\right |\lesssim \lambda ^{\frac{n-1}{2}-l},\ \ \lambda >\frac{\lambda_0}{2}.
		\end{equation}
		Given the symmetric roles of $x-x_1, y-x_2$ in 
		\eqref{equ3.9.1}, we can further assume that
		\begin{align}\label{eq-sym-xy}
			|x|\ge |y|\,\quad\mbox{and}\,\quad |x-{x}_{1}|\ge |{x}_{2}-y|.
		\end{align}
		
		\noindent{\bf Step 1:  Estimates for ${K}_{n,1}^{\pm ,h}(t,x,y)$.}
		
		Since ${t}^{-\frac{1}{2m}}(| x-{x}_{1}|+|{x}_{2}-y|)\le 1$  in the region $\Omega_1$. Thanks to \eqref{equ3.12-psi}, we apply \eqref{eq-osc-l} in Lemma \ref{lm2.3} with $b=n-1$ ($\lambda^{\frac{n-1}{2}}\lesssim\lambda^{n-1}$ when $\lambda >\frac{\lambda_0}{2}$).  Then it follows that 
		\begin{equation}\label{equ-high-n<4m-I-n1}
			\left | {I}_{n}^{\pm ,h}(t,|x-x_1|,|x_2-y|)\right |\lesssim {t}^{-\frac{n}{2m}}.
		\end{equation}
		
		\noindent ($i$) If ${t}^{-\frac{1}{2m}}|x-y|\le 1$, by \eqref{equ-high-n<4m-I-n1} and the inequality \eqref{eq2.20} in Lemma \ref{lm2.8}, we have
		\begin{equation*}\label{equ-high-n<4m-I-n1'}
			\left | {K}_{n,1}^{\pm ,h}(t,x,y)\right |\lesssim {t}^{-\frac{n}{2m}}\lesssim t^{-\frac{n}{2m}}(1+{t}^{-\frac{1}{2m}}| x-y|)^{-\frac{n(m-1)}{2m-1}}.
		\end{equation*}
		
		\noindent ($ii$)  If ${t}^{-\frac{1}{2m}}|x-y|>1$, we deduce the following
		\begin{align*}\label{equ-K_n1}
			|{K}_{n,1}^{\pm ,h}(t, x, y)|&\lesssim \int_{{\mathbb{R}}^{n}}\int_{{\mathbb{R}}^{n}}t^{-\frac{n}{2m}}(t^{-\frac{1}{2m}}|x-{x}_{1} |)^{-\frac{n(m-1)}{2m-1}}\frac{|\varphi(x_1)|}{{|x-x_1 |}^{\frac{n-1}{2}}}\frac{|\varphi ( {x}_{2})|}{{|{x}_{2}-y|}^{\frac{n-1}{2}}}{\rm d}{x}_{1}{\rm d}{x}_{2}\nonumber\\
			&\lesssim t^{-\frac{n}{2m}}({t}^{-\frac{1}{2m}}\langle x\rangle)^{-\frac{n(m-1)}{2m-1}}\nonumber\\
			&\lesssim {t}^{-\frac{n}{2m}}(1+{t}^{-\frac{1}{2m}}\left | x-y\right |)^{-\frac{n(m-1)}{2m-1}},
		\end{align*}
		where in the first inequality above,  we used \eqref{equ-high-n<4m-I-n1} and the fact that $t^{-\frac{1}{2m}}|x-{x}_{1}| \le 1$ in $\Omega_1$; in the second inequality, we employed the following  estimate
		\begin{equation*}
			\int_{\mathbb{R}^n}^{} |x-x_1|^{- \frac{n(m-1)}{2m-1}-\frac{n-1}{2}}|\varphi(x_1)| \d x_1  \lesssim \langle x\rangle^{-\frac{n(m-1)}{2m-1}-\frac{n-1}{2} }\lesssim \langle x\rangle^{- \frac{n(m-1)}{2m-1}},
		\end{equation*}
		which, in turn, is derived from  the decay assumption \eqref{eq1.2}-\eqref{eq2}, Lemma \ref{lm2.8}, and the fact that
		$\frac{n(m-1)}{2m-1}+\frac{n-1}{2}< n $;
		while in the last inequality, we utilized the 
		relationship that $ |x-y| \lesssim  \langle x\rangle$, which is a consequence of the assumption  $|x|\ge |y|$.
		
		Thus, by ($i$)  and ($ii$), it follows that 
		${K}_{n,1}^{\pm ,h}(t,x,y)$ satisfies the desired estimate \eqref{est-for-high-fi-thm3.1}.

		\noindent{\bf Step 2:  Estimates for ${K}_{n,2}^{\pm ,h}(t,x,y)$.}
		
		Recall that in the region $\Omega_2$, we have  ${t}^{-\frac{1}{2m}}(| x-{x}_{1}|+|{x}_{2}-y|)> 1$. Thanks to \eqref{equ3.12-psi}, we apply \eqref{eq-osc-h} in Lemma \ref{lm2.3} with $b=\frac{n-1}{2}$. Then it follows that 
		\begin{equation*}
			\begin{split}
				\left | {I}_{n}^{\pm ,h}(t,|x-x_1|,|x_2-y|)\right |\lesssim {t}^{-\frac{n+1}{4m}}({t}^{-\frac{1}{2m}}\left | x-{x}_{1}\right |)^{-\frac{m-1-\frac{n-1}{2}}{2m-1}}
			\end{split}
		\end{equation*}
		
		\noindent ($i$)  If ${t}^{-\frac{1}{2m}}|x-y|\le 1$,  inserting the above inequality into \eqref{equ3.9-high}, we deduce that
		\begin{align}\label{equ-K_n2}
			|{K}_{n,2}^{\pm ,h}(t, x, y)|&\lesssim
			\int_{{\mathbb R}^{n}}^{}\int_{{\mathbb R}^{n}}^{} {t}^{-\frac{n+1}{4m}}({t}^{-\frac{1}{2m}}\left | x-{x}_{1}\right |)^{-\frac{m-1-\frac{n-1}{2}}{2m-1}}\frac{|\varphi( {x}_{1})|}{{|x_1-x|}^{\frac{n-1}{2}}}\frac{|\varphi( {x}_{2})|}{{|x_2-y|}^{\frac{n-1}{2}}}{\rm d}{x}_{1}{\rm d}{x}_{2}\nonumber\\
			&=\int_{{\mathbb R}^{n}}^{}\int_{{\mathbb R}^{n}}^{}{t}^{-\frac{n}{2m}}(t^{-\frac{1}{2m}}|x-{x}_{1}|)^{-\frac{n(m-1)}{2m-1}}|\varphi ({x}_{1})|\frac{|\varphi( {x}_{2})|}{{|x_2-y|}^{\frac{n-1}{2}}}{\rm d}{x}_{1}{\rm d}{x}_{2}\nonumber\\
			&\lesssim{t}^{-\frac{n}{2m}}\lesssim{t}^{-\frac{n}{2m}}(1+{t}^{-\frac{1}{2m}}\left | x-y\right |)^{-\frac{n(m-1)}{2m-1}},
		\end{align}
		where in the  equality above, we used the  identity:
		\begin{equation}\label{equ-ident}
			\frac{m-1-\frac{n-1}{2}}{2m-1}+\frac{n-1}{2}=\frac{n(m-1)}{2m-1};
		\end{equation}
		the second inequality follows from Lemma \ref{lm2.8} and the fact that ${t}^{-\frac{1}{2m}}|x-{x}_{1}|\ge 1$ in $\Omega_2$; while in the last inequality, we used the assumption that ${t}^{-\frac{1}{2m}}|x-y|\le 1$.
		
		\noindent  ($ii$)  If ${t}^{-\frac{1}{2m}}|x-y|>1$,  we derive that
		\begin{align}\label{equ-K_n2-1}
			|{K}_{n,2}^{\pm ,h}(t, x, y)|&\lesssim\int_{{\mathbb R}^{n}}^{}\int_{{\mathbb R}^{n}}^{}{t}^{-\frac{n}{2m}}(t^{-\frac{1}{2m}}|x-{x}_{1}|)^{-\frac{n(m-1)}{2m-1}}|\varphi ({x}_{1})|\frac{|\varphi( {x}_{2})|}{{|x_2-y|}^{\frac{n-1}{2}}}{\rm d}{x}_{1}{\rm d}{x}_{2}\nonumber\\
			&\lesssim{t}^{-\frac{n}{2m}}(t^{-\frac{1}{2m}}\langle x\rangle)^{-\frac{n(m-1)}{2m-1}}\nonumber\\
			&\lesssim{t}^{-\frac{n}{2m}}(1+{t}^{-\frac{1}{2m}}\left | x-y\right |)^{-\frac{n(m-1)}{2m-1}},
		\end{align}
		where the above first inequality follows from the same as the second line in \eqref{equ-K_n2};
		in the second inequality, we utilized  the following estimate
		\begin{equation*}
			\int_{{\mathbb R}^{n}}^{}|x-{x}_{1}|^{-\frac{n(m-1)}{2m-1}}|\varphi ({x}_{1})|{\rm d}{x}_{1}\lesssim \langle x\rangle^{-\frac{n(m-1)}{2m-1}},
		\end{equation*}
		which follows from Lemma \ref{lm2.8} and the decay assumption \eqref{eq1.2}.
		
		By \eqref{equ-K_n2} and \eqref{equ-K_n2-1}, it follows that${K}_{n,2}^{\pm ,h}(t,x,y)$ also satisfies the desired estimate \eqref{est-for-high-fi-thm3.1}.
		Therefore,  combining  results from {\bf Step   1} and  {\bf Step   2}, we prove the pointwise estimate \eqref{est-for-high-fi-thm3.1} in dimensions $n\le 4m-1$.

		\subsection{The proof of Theorem \ref{thm3.1} for $n\geq 4m$}\label{sec3.3}\
		
		In view of  \eqref{equ-2-high} and the  form of the resolvent kernel given in \eqref{eq-decom-r-nl2m}, it suffices to consider the following four terms:
		\begin{align*}
			{K}_{n,1}^{\pm ,h}(t,x,y) := & \int_{0}^{+\infty }\int_{{\mathbb{R}}^{n}}^{}\int_{{\mathbb{R}}^{n}}^{}{e}^{-{\rm i}t{\lambda }^{2m}\pm {\rm i}\lambda( | x-{x}_{1}|+|{x}_{2}-y|)}{F}_{\pm }^{\alpha ,h} ( {\lambda }^{2m}){\lambda }^{2m-1}\\
			&\times W_{0}^{\pm}(\lambda |x-x_1|)W_{0}^{\pm}(\lambda |x_2-y|)
			\frac{\varphi (x_1)}{{\left | x-{x}_{1}\right |}^{n-2m}}\frac{\varphi (x_2)}{{\left | {x}_{2}-y\right |}^{n-2m}}{\rm d} \lambda {\rm d}{x}_{1}{\rm d}{x}_{2},
		\end{align*}
		\begin{align*}
			{K}_{n,2}^{\pm ,h}(t,x,y):= &  \int_{0}^{+\infty }\int_{{\mathbb{R}}^{n}}^{}\int_{{\mathbb{R}}^{n}}^{}{e}^{-{\rm i}t{\lambda }^{2m}\pm {\rm i}\lambda(| x-{x}_{1}|+|x_2-y|)}{F}_{\pm }^{\alpha ,h} ({\lambda }^{2m}){\lambda }^{\frac{n-1}{2}}\\
			&\times W_{1}^{\pm}(\lambda |x-x_1|)W_{0}^{\pm}(\lambda |x_2-y|) \frac{\varphi (x_1)}{{| x-{x}_{1}|}^{\frac{n-1}{2}}}\frac{\varphi (x_2)}{{|{x}_{2}-y|}^{n-2m}}{\rm d} \lambda {\rm d}{x}_{1}{\rm d}{x}_{2},
		\end{align*}
		\begin{align*}
			{K}_{n,3}^{\pm ,h} (t,x,y)& := \int_{0}^{+\infty }\int_{{\mathbb{R}}^{n}}^{}\int_{{\mathbb{R}}^{n}}^{}{e}^{-{\rm i}t{\lambda }^{2m}\pm {\rm i}\lambda(|x-x_1|+|{x}_{2}-y|)}{F}_{\pm }^{\alpha ,h} ( {\lambda }^{2m}){\lambda }^{\frac{n-1}{2}}\\
			&\times W_{0}^{\pm}(\lambda |x-x_1|)W_{1}^{\pm}(\lambda |x_2-y|) \frac{\varphi (x_1)}{{| x-{x}_{1}|}^{n-2m}}\frac{\varphi (x_2)}{{|{x}_{2}-y|}^{\frac{n-1}{2}}}{\rm d} \lambda {\rm d}{x}_{1}{\rm d}{x}_{2},
		\end{align*}
		and
		\begin{align*}
			{K}_{n,4}^{\pm ,h}\left ( t,x,y\right )&:= \int_{0}^{+\infty }\int_{{\mathbb{R}}^{n}}^{}\int_{{\mathbb{R}}^{n}}^{}{e}^{-{\rm i}t{\lambda }^{2m}\pm {\rm i}\lambda(|x-x_1|+|{x}_{2}-y|)}{F}_{\pm }^{\alpha ,h} ( {\lambda }^{2m}){\lambda }^{n-2m}\\
			&\times W_{1}^{\pm}(\lambda |x-x_1|)W_{1}^{\pm}(\lambda |x_2-y|) \frac{\varphi (x_1)}{{| x-{x}_{1}|}^{\frac{n-1}{2}}}\frac{\varphi (x_2)}{{|{x}_{2}-y|}^{\frac{n-1}{2}}}{\rm d} \lambda {\rm d}{x}_{1}{\rm d}{x}_{2}.
		\end{align*}
		
		In the following we also assume that \eqref{eq-sym-xy} holds and divide the proof into two steps.
		
		\noindent{\bf Step 1:  Estimates for ${K}_{n,j}^{\pm ,h}(t,x,y)$, $j=1,2,3$.}
		
		We shall reduce the analysis of ${K}_{n,j}^{\pm ,h}(t,x,y)\ (j=1,2,3)$ to the estimate of the following oscillatory integrals :
		\begin{equation}\label{n>4m-osc-form}
			I_{j}(t,|x-x_1|,|x_2-y|):=\int_0^\infty  e^{\i t\lambda^{2m}\pm \i \lambda(|x-x_1|+|x_2-y|)}{F}_{\pm }^{\alpha ,h}( {\lambda }^{2m})\psi_{j}(\lambda,|x-x_1|,|x_2-y|)\d\lambda,\quad t>0,
		\end{equation}
		where 
		$$\psi_{1}(\lambda,|x-x_1|,|x_2-y|)={\lambda }^{2m-1} W_{0}^{\pm}(\lambda |x-x_1|)W_{0}^{\pm}(\lambda |x_2-y|),$$
		$$\psi_{2}(\lambda,|x-x_1|,|x_2-y|)={\lambda }^{\frac{n-1}{2}} W_{1}^{\pm}(\lambda |x-x_1|)W_{0}^{\pm}(\lambda |x_2-y|),$$
		$$\psi_{3}(\lambda,|x-x_1|,|x_2-y|)={\lambda }^{\frac{n-1}{2}} W_{0}^{\pm}(\lambda |x-x_1|)W_{1}^{\pm}(\lambda |x_2-y|).$$
		By \eqref{eq-expansion-high-6}, it follows that $W_0^{\pm}(\cdot)$ and $W_1^{\pm}(\cdot)$ have the same estimate. Meanwhile, we have  $\lambda^{2m-1}\le \lambda^{\frac{n-1}{2}}$ for $\lambda>\lambda_0/2$. Thus we have
		\begin{equation}\label{eq-est-phi-1}
			\sup_{x,x_1,x_2,y}\left | \partial_\lambda^l \psi_j (\lambda,|x-x_1|,|x_2-y|)\right |\lesssim \lambda^{\frac{n-1}{2}-l},\quad  0\le l\le \left [\frac{n}{2m}\right]+1.
		\end{equation}
		We proceed to  prove the estimate specifically for ${K}_{n,1}^{\pm ,h}$, noting that the proofs for ${K}_{n,j}^{\pm ,h}$ ($j=2,3$) are the same. We shall discuss the following three cases separately.
		
		{\it Case 1: 
			$t^{-\frac{1}{2m}}|x-y|\le 1$ and  $t^{-\frac{1}{2m}}(|x-x_1|+|x_2-y|)\le 1$.} We perform integration by parts $n_0=\left[\frac{n}{2m}\right]$ times in the integral \eqref{n>4m-osc-form} by using the following identity:
		\begin{equation*}
			\frac{1}{2\i tm\lambda^{2m-1}}\cdot \frac{d}{d\lambda}e^{\i t\lambda^{2m}}=e^{\i t\lambda^{2m}}.
		\end{equation*}
		Then, we have
		\begin{equation*}
			\begin{split}
				I_1(t,|x-x_1|,|x_2-y|)=&\sum_{s_1+s_2+s_3=n_0} C_{s_1 s_2 s_3}t^{-n_0}(|x-x_1|+|x_2-y|)^{s_1}\int_\Omega e^{\i t\lambda^{2m}\pm \i \lambda(|x-x_1|+|x_2-y|)}\\
				&\qquad\cdot\frac{d^{s_2}}{d\lambda^{s_2}}{F}_{\pm }^{\alpha ,h} ( {\lambda }^{2m})\frac{d^{s_3}}{d\lambda^{s_3}}\psi_1(\lambda,|x-x_1|,|x_2-y|)\lambda^{-(2m-1)n_0} \d\lambda,
			\end{split}
		\end{equation*}
		and by \eqref{equ2.2.1.1} and \eqref{eq-est-phi-1}, it follows that for $0\le r\le 1$,
		\begin{equation*}
			\left | \frac{d^r}{d\lambda^r}\left ( \frac{d^{s_2}}{d\lambda^{s_2}}{F}_{\pm }^{\alpha ,h} ( {\lambda }^{2m})\frac{d^{s_3}}{d\lambda^{s_3}}\psi_1(\lambda,|x-x_1|,|x_2-y|)\lambda^{-(2m-1)n_0} \right )  \right |\lesssim \lambda^{\frac{n-1}{2}-(2m-1)n_0-s_3-r}.
		\end{equation*}
		
		Note that if $n\notin 2m\mathbb{N}$, in order to utilize Lemma \ref{lm2.3}, we observe that
		$$\frac{n-1}{2}-(2m-1)n_0-s_3\le n-1-2mn_0 < 2m-1.$$ 
		This allows us to apply \eqref{eq-osc-l} with $b=n-1-2mn_0$ and $ k=1$ and derive that
		\begin{equation}\label{eq-I(t,x)-l}
			\begin{split}
				\left|I_1(t,|x-x_1|,|x_2-y|)\right|\lesssim & \sum_{s_1=0}^{n_0} t^{-n_0}|x-x_1|^{s_1}t^{-\frac{1+n-1-2mn_0}{2m}} = \sum_{s_1=0}^{n_0}t^{-\frac{n}{2m}}|x-x_1|^{s_1}\\
				&\lesssim t^{-\frac{n}{2m}}(1+|x-x_1|^{\frac{n-1}{2}}),
			\end{split}
		\end{equation}
		where in the last inequality, we used the fact 
		$$\sum_{s_1=0}^{n_0}|x-x_1|^{s_1}\lesssim 1+|x-x_1|^{\frac{n-1}{2}}.$$
		This is valid since $n_0=\left[\frac{n}{2m}\right] \le \frac{n-1}{2}$.
		
		If $n\in 2m\mathbb{N}$, then $n_0=\frac{n}{2m}\in \mathbb{N}$. To apply Lemma \ref{lm2.3}, we note that
		$$\frac{n-1}{2}-(2m-1)n_0-s_3\le 0,$$ 
		thus, we apply \eqref{eq-osc-l} with $b=0,\ k=1$ to obtain that
		\begin{equation}\label{eq-I(t,x)-l'}
			\begin{split}
				|I_1(t,|x-x_1|,|x_2-y|)|\lesssim & \sum_{s_1=0}^{n_0}t^{-\frac{n}{2m}}|x-x_1|^{s_1}t^{-\frac{1}{2m}} \lesssim t^{-\frac{n}{2m}}\sum_{s_1=0}^{n_0}|x-x_1|^{s_1-1}\\
				\lesssim & t^{-\frac{n}{2m}}(|x-x_1|^{-1}+|x-x_1|^{\frac{n-1}{2}}),
			\end{split}
		\end{equation}
		where in the second inequality, we used the assumption $t^{-\frac{1}{2m}}(|x-x_1|+|x_2-y|)\le 1$ and $|x-x_1|\ge |x_2-y|$ in this case; while in the last inequality, we employed the inequality
		$$\sum_{s_1=0}^{n_0}|x-x_1|^{s_1-1}\lesssim |x-x_1|^{-1}+|x-x_1|^{\frac{n-1}{2}},$$ 
		which, in turn,  follows from the fact  $-1\le s_1-1\le\frac{n}{2m}-1 \le \frac{n-1}{2}$ in this case.
		
		Therefore, combining \eqref{eq-I(t,x)-l} and \eqref{eq-I(t,x)-l'}, we derive that
		\begin{align}\label{eq-kn1-1}
			\left|{K}_{n,1}^{\pm ,h} (t,x,y)  \right| &\lesssim \int_{{\mathbb{R}}^{n}}\int_{{\mathbb{R}}^{n}}  t^{-\frac{n}{2m}}(|x-x_1|^{-1}+|x-x_1|^{\frac{n-1}{2}})  \frac{|\varphi (x_1)|}{{|x-{x}_{1}|}^{n-2m}}\frac{|\varphi (x_2)|}{{\left | {x}_{2}-y\right |}^{n-2m}}{\rm d}{x}_{1}{\rm d}{x}_{2} \nonumber\\ 
			&\lesssim t^{-\frac{n}{2m}}\lesssim t^{-\frac{n}{2m}}(1+{t}^{-\frac{1}{2m}}| x-y|)^{-\frac{n(m-1)}{2m-1}},
		\end{align}
		where in the second inequality, we used the estimate
		\begin{equation*}
			\int_{{\mathbb{R}}^{n}}|x-x_1|^{-(1+n-2m)}|\varphi (x_1)|\d x_1\lesssim \langle x \rangle^{-(1+n-2m)} \lesssim C,
		\end{equation*}
		and 
		\begin{equation*}
			\int_{{\mathbb{R}}^{n}}|x-x_1|^{-(n-2m-\frac{n-1}{2})}|\varphi (x_1)|\d x_1\lesssim \langle x \rangle^{-(n-2m-\frac{n-1}{2})} \lesssim C,
		\end{equation*}
		which follow from Lemma \ref{lm2.8} as well as the facts that $0\le 1+n-2m<n$ and $0\le n-2m-\frac{n-1}{2}<n$; in the last inequality, we used the assumption that $t^{-\frac{1}{2m}}|x-y|\le 1$.

		{\it Case 2: $t^{-\frac{1}{2m}}|x-y|\le 1$ and  $t^{-\frac{1}{2m}}(|x-x_1|+|x_2-y|)>1$.} In this case, we apply \eqref{eq-osc-h} with $b=(2m-1)(\frac{n}{2m}-n_0)-\frac12$ and $k=1$ to deduce that
		\begin{equation*}\label{eq-I(t,x)-3}
			\begin{split}
				\left|I_1(t,|x-x_1|,|x_2-y|)\right|\lesssim &\sum_{s_1=0}^{n_0}t^{-\frac{n}{2m}}|x-x_1|^{s_1+\frac{n}{2m}-n_0-\frac 12}\lesssim t^{-\frac{n}{2m}}(1+|x-x_1|^{\frac{n-1}{2}}),
			\end{split}
		\end{equation*}
		where the second inequality is justified by the fact that $0\le s_1+\frac{n}{2m}-n_0-\frac 12\le \frac{n-1}{2}$. By employing the same arguments to that used in  \eqref{eq-kn1-1}, we deduce that
		\begin{equation}\label{eq-kn1-2}
			\begin{aligned}
				\left|{K}_{n,1}^{\pm ,h} (t,x,y)  \right|  \lesssim t^{-\frac{n}{2m}}(1+{t}^{-\frac{1}{2m}}|x-y|)^{-\frac{n(m-1)}{2m-1}}.
			\end{aligned}
		\end{equation}

		{\it Case 3: $t^{-\frac{1}{2m}}|x-y|> 1$ and $t^{-\frac{1}{2m}}(|x-x_1|+|x_2-y|)\le 1$.} We perform integration by parts $n_1=\left[\frac{n-1}{2(2m-1)}\right]$ times  in the integral \eqref{n>4m-osc-form} to obtain
		\begin{equation*}
			\begin{split}
				I_1(t,|x-x_1|,|x_2-y|)=&\sum_{s_1+s_2+s_3=n_1} C_{s_1 s_2 s_3}t^{-n_1}(|x-x_1|+|x_2-y|)^{s_1}\int_\Omega e^{\i t\lambda^{2m}\pm \i \lambda(|x-x_1|+|x_2-y|)}\\
				&\qquad\cdot\frac{d^{s_2}}{d\lambda^{s_2}}{F}_{\pm }^{\alpha ,h} ({\lambda }^{2m})\frac{d^{s_3}}{d\lambda^{s_3}}\psi_1(\lambda,|x-x_1|,|x_2-y|)\lambda^{-(2m-1)n_1} \d\lambda,
			\end{split}
		\end{equation*}
		and for $r=0, 1$, it follows that
		\begin{equation*}
			\left | \frac{d^r}{d\lambda^r}\left ( \frac{d^{s_2}}{d\lambda^{s_2}}{F}_{\pm }^{\alpha ,h} ({\lambda }^{2m})\frac{d^{s_3}}{d\lambda^{s_3}}\psi_1(\lambda,|x-x_1|,|x_2-y|)\lambda^{-(2m-1)n_1} \right )  \right |\lesssim \lambda^{\frac{n-1}{2}-(2m-1)n_1-s_3-r}.
		\end{equation*}
		Moreover,
		note that $\frac{n-1}{2}-(2m-1)n_1-s_3\le \frac{n-1}{2}-(2m-1)n_1< 2m-1$, this allows us to apply \eqref{eq-osc-l} with $b=\frac{n-1}{2}-(2m-1)n_1,\ k=1$ and derive the following: 
		\begin{equation}\label{eq-1149}
			\begin{split}
				\left|I_1(t,|x-x_1|,|x_2-y|)\right|\lesssim \sum_{s_1=0}^{n_1} t^{-n_1}|x-x_1|^{s_1}t^{-\frac{1+\frac{n-1}{2}-(2m-1)n_1}{2m}}=\sum_{s_1=0}^{n_1}t^{-\frac{\frac{n+1}{2}}{2m}} t^{-\frac{n_1}{2m}}|x-x_1|^{s_1}.
			\end{split}
		\end{equation}
		
		$(i)$ If $n_1\ge -\frac{m-1-\frac{n-1}{2}}{2m-1}$, we deduce that
		\begin{align*}
			\left|I_1(t,|x-x_1|,|x_2-y|)\right|&\lesssim\sum_{s_1=0}^{n_1}t^{-\frac{\frac{n+1}{2}}{2m}}|x-x_1|^{s_1-n_1}(t^{-\frac{1}{2m}}|x-x_1|)^{n_1}\nonumber\\
			&\lesssim  \sum_{s_1=0}^{n_1} t^{-\frac{\frac{n+1}{2}}{2m}}|x-x_1|^{s_1-n_1} (t^{-\frac{1}{2m}}|x-x_1|)^{-\frac{m-1-\frac{n-1}{2}}{2m-1}}, 
		\end{align*}
		where we used $n_1\ge -\frac{m-1-\frac{n-1}{2}}{2m-1}$ and $t^{-\frac{1}{2m}}|x-x_1|\le 1$ in this case. Thus, it follows that
		\begin{align}\label{eq-1162}
			&\left|{K}_{n,1}^{\pm ,h} (t,x,y)  \right| \nonumber\\
			\lesssim& \sum_{s_1=0}^{n_1} t^{-\frac{\frac{n+1}{2}}{2m}} \int_{{\mathbb{R}}^{n}}\int_{{\mathbb{R}}^{n}} |x-x_1|^{s_1-n_1} (t^{-\frac{1}{2m}}|x-x_1|)^{-\frac{m-1-\frac{n-1}{2}}{2m-1}} \frac{|\varphi (x_1)|}{{\left | x-{x}_{1}\right |}^{n-2m}}\frac{|\varphi (x_2)|}{{\left | {x}_{2}-y\right |}^{n-2m}}{\rm d}{x}_{1}{\rm d}{x}_{2} \nonumber\\ 
			\lesssim&  t^{-\frac{\frac{n+1}{2}}{2m}}({t}^{-\frac{1}{2m}}\langle x\rangle)^{-\frac{m-1-\frac{n-1}{2}}{2m-1}} \langle x\rangle^{-\frac{n-1}{2}} \nonumber\\
			\lesssim& t^{-\frac{n}{2m}}(1+{t}^{-\frac{1}{2m}}| x-y|)^{-\frac{n(m-1)}{2m-1}},
		\end{align}
		where in the second inequality, we observe that by the definition of $n_1$, we have $n_1\le \frac{n-1}{2(2m-1)}$. This implies that  $-1\le s_1-n_1-\frac{m-1-\frac{n-1}{2}}{2m-1}\le n-2m$. Thus the following inequality holds 
		\begin{equation*}
			0\le n_1-s_1+\frac{m-1-\frac{n-1}{2}}{2m-1}+n-2m<n.
		\end{equation*}
		This, together with Lemma \ref{lm2.8} and $n-2m\ge \frac{n-1}{2}$, yields
		\begin{align*}
			\int_{{\mathbb{R}}^{n}}|x-x_1|^{-(n_1-s_1+\frac{m-1-\frac{n-1}{2}}{2m-1}+n-2m)}|\varphi (x_1)|\d x_1\lesssim \langle x\rangle^{-(n_1-s_1)}\langle x\rangle^{-\frac{m-1-\frac{n-1}{2}}{2m-1}-\frac{n-1}{2}}
			\lesssim \langle x\rangle^{-\frac{m-1-\frac{n-1}{2}}{2m-1}-\frac{n-1}{2}};
		\end{align*}
		in the last inequality, we used the assumption $|x-y|\lesssim \langle x\rangle $ and the identity \eqref{equ-ident}.
		
		$(ii)$ If $n_1< -\frac{m-1-\frac{n-1}{2}}{2m-1}$, thanks to \eqref{eq-1149}, we have
		\begin{align}\label{eq-1177}
			\left|{K}_{n,1}^{\pm ,h} (t,x,y)  \right|&\lesssim \sum_{s_1=0}^{n_1} t^{-\frac{\frac{n+1}{2}}{2m}}t^{-\frac{n_1}{2m}} \int_{{\mathbb{R}}^{n}}\int_{{\mathbb{R}}^{n}} |x-x_1|^{s_1} \frac{|\varphi (x_1)|}{{|x-{x}_{1}|}^{n-2m}}\frac{|\varphi (x_2)|}{{|x_2-y|}^{n-2m}}{\rm d}{x}_{1}{\rm d}{x}_{2}\nonumber\\
			&\lesssim \sum_{s_1=0}^{n_1} t^{-\frac{\frac{n+1}{2}}{2m}}t^{-\frac{n_1}{2m}}\langle x \rangle^{-(n-2m-s_1)}\nonumber\\
			&\lesssim t^{-\frac{n}{2m}}(t^{-\frac{1}{2m}}\langle x \rangle)^{-(\frac{n-1}{2}-n_1)}\nonumber\\
			&\lesssim t^{-\frac{n}{2m}}(1+{t}^{-\frac{1}{2m}}| x-y|)^{-\frac{n(m-1)}{2m-1}},
		\end{align}
		in the second inequality, we used
		\begin{equation*}
			\int_{{\mathbb{R}}^{n}}|x-{x}_{1}|^{-(n-2m-s_1)}|\varphi (x_1)|\d x_1 \lesssim \langle x \rangle^{-(n-2m-s_1)},
		\end{equation*}
		which is deduced by Lemma \ref{lm2.8}, with the condition $0\le n-2m-s_1<n$; in the third inequality, we used  $n-2m-s_1\ge \frac{n-1}{2}-n_1$, which holds  when $n\ge 4m$; in the last inequality, we made use of $|x-y|\lesssim \langle x \rangle$ and
		$$-\left(\frac{n-1}{2}-n_1\right)\le -\frac{n(m-1)}{2m-1},$$
		which follows from the assumption $n_1< -\frac{m-1-\frac{n-1}{2}}{2m-1}$ and the identity \eqref{equ-ident}.
		
		It follows from \eqref{eq-1162} and \eqref{eq-1177} that
		\begin{equation}\label{eq-kn1-5}
			\left|{K}_{n,1}^{\pm ,h} (t,x,y)  \right|\lesssim t^{-\frac{n}{2m}}(1+{t}^{-\frac{1}{2m}}| x-y|)^{-\frac{n(m-1)}{2m-1}}.
		\end{equation}
		
		{\it Case 4: $t^{-\frac{1}{2m}}|x-y|> 1$ and $t^{-\frac{1}{2m}}(|x-x_1|+|x_2-y|)> 1$.} Following the same approach used in {\it Case 3}, we employ \eqref{eq-osc-h} with $b=\frac{n-1}{2}-(2m-1)n_1,\ k=1$ and obtain
		\begin{align*}
			\left|I_1(t,|x-x_1|,|x_2-y|)\right|&\lesssim \sum_{s_1=0}^{n_1} t^{-n_1}|x-x_1|^{s_1}t^{-\frac{1+\frac{n-1}{2}-(2m-1)n_1}{2m}}(t^{-\frac{1}{2m}}|x-x_1|)^{-\frac{m-1-\frac{n-1}{2}+(2m-1)n_1}{2m-1}}\nonumber\\
			& = \sum_{s_1=0}^{n_1} t^{-\frac{\frac{n+1}{2}}{2m}}|x-x_1|^{s_1-n_1} (t^{-\frac{1}{2m}}|x-x_1|)^{-\frac{m-1-\frac{n-1}{2}}{2m-1}}.
		\end{align*}
		Thus,  by employing the same  arguments as in \eqref{eq-1162}, we deduce that
		\begin{equation}\label{eq-kn1-3}
			\left|{K}_{n,1}^{\pm ,h} (t,x,y)  \right|\lesssim t^{-\frac{n}{2m}}(1+{t}^{-\frac{1}{2m}}| x-y|)^{-\frac{n(m-1)}{2m-1}}.
		\end{equation}
		This, together with \eqref{eq-kn1-1}, \eqref{eq-kn1-2} and \eqref{eq-kn1-5}, yields the desired estimate for ${K}_{n,1}^{\pm ,h} (t,x,y)$.

		\noindent{\bf Step 2:  Estimate for ${K}_{n,4}^{\pm ,h}(t,x,y)$.}
		
		For ${K}_{n,4}^{\pm ,h}( t,x,y)$, in contrast to ${K}_{n,j}^{\pm ,h}( t,x,y)$ for $j=1,2,3$, the main difference is that when $n\ge 4m$, $n-2m$ can exceed $\frac{n-1}{2}$, thus differing from the structure in \eqref{n>4m-osc-form}. In order to overcome this difficulty, we borrow ideas from  \cite{EG10, CHZ} and employ integration by parts arguments with respect to the oscillating term $e^{\i \lambda(|x-x_1|+ |x_2-y|)}$. To be more precise, let us define
		\begin{equation*}
			{L}_{{x}_{1}}:=\frac{1}{{\rm i}\lambda }\frac{{x}_{1}-x}{\left |{x}_{1}-x \right |}\cdotp {\nabla}_{{x}_{1}},\quad  \quad {L}_{{x}_{2}}:=\frac{1}{{\rm i}\lambda }\frac{{x}_{2}-y}{\left |{x}_{2}-y \right |}\cdotp {\nabla}_{{x}_{2}},
		\end{equation*}
		and their dual operators
		\begin{equation*}
			{L}_{{x}_{1}}^{*}=\frac{{\rm i}}{\lambda }{\nabla}_{{x}_{1}}\left ( \frac{{x}_{1}-x}{\left |{x}_{1}-x \right |}\right ),\quad  \quad  {L}_{{x}_{2}}^{*}=\frac{{\rm i}}{\lambda }{\nabla}_{{x}_{2}}\left ( \frac{{x}_{2}-y}{\left |{x}_{2}-y \right |}\right ),
		\end{equation*}
		Note that
		\begin{equation*}
			{L}_{{x}_{1}}{e}^{{\rm i}\lambda |x-x_1|}={e}^{\i\lambda |x-x_1|},\qquad
			{L}_{{x}_{2}}{e}^{{\rm i}\lambda |x_2-y|}={e}^{\i\lambda |x_2-y|}.
		\end{equation*}
		Set
		\begin{equation*}
			v=\left[\frac{n+5}{4}\right]-m.
		\end{equation*}
		We apply the operators ${L}_{{x}_{1}}$ and  ${L}_{{x}_{2}}$ to the terms ${e}^{{\rm i}\lambda  \left | x-{x}_{1}\right |}$  and ${e}^{{\rm i}\lambda \left | {x}_{2}-y\right |}$ $v$ times respectively. Then we can express ${K}_{n,4}^{\pm ,h}( t,x,y)$ as
		\begin{equation*}
			\begin{split}
				&\int_0^{+\infty} \int_{\mathbb{R}^n} \int_{\mathbb{R}^n} e^{-{\rm i} t \lambda^{2m} \pm {\rm i}\lambda(|x-x_1|+|x_2-y|)} F_{\pm}^{\alpha,h}(\lambda^{2m}) \lambda^{n-2m}\\
				&\cdotp{({L}_{{x}_{1}}^{*})}^{v}(\frac{W_{1}^{\pm}(\lambda|x-x_1|) \varphi(x_1)}{|x-x_1|^{\frac{n-1}{2}}}){({L}_{{x}_{2}}^{*})}^{v}(\frac{W_{1}^{\pm}(\lambda|x_2-y|) \varphi(x_2)}{|x_2-y|^{\frac{n-1}{2}}}){\rm d} \lambda {\rm d}x_1 {\rm d}x_2,
			\end{split}
		\end{equation*}
		which can be further written as a combination of
		\begin{equation*}
			\int_{\mathbb{R}^n} \int_{\mathbb{R}^n} I_{n, 4, \beta_1, \beta_2}^{\pm , h}\left(t,\left|x-x_1\right|,\left|x_2-y\right|\right){G}_1\left(x, x_1\right){G}_2\left(x_2, y\right) {\rm d} x_1 {\rm d} x_2 ,
		\end{equation*}
		where $0\le \left | {\beta }_{1}\right |,\left | {\beta }_{2}\right |\le n$ and
		\begin{align*}
			\begin{split}
				&I_{n, 1, \beta_1, \beta_2}^{ \pm, h}\left(t,\left|x-x_1\right|,\left|x_2-y\right|\right)=\int_{0}^{+\infty } e^{-{\rm i} t \lambda^{2 m} \pm {\rm i}\lambda\left(|x-x_1|+|x_2-y|\right)} F_{\pm}^{\alpha,h}(\lambda^{2m})\psi_4(\lambda,|x-x_1|,|x_2-y|)\rm d \lambda,
			\end{split}
		\end{align*}
		with
		\begin{equation*}
			\begin{split}
				\psi_4(\lambda,|x-x_1|,|x_2-y|)=&\lambda^{n-2m-2v}\cdot(\lambda|x-x_1|)^{|\beta_1|} (W_{1}^{\pm})^{ (|\beta_1|)}(\lambda|x-x_1|)\\
				&\times(\lambda|x_2-y|)^{|\beta_2|}(W_1^{\pm})^{(|\beta_2|)}(\lambda|x_2-y|).
			\end{split}
		\end{equation*}
		A direct computation yields that
		\begin{equation*}
			\begin{split}
				&\left|G_1\left(x, x_1\right)\right| \leq C \cdot M(\left|x-x_1\right|^{-n+2m}+\left|x-x_1\right|^{-\frac{n-1}{2}}) \cdot{\left \langle {x}_{1}\right \rangle}^{-\delta }, \delta>n+\frac{3}{2},\\
				&\left|G_2\left(x_2,y\right)\right| \leq C \cdot M(\left|x_2-y\right|^{-n+2m}+\left|x_2-y\right|^{-\frac{n-1}{2}}) \cdot{\left \langle {x}_{2}\right \rangle}^{-\delta }, \delta>n+\frac{3}{2}.
			\end{split}
		\end{equation*}
		Thanks to \eqref{eq-expansion-high-6}, \eqref{equ2.2.1.1} and the fact that $n-2m-2v\le \frac{n-1}{2}$, we deduce that
		\begin{equation*}
			\left | \partial_\lambda^l \psi_4(\lambda,|x-x_1|,|x_2-y|)\right |\lesssim \lambda^{\frac{n-1}{2}-l},\ 0\le l\le \left [\frac{n}{2m}\right]+1.
		\end{equation*}
		Now the condition in \eqref{n>4m-osc-form} is satisfied, then the result for ${K}_{n,4}^{\pm ,h}(t,x,y)$ can be derived using the same arguments as those presented in {\bf Step 1}.
		
		Therefore, we complete the proof of Theorem \ref{thm3.1}.
		\qed

		\begin{remark}\label{rmk3.1}
			We mention that during the proof in the high energy part, the only property that we need for $F^{\alpha, h}_{\pm}(\lambda^{2m})$ is \eqref{equ2.2.1.1}. This will be used in the analysis of finite rank perturbations.
		\end{remark}
		
		\section{Proof of Theorem \ref{thm-1.1}: Low energy estimate}\label{sec4}
		
		The main result in this section is the following theorem.
		\begin{theorem}\label{thm3.2}
			Under the assumptions of Theorem \ref{thm-1.1},  the low energy part ${K}_{n}^{\pm, l}$, defined by \eqref{equ-2-low}, has integral kernel ${K}_{n}^{\pm, l}(t,x,y)$ satisfying
			\begin{equation}\label{eq3.22.11}
				\left|({K}_{n}^{+,l}-{K}_{n}^{-,l})(t,x,y)\right| \le C{t}^{-\frac{n}{2m}}(1+{t}^{-\frac{1}{2m}}|x-y|)^{-\frac{n(m-1)}{2m-1}},\ \ \ t>0,\ x,y\in\mathbb{R}^n.
			\end{equation}
		\end{theorem}
		
		In view of \eqref{equ-2-low}, we first investigate some properties of the function
		\begin{equation*}
			{F}_{\pm }^{\alpha,l}(\lambda^{2m}):=\frac{\chi (\lambda )}{1+\alpha {F}^{\pm }({\lambda }^{2m})}.
		\end{equation*}
		These properties play important roles  during the proof of Theorem \ref{thm3.2}, which will be established  in subsection \ref{sec-rank-one-F}.
		It turns out that the asymptotic behavior of  ${F}_{\pm}^{\alpha ,l}$ near zero  varies depending on whether  $n>2m$ or $n\le 2m$.
		Following this analysis, we proceed to  prove Theorem \ref{thm3.2} for the cases where  $n>2m$, $1\le n < 2m$ and   $n=2m$ respectively.
		
		\subsection{Properties of $F_{\pm }^{\alpha,l}$.}\label{sec-rank-one-F}\

		We first state the result for $n>2m$.
		
		\begin{lemma}\label{lm3.3}
			Let $n>2m$ and $\varphi$ satisfy Assumption \ref{assumption-1}. Then there exists a small constant $\lambda_0\in (0,1)$ such that for $0\le l\le \left [\frac{n}{2m}\right ]+1$, we have
			\begin{equation}\label{equ 2.4.11.48}
				\left|\frac{{d}^{l}}{d{\lambda }^{l}}{F}_{\pm }^{\alpha ,l}({\lambda }^{2m})\right|\le C{\lambda }^{-l},\ 0<\lambda<\lambda_0,
			\end{equation}
			and
			\begin{equation}\label{eq3.11.48}
				\left| \frac{{d}^{l}}{d{\lambda }^{l}}({F}_{+}^{\alpha ,l}({\lambda }^{2m})-{F}_{-}^{\alpha ,l}({\lambda }^{2m}))\right|\le C{\lambda }^{n-2m-l}, \ 0<\lambda<\lambda_0.
			\end{equation}
		\end{lemma}
		
		\begin{proof}
			The estimate (\ref{equ 2.4.11.48}) follows directly from the expression \eqref{equ2.2.1.2-leb} and the spectral assumption (\ref{eq1.4}). Next we prove (\ref{eq3.11.48}). Write
			\begin{equation*}
				{F}_{+}^{\alpha ,l}({\lambda }^{2m})-{F}_{-}^{\alpha ,l}({\lambda }^{2m})=\frac{\alpha \chi (\lambda )({F}^{-}({\lambda }^{2m})-{F}^{+}({\lambda }^{2m}))}{(1+\alpha {F}^{+}({\lambda }^{2m}))(1+\alpha {F}^{-}({\lambda }^{2m}))},
			\end{equation*}
			it follows from \eqref{eq-R+-R-} that
			\begin{equation}\label{equ-n>2m-F+-}
				|\frac{{d}^{l}}{d{\lambda }^{l}}({F}^{+}({\lambda }^{2m})-{F}^{-}({\lambda }^{2m}))|\le C{\lambda }^{n-2m-k}.
			\end{equation}
			Combining this and the spectral assumption (\ref{eq1.4}), we prove (\ref{eq3.11.48}).
		\end{proof}
		
		When $n\le 2m$, the behavior of ${F}_{\pm}^{\alpha ,l }({\lambda }^{2m})$ is also related to the vanishing moment conditions of $\varphi$ (due to Lemma \ref{lm2.0}), where $\varphi$ is given in \eqref{eq1.1} and fulfills Assumption \ref{assumption-1}. More precisely, we present the following three lemmas.
		
		\begin{lemma}\label{lm3.4}
			Let $1\leq n<2m$ be odd. There exists a small constant $\lambda_0 \in (0,1)$ depending on $\varphi$ such that the following results hold:
			
			\noindent $(i)$ If there exists a  $\kappa_0 \in \N_0^{n}$ with $|\kappa_0|=\mathbf{k}_0\in [0, m-\frac{n+1}{2}]$ such that
			$\int_{{\mathbb R}^{n}}^{}{x}^{\kappa_0}\varphi(x)\d x\neq0$, but for all $0 \le |\beta|<\mathbf{k}_0$, 
			\begin{align}\label{eq4-vanish}
				\int_{{\mathbb R}^{n}}^{}{x}^{\beta}\varphi(x)dx=0, 
			\end{align}
			then for  $0\le l\le 1$,
			\begin{equation}\label{equ4-F1}
				\left | \frac{{d}^{l}}{d{\lambda }^{l}}\left({\lambda }^{n-2m+2\mathbf{k}_0}{F}_{\pm}^{\alpha ,l }({\lambda }^{2m})\right)\right |\lesssim \lambda ^{-l},\qquad 0<\lambda<\lambda_0.
			\end{equation}
			\noindent $(ii)$ If \eqref{eq4-vanish} holds for all $0 \le |\beta|\le m-\frac{n+1}{2}$,
			then for  $0\le l\le 1$,
			\begin{equation}\label{equ4-F2}
				\left | \frac{{d}^{l}}{d{\lambda }^{l}}{F}_{\pm}^{\alpha ,l }({\lambda }^{2m})\right |\lesssim \lambda ^{-l},\qquad 0<\lambda<\lambda_0.
			\end{equation}
		\end{lemma}
		
		\begin{proof}
			We first prove \eqref{equ4-F1}. By the definition of ${F}^{\pm }({\lambda }^{2m})$ and \eqref{eq.2-low-odd} in Lemma \ref{lm2.0}, we have 
			\begin{equation*}
				\begin{split}
					{F}^{\pm }({\lambda }^{2m})&=
					\langle {R}_{0}^{\pm }({\lambda }^{2m})\varphi ,\varphi \rangle\\
					&=\displaystyle\sum_{j=0}^{m-\frac{n+1}{2}}{a}_{j}^{\pm} {\lambda }^{n-2m+2j}\int_{\R^{2n}}{\left | x-y\right |^{2j}\varphi (y)\varphi (x)}\d y\d x+{b}_{0} \int_{\R^{2n}}{\left | x-y\right |^{2m-n}\varphi (y)\varphi (x)\d y\d x}\\
					&+\lambda^{n-2m}\int_{\R^{2n}}{ {r}_{2m-n+1}^{\pm }(\lambda\left | x-y\right |)\varphi (y)\varphi (x)\d y\d x}.
				\end{split}
			\end{equation*}
			Note that 
			\begin{equation}\label{eq-expansion-x-y}
				|x-y|^{2j}=\sum_{\substack{\eta,\gamma\in \N_0^{n}\\ |\eta|+|\gamma|\le 2j}}C_{\eta,\gamma}\, x^\eta y^\gamma, \quad j\in \mathbb{N}_0,
			\end{equation}
			where $C_{\eta,\gamma}$ are absolute constants.
			This, together with the assumption \eqref{eq4-vanish}, yields that  
			$$\int_{\R^{2n}}{\left | x-y\right |^{2j}\varphi (y)\varphi (x)}\d y\d x=0, \ \ \ 2j<2\k_0.$$ 
			Then we obtain that
			\begin{align}\label{n<2m-F+--expansion}
				{F}^{\pm }({\lambda }^{2m})
				&=\displaystyle\sum_{j=\k_0}^{m-\frac{n+1}{2}}{a}_{j}^{\pm} {\lambda }^{n-2m+2j}\int_{\R^{2n}}{|x-y|^{2j}\varphi (y)\varphi (x)}\d y\d x+{b}_{0} \int_{\R^{2n}}{\left | x-y\right |^{2m-n}\varphi (y)\varphi (x)\d y\d x} \nonumber\\	&\quad+\lambda^{n-2m}\int_{\R^{2n}}{ {r}_{2m-n+1}^{\pm }(\lambda\left | x-y\right |)\varphi (y)\varphi (x)\d y\d x}.
			\end{align}
			The estimate \eqref{eq3.5.0-rem} and our decay assumption \eqref{eq1.2} on $\varphi$ indicate that for $l=0,1$, 
			\begin{align}\label{est-for-r-F}
				\left| \frac{d^l}{d\lambda^l}\lambda^{n-2m}\int_{\R^{2n}}{ {r}_{2m-n+1}^{\pm }(\lambda\left | x-y\right |)\varphi (y)\varphi (x)\d y\d x}\right| 
				\lesssim & \lambda^{1-l}\left|\int_{\R^{2n}}  \left | x-y\right |^{2m-n+1}\varphi (y)\varphi (x)\d y\d x\right| \nonumber\\
				\lesssim& \lambda^{1-l}.
			\end{align}
			Then, for $l=0,1$, the following estimate holds:
			\begin{align}\label{est-F-1}
				\Big|\frac{d^l}{d\lambda^l} {F}^{\pm }({\lambda }^{2m})\Big|\lesssim \lambda ^{n-2m+2\k_0-l}, \qquad\,\, ~~0<\lambda<1.
			\end{align}
			Meanwhile,  note that by \eqref{eq-def-A} and \eqref{eq5.48}, we have $a_j^\pm\neq 0$ for $j=0,1,\cdots, m-\frac{n+1}{2}$.
			Furthermore, the integral   $\int_{\R^{2n}}{|x-y|^{2\k_0}\varphi (y)\varphi (x)}\d y\d x\neq 0$ by the assumption \eqref{eq4-vanish}.
			Consequently, there exists a small $\lambda_0<1$ depending on $m,\, n, \, V, \, \alpha$ such that 
			\begin{equation}\label{eq-4.14}
				\Big|1+\alpha{F}^{\pm }({\lambda }^{2m}) \Big|\ge \frac{1}{2}{\lambda }^{n-2m+2\k_0},\qquad\, 0<\lambda<\lambda_0,
			\end{equation}
			indicating a gain of ${\lambda }^{n-2m+2\k_0}$ compared with the spectral assumption \eqref{eq1.4}.
			The above inequality, together with \eqref{equ2.2.1.2-leb}, \eqref{est-F-1} and the Leibniz rule, yields that for $l=0,1$, 
			\begin{equation*}
				\begin{aligned}
					\left|\frac{d^l}{d\lambda^l} {F}_{\pm}^{\alpha ,l }({\lambda }^{2m}) \right|  
					\lesssim \lambda ^{2m-n-2\k_0-l},\qquad 0<\lambda<\lambda_0.      
				\end{aligned}
			\end{equation*}
			Hence, we obtain \eqref{equ4-F1}.
			
			Next, we shall establish \eqref{equ4-F2}. Now \eqref{n<2m-F+--expansion} becomes 
			\begin{equation*}
				{F}^{\pm }({\lambda }^{2m})
				={b}_{0} \int_{\R^{2n}}{\left | x-y\right |^{2m-n}\varphi (y)\varphi (x)\d y\d x} +\int_{\R^{2n}}{ {r}_{2m-n+1}^{\pm }(\lambda ,\left | x-y\right |)\varphi (y)\varphi (x)\d y\d x}.    
			\end{equation*}
			This and  \eqref{est-for-r-F} imply that for $l=0,1$,  
			\begin{align}\label{est-F-2}
				\left|\frac{d^l}{d\lambda^l} {F}^{\pm }({\lambda }^{2m})\right|\lesssim \lambda ^{-l}, \qquad \,\,~~0<\lambda<1.
			\end{align}
			By the Plancherel theorem, the expression \eqref{eq.2-low-odd} and the fact that $\varphi\ne 0$, we derive that
			\begin{equation}\label{eq-equal-laplacian-m}
				{b}_{0} \int_{\R^{2n}}{\left | x-y\right |^{2m-n}\varphi (y)\varphi (x)\d y\d x}= \langle (-\Delta)^{-m}\varphi,\, \varphi\rangle
				=\int_{\R^{n}} |\xi|^{-2m}|\hat{\varphi} (\xi)|^2 \d \xi >0.
			\end{equation}
			This, along with \eqref{est-for-r-F}, indicates that there exists some small $0<\lambda_0<1$ such that $1+\alpha{F}^{\pm }({\lambda }^{2m})\ge 1/2$ holds for all $0<\lambda<\lambda_0$. Therefore, following the same argument to the proof of \eqref{equ4-F1},  we obtain  \eqref{equ4-F2}.
		\end{proof}
		
		\begin{lemma}\label{lm3.6}
			Let $1\leq n<2m$ be even. There exists some small constant $\lambda_0 \in (0,1)$ 
			such that the following results hold:
			
			\noindent (i) 
			If there exists a $\kappa_0 \in \N_0^{n}$ with $|\kappa_0|=\mathbf{k}_0\in [0,m-\frac{n}{2}-1]$ such that
			$\int_{{\mathbb R}^{n}}^{}{x}^{\kappa_0}\varphi(x)\d x\neq0$, but \eqref{eq4-vanish} holds for all $0 \le |\beta|<\mathbf{k}_0$. Then for $0\le l\le 1$,
			\begin{equation}\label{eq3.6.1}
				\left | \frac{{d}^{l}}{d{\lambda }^{l}}\left({\lambda }^{n-2m+2\mathbf{k}_0}{F}_{\pm}^{\alpha ,l }({\lambda }^{2m})\right)\right |\lesssim \lambda ^{-l},\qquad 0<\lambda<\lambda_0.
			\end{equation}
			
			\noindent (ii) If there exists a $\kappa_0 \in \N_0^{n}$ with $|\kappa_0|=m-\frac{n}{2}$  such that
			$\int_{{\mathbb R}^{n}}^{}{x}^{\kappa_0}\varphi(x)\d x\neq0$, but \eqref{eq4-vanish} holds for all $0 \le |\beta|<\mathbf{k}_0$. Then for $0\le l\le 1$,
			\begin{equation}\label{eq3.6.2}
				\left | \frac{{d}^{l}}{d{\lambda }^{l}}\left(\log\lambda\ {F}_{\pm }^{\alpha ,l}({\lambda }^{2m})\right)\right |\lesssim \lambda ^{-l},\qquad 0<\lambda<\lambda_0.
			\end{equation}
			
			\noindent (iii) If \eqref{eq4-vanish} holds for all $0 \le |\beta|\le m-\frac{n}{2}$, then for $0\le l\le 1$,
			\begin{equation}\label{eq3.6.3}
				\left | \frac{{d}^{l}}{d{\lambda }^{l}}{F}_{\pm}^{\alpha ,l }({\lambda }^{2m})\right |\lesssim \lambda ^{-l},\qquad 0<\lambda<\lambda_0.
			\end{equation}
		\end{lemma}
		
		\begin{proof}
			We begin by pointing out that if there exists a $\kappa_0 \in \N_0^{n}$ such that \eqref{eq4-vanish} holds for all $0 \le |\beta|<\mathbf{k}_0$,
			then by  the expansion \eqref{eq.2-low-even} and the identity \eqref{eq-expansion-x-y}, we have 
			\begin{equation}\label{equ4-R_0-even}
				\begin{split}
					{F}^{\pm }({\lambda }^{2m})&=
					\langle {R}_{0}^{\pm }({\lambda }^{2m})\varphi ,\varphi \rangle
					=\sum_{j=\mathbf{k}_0}^{m-\frac{n}{2}}{a}_{j}^{\pm} {\lambda }^{n-2m+2j}\int_{\R^{2n}}{\left | x-y\right |^{2j}\varphi (y)\varphi (x)}\d y\d x\\
					&+{b}_{0} \int_{\R^{2n}}{{\left | x-y\right |}^{2m-n}\log(\lambda \left | x-y\right |)\varphi (y)\varphi (x)}\d y\d x\\
					&+\lambda^{n-2m}\int_{\R^{2n}}{{r}_{2m-n+1}^{\pm }(\lambda ,\left | x-y\right |)\varphi (y)\varphi (x)}\d y\d x,
				\end{split}
			\end{equation}
			moreover,  for $0 \le l\le 1$,  the remainder term satisfies
			\begin{equation}\label{eq-est-r-even}
				\left| \frac{d^l}{d\lambda^l}\left(\lambda^{n-2m}\int_{\R^{2n}}{ {r}_{2m-n+1}^{\pm }(\lambda\left | x-y\right |)\varphi (y)\varphi (x)\d y\d x}\right)\right| 
				\lesssim \lambda^{1-l},\ 0<\lambda<1.
			\end{equation}
			
			We first prove \eqref{eq3.6.1}. If $\mathbf{k}_0\in [0,m-\frac{n}{2}-1]$,  thus we have $n-2m+2\mathbf{k}_0<0$. This implies that  $|\log\lambda|\le {\lambda }^{n-2m+2\mathbf{k}_0} $ and $\lambda \le {\lambda }^{n-2m+2\mathbf{k}_0}$  when $0<\lambda<1$. Hence we have
			\begin{equation}\label{equ-n<2m-even-F+-1}
				\left | \frac{{d}^{l}}{d{\lambda }^{l}}{F}^{\pm }({\lambda }^{2m})\right |\lesssim {\lambda }^{n-2m+2\mathbf{k}_0-l},\quad \,\, 0<\lambda<1.
			\end{equation}
			Meanwhile, similar to \eqref{eq-4.14}, there also exists a small $\lambda_0<1$ depending on $m,\, n, \, V, \, \alpha$ such that 
			$$
			\Big|1+\alpha{F}^{\pm }({\lambda }^{2m}) \Big|\ge \frac{{\lambda }^{n-2m+2\mathbf{k}_0}}{2},\quad \,\,0<\lambda<\lambda_0.
			$$
			This, together with \eqref{equ2.2.1.2-leb} and \eqref{equ-n<2m-even-F+-1}, yields \eqref{eq3.6.1}.
			
			Next we prove \eqref{eq3.6.2}, if $\k_0=m-\frac{n}{2}$, thus $n-2m+2\k_0=0$. Then by \eqref{equ4-R_0-even},  there  exists a small $\lambda_0<1$ depending on $m,\, n, \, V, \, \alpha$ such that 
			$$
			\Big|1+\alpha{F}^{\pm }({\lambda }^{2m}) \Big|\ge \frac{|\log \lambda|}{2},\ 0<\lambda<\lambda_0.
			$$
			and
			\begin{equation*}\label{eq-kappa=m-n/2}
				\left | \frac{d}{d \lambda}F^{\pm }(\lambda^{2m})\right |\lesssim \lambda^{-1},\ \  0<\lambda<{\lambda }_{0}<1.
			\end{equation*}
			Combining these two facts with \eqref{equ2.2.1.2-leb}, we obtain \eqref{eq3.6.2}. 
			
			Finally, if \eqref{eq4-vanish} holds for all $0 \le |\beta|\le m-\frac{n}{2}$, then the first term in the right hand side of \eqref{equ4-R_0-even} vanishes. Furthermore, we have
			$${b}_{0} \log\lambda\int_{\R^{2n}}{|x-y|^{2m-n}\varphi (y)\varphi (x)\d y\d x}=0.$$
			Similar to \eqref{eq-equal-laplacian-m}, we deduce that
			$${b}_{0} \int_{\R^{2n}}{|x-y|^{2m-n}\log |x-y|\varphi (y)\varphi (x)\d y\d x}>0,$$
			thus by \eqref{eq-est-r-even}, there exists a fixed small $\lambda_0<1$ such that
			$$\left|1+\alpha{F}^{\pm }({\lambda }^{2m})\right|\ge 1/2,\quad\, 0<\lambda<\lambda_0,$$
			and for $0\le l\le 1$, we have
			\begin{equation*}
				\left | \frac{d^l}{d \lambda^l}F^{\pm }(\lambda^{2m})\right |\lesssim \lambda^{-l},\quad\, \,\,0<\lambda<\lambda_0.
			\end{equation*}
			Combining these with \eqref{equ2.2.1.2-leb}, we obtain \eqref{eq3.6.3}. Therefore the proof is complete.
		\end{proof}
		
		The following lemma concerns the properties of  ${F}_{+}^{\alpha ,l}({\lambda }^{2m})-{F}_{-}^{\alpha ,l}( {\lambda }^{2m})$ and ${F}_{-}^{\alpha ,l}( {\lambda }^{2m})$ when $n=2m$. The proof of this lemma is similar to Lemma \ref{lm3.6}, so we omit the detailed proof for brevity.
		
		\begin{lemma}\label{lm3.7}
			Let  $n=2m$.  There exists some small constant $\lambda_0 \in (0,1)$ 
			such that the following results hold:
			
			(i) If $\int_{}^{} \varphi \left ( x\right )\d x=0$, then for $0\le l\le 2$, we have
			\begin{equation}\label{eq-n=2m-F+_F-}
				\left | \frac{d^l}{d\lambda^l}\left [ {F}_{+}^{\alpha ,l}({\lambda }^{2m})-{F}_{-}^{\alpha ,l}({\lambda }^{2m})\right ]\right|\lesssim\lambda ^{-l},\quad 0<\lambda<\lambda_0,
			\end{equation}
			\begin{equation}\label{eq-n=2m-F-}
				\left | \frac{d^l}{d\lambda^l}{F}_{-}^{\alpha ,l}( {\lambda }^{2m})\right|\lesssim\lambda ^{-l},\qquad\qquad\qquad\, \,\,0<\lambda<\lambda_0.
			\end{equation}
			
			(ii) If $\int_{}^{} \varphi \left ( x\right )\d x\ne 0$, then for $0\le l\le 2$, we have
			\begin{equation}\label{eq-n=2m-F+F-_log}
				\left | \frac{d^l}{d\lambda^l}\left [{\left ( \log_{}{\lambda }\right )}^{2} ({F}_{+}^{\alpha ,l}({\lambda }^{2m})-{F}_{-}^{\alpha ,l}({\lambda }^{2m}))\right ]\right|\lesssim\lambda ^{-l},\quad 0<\lambda<\lambda_0,
			\end{equation}
			\begin{equation}\label{eq-n=2m-F-log}
				\left | \frac{d^l}{d\lambda^l}(\left ( \log_{}{\lambda } \right ){F}_{-}^{\alpha ,l}( {\lambda }^{2m}))\right|\lesssim\lambda ^{-l},\qquad\qquad\qquad\,\,\,0<\lambda<\lambda_0.
			\end{equation}
		\end{lemma}
		
		
		
		\subsection{The proof of Theorem \ref{thm3.2} for $n>2m$}\label{sec4.1}\
		
		We begin by rewriting $R_0^{\pm}(\lambda^{2m})\varphi$ in an appropriate form that allows us to isolate the oscillatory component. More precisely, we have 
		
		\begin{proposition}\label{lm-res-low-nl2m-1}
			Suppose that the decay assumption \eqref{eq2} is satisfied for $\varphi$. Then, for $0<\lambda<1$ and  $0\le l\le \left [\frac{n}{2m}\right ]+1$, the following statements hold:\\
			\noindent (i) When  $2m<n<4m$, we have
			\begin{equation}\label{eq-res-low-nl2m-1-0}
				\int_{\mathbb{R}^n } R_0^{\pm}(\lambda^{2m})(x-y) \varphi(y)\d y=e^{ \pm \i\lambda|x|}W_{1,1}^\pm(\lambda, x),
			\end{equation}
			and
			\begin{equation} \label{eq-res-low-nl2m-1}
				\left|  \partial_\lambda ^lW_{1,1}^\pm(\lambda, x) \right|\lesssim \min\{\lambda^{-l}, \  \lambda^{\frac{n+1}{2}-2m-l}\langle x \rangle^{-\frac{n-1}{2}}\}.
			\end{equation}
			
			\noindent (ii) When  $n\ge 4m$, we have
			\begin{equation}\label{eq-res-low-nl2m-2-0}
				\int_{\mathbb{R}^n } R_0^{\pm}(\lambda^{2m})(x-y) \varphi(y)\d y=e^{ \pm \i\lambda|x|}W_{0,0}^\pm(\lambda, x)+e^{ \pm \i\lambda|x|}W_{0,1}^\pm(\lambda, x),
			\end{equation}
			and 
			\begin{equation} \label{eq-res-low-nl2m-2}
				\left| \partial_\lambda ^l W_{0,0}^\pm(\lambda, x) \right|\lesssim \lambda^{-l}\langle x \rangle^{-(n-2m)},
			\end{equation}
			\begin{equation} \label{eq-res-low-nl2m-3}
				\left| \partial_\lambda ^l W_{0,1}^\pm(\lambda, x)\right|\lesssim \lambda^{\frac{n+1}{2}-2m-l}\langle x \rangle^{-\frac{n-1}{2}}.
			\end{equation}
			(iii) Further, for  $n>2m$, we have
			\begin{equation}\label{eq-res-low-nl2m-4-0}
				\int_{\mathbb{R}^n } \left[R_0^{+}(\lambda^{2m})(x-y)-R_0^{-}(\lambda^{2m})(x-y)\right]\varphi(y)\d y =e^{\i\lambda|x|} U_{0,0}^+(\lambda,x)-e^{-\i\lambda|x|} U_{0,0}^-(\lambda,x),  
			\end{equation}
			and
			\begin{equation} \label{eq-res-low-nl2m-4}
				\sup_x\left| \partial_\lambda ^l \widetilde{U}_0^\pm (\lambda,x) \right|\lesssim \lambda^{n-2m-l}. 
			\end{equation}
		\end{proposition}
		\begin{proof}
			We first prove statement (i). 
			It follows from \eqref{eq-decom-r-nl2m} that
			\begin{equation*}\label{eq-int-for-res-10}
				\begin{split}
					&\int_{\mathbb{R}^n } R_0^{\pm}(\lambda^{2m})(x-y) \varphi(y)\d y \\
					=&\int_{\mathbb{R}^n } e^{\pm \i \lambda |x-y|}\left(\frac{1}{|x-y|^{n-2m}}W_{0}^{\pm}(\lambda |x-y|)+\frac{\lambda^{\frac{n+1}{2}-2m}}{|x-y|^{\frac{n-1}{2}}}W_{1}^{\pm}(\lambda |x-y|)\right) \varphi(y)\d y \\  
					=&e^{\pm \i \lambda |x|}\int_{\mathbb{R}^n } e^{\pm \i \lambda (|x-y|-|x|)}\left(\frac{1}{|x-y|^{n-2m}}W_{0}^{\pm}(\lambda |x-y|)+\frac{\lambda^{\frac{n+1}{2}-2m}}{|x-y|^{\frac{n-1}{2}}}W_{1}^{\pm}(\lambda |x-y|)\right) \varphi(y)\d y \\
					:=&e^{\pm \i \lambda |x|} W_{1,1}^\pm(\lambda, x).
				\end{split}
			\end{equation*}
			Note that $\text{\supp}W_{0}^{\pm}(\cdot) \subset[0, 1]$, $\text{\supp}W_{1}^{\pm}(\cdot) \subset[\frac12, \infty)$ and by \eqref {eq-expansion-high-6}, we have 
			\begin{equation*}
				\begin{split}
					\left| \partial_\lambda^l \left(\frac{1}{|x-y|^{n-2m}}W_{0}^{\pm}(\lambda |x-y|)+\frac{\lambda^{\frac{n+1}{2}-2m}}{|x-y|^{\frac{n-1}{2}}}W_{1}^{\pm}(\lambda |x-y|)) \right) \right| 
					\lesssim_l \lambda^{-l}\min\left\{|x-y|^{2m-n}, \,  \frac{\lambda^{\frac{n+1}{2}-2m}}{|x-y|^{\frac{n-1}{2}}}\right\},  
				\end{split}
			\end{equation*}
			where $l\in \N_0$, and  $2m<n<4m$.
			Meanwhile, a direct computation shows that
			\begin{equation*}
				\left| \frac{d^l}{d\lambda^l} \left(e^{\pm \i \lambda (|x-y|-|x|)} \right) \right| \lesssim_l   |y|^{l}.
			\end{equation*}
			Thus, by \eqref{eq2.20},  we deduce that \eqref{eq-res-low-nl2m-1} holds when $0\le l\le \left [\frac{n}{2m}\right ]+1\le 2$ and $2m<n<4m$. By Lemma \ref{lem-expasion-R+-R-1} and Lemma \ref{lm-expansion-high}, the results in  (ii) and (iii) are derived through analogous arguments.  This completes the proof.
		\end{proof}

		
		\begin{remark}\label{rmk-low-sym}
			We derive the pointwise decay   \eqref{eq3.22.11}  by examining
			two distinct regions:  $t^{-\frac{1}{2m}}(|x|+|y|)>1$ and $t^{-\frac{1}{2m}}(|x|+|y|)\leq 1$. Moreover, as highlighted in Remark \ref{sec2.1}, we may, without loss of generality, assume that $|x|\ge |y|$ during the proof.
		\end{remark}
		
		
		\subsubsection{In the region where $t^{-\frac{1}{2m}}(|x|+|y|)>1$ .}
		
		In light of  \eqref{eq-expansion-high-1} and \eqref{eq-decom-r-nl2m}, we further divide the proof into the following subcases. 
		
		$\bullet$ {\emph{Subcase 1 : $2m<n\le 4m-1$.}}
		
		Using \eqref{equ-2-low} and \eqref{eq-res-low-nl2m-1-0}, $K_n^{\pm,l}(t,x,y)$ can be written as
		\begin{equation*}
			\begin{split}
				K_n^{\pm,l}(t,x,y)=&-\frac{m\alpha }{\pi {\rm i}}\int_{0}^{+\infty}{e}^{-{\rm i}t{\lambda }^{2m}\pm{\rm i}\lambda (|x|+|y|)}{F}_{\pm}^{\alpha ,l}(\lambda^{2m})W_{1,1}^\pm(\lambda, x)W_{1,1}^\pm(\lambda, y)\lambda^{2m-1} {\rm d}\lambda.
			\end{split}
		\end{equation*}
		By \eqref{equ4-F1} and \eqref{eq-res-low-nl2m-1}, we deduce that for $ 0\le l\le \left [ \frac{n}{2m}\right]+1,$
		\begin{align}\label{eq-symbol-n-1/2}
			\left | \partial_\lambda^l \left(\lambda^{2m-1} {F}_{\pm}^{\alpha ,l}(\lambda^{2m})W_{1,1}^\pm(\lambda, x)W_{1,1}^\pm(\lambda, y)\right)\right |\lesssim \lambda^{\frac{n-1}{2}-l}\langle x\rangle^{-\frac{n-1}{2}}, \ \ \ 0 \le \lambda \le 1.
		\end{align}
		We apply \eqref{eq-osc-g-1} with $b=\frac{n-1}{2}$ to obtain
		\begin{equation*}
			\begin{split}
				|K_n^{\pm,l}(t,x,y)|&\lesssim t^{-\frac{n+1}{4m}} ({t}^{-\frac{1}{2m}}\left ||x|+|y|\right |)^{-\frac{m-1-\frac{n-1}{2}}{2m-1}}\langle x\rangle ^{-\frac{n-1}{2}}\lesssim t^{-\frac{n}{2m}}(1+t^{-\frac{1}{2m}}|x|)^{-\frac{n(m-1)}{2m-1}}\\
				&\lesssim {t}^{-\frac{n}{2m}}(1+{t}^{-\frac{1}{2m}}|x-y|)^{-\frac{n(m-1)}{2m-1}},
			\end{split}
		\end{equation*}
		where in the second inequality above, we used the identity
		$$\frac{n(m-1)}{2m-1}=\frac{n-1}{2}+\frac{m-1-\frac{n-1}{2}}{2m-1}$$ 
		and the assumption $t^{-\frac{1}{2m}}|x|\ge \frac12$; in the last inequality, we used the fact  $|x-y|\lesssim  2|x|$ since we have assumed that $|x|\ge |y|$.

		$\bullet$ {\emph{Subcase 2 : $n\ge 4m$.}}
		
		By \eqref{equ-2-low} and \eqref{eq-res-low-nl2m-2-0}, we have
		\begin{equation*}
			\left | K_n^{\pm,l}(t,x,y) \right | \lesssim \sum_{s,j=0}^{1} \left | K_{n,s,j}^{\pm,l}(t,x,y) \right |,
		\end{equation*}
		where
		\begin{equation*}
			\begin{split}
				K_{n,s,j}^{\pm,l}(t,x,y)=-\frac{m\alpha }{\pi {\rm i}}\int_{0}^{+\infty}{e}^{-{\rm i}t{\lambda }^{2m}\pm{\rm i}\lambda (|x|+|y|)}{F}_{\pm}^{\alpha ,l}(\lambda^{2m})W_{0,s}^\pm(\lambda, x)W_{0,j}^\pm(\lambda, y)\lambda^{2m-1} {\rm d}\lambda.
			\end{split}
		\end{equation*}
		Observe that $\frac{n+1}{2}-2m\ge 0$ when $n\ge 4m$. Thanks to \eqref{equ4-F1}, \eqref{eq-res-low-nl2m-2} and \eqref{eq-res-low-nl2m-3}, when $s=1$, ${F}_{\pm}^{\alpha ,l}(\lambda^{2m})W_{0,s}^\pm(\lambda, x)W_{0,j}^\pm(\lambda, y)\lambda^{2m-1}$ also satisfies the estimate in \eqref{eq-symbol-n-1/2}. Thus
		$K_{n,s,j}^{\pm,l}(t,x,y)$ with $s=1$ satisfies \eqref{eq3.22.11} by the same arguments used in Subcase 1. We omit the details here.
		
		Next, it remains to consider $K_{n,0,j}^{\pm,l}(t,x,y)$ ($j=0,1$). Similarly, using \eqref{equ4-F1},  \eqref{eq-res-low-nl2m-2} and \eqref{eq-res-low-nl2m-3}, we have for $0\le l\le \left [ \frac{n}{2m}\right]+1$,
		\begin{align*}
			\left | \partial_\lambda^l\left(\lambda^{2m-1} {F}_{\pm}^{\alpha ,l}(\lambda^{2m})W_{0,0}^\pm(\lambda, x)W_{0,j}^\pm(\lambda, y)\right)\right |\lesssim \lambda^{2m-1-l}\langle x\rangle^{-(n-2m)}, \ \  0 \le \lambda \le 1.
		\end{align*}
		It follows from \eqref{eq-osc-g-1} with $b=2m-1$ that
		\begin{align*}
			\left | K_{n,0,j}^{\pm,l}(t,x,y)\right|\lesssim t^{-1} ({t}^{-\frac{1}{2m}}\left ||x|+|y|\right |)^{\frac{m}{2m-1}}\langle x\rangle ^{-(n-2m)}\lesssim t^{-\frac{n}{2m}}(1+t^{-\frac{1}{2m}}|x-y|)^{-\frac{n(m-1)}{2m-1}},
		\end{align*}
		where we have used the fact that $n-2m-\frac{m}{2m-1}\ge {\frac{n(m-1)}{2m-1}}$.
		
		\subsubsection{In the region where $t^{-\frac{1}{2m}}(|x|+|y|)\le 1$.}
		
		In this case, we shall take advantage of the difference $R_0^{+}(\lambda^{2m})-R_0^{-}(\lambda^{2m})$,  which provides a gain of $\lambda^{n-2m}$. Specifically, by the following algebraic identity
		\begin{equation*}\label{eq.17.10}
			\prod_{k=1}^3{A_k^+}-\prod_{k=1}^3{A_k^-}=(A_1^+-A_1^-)A_2^+A_3^++A_1^-A_2^-(A_3^+-A_3^-)+A_1^-(A_2^+-A_2^-)A_3^+,
		\end{equation*}
		we break $K_n^{+, l}-K_n^{-, l}$ into
		\begin{align}\label{eq.l7}
			K_n^{+, l}-K_n^{-, l}&=
			-\frac{m\alpha }{\pi {\rm i}}\int_{0}^{+\infty }{e}^{-{\rm i}t{\lambda }^{2m}}({F}_{+}^{\alpha ,l}({\lambda }^{2m})-{F}_{-}^{\alpha ,l}({\lambda }^{2m})){R}_{0}^{+}({\lambda }^{2m})\varphi\left\langle{R}_{0}^{+}({\lambda }^{2m})\cdotp ,\ \varphi  \right\rangle{\lambda }^{2m-1}\rm d \lambda\nonumber\\
			&-\frac{m\alpha }{\pi {\rm i}}\int_{0}^{+\infty }{e}^{-{\rm i}t{\lambda }^{2m}}{F}_{-}^{\alpha ,l}({\lambda }^{2m}){R}_{0}^{-}({\lambda }^{2m})\varphi\left\langle({R}_{0}^{+}({\lambda }^{2m})-{R}_{0}^{-}({\lambda }^{2m}))\cdotp ,\ \varphi \right\rangle{\lambda }^{2m-1}\rm d \lambda\nonumber\\
			&-\frac{m\alpha }{\pi {\rm i}}\int_{0}^{+\infty }{e}^{-{\rm i}t{\lambda }^{2m}}{F}_{-}^{\alpha ,l}({\lambda }^{2m})
			({R}_{0}^{+}({\lambda }^{2m})-{R}_{0}^{-}({\lambda }^{2m}))\varphi \left\langle{R}_{0}^{+}({\lambda }^{2m})\cdotp ,\ \varphi \right\rangle{\lambda }^{2m-1}\rm d \lambda\nonumber\\
			&:=-\frac{m\alpha }{\pi {\rm i}}\left(K_{n, 1}^{ l}+K_{n, 2}^{ l}+K_{n, 3}^{ l}\right).
		\end{align}
		We remark that  $K_{n, 2}^{l}(t,x,y)$ and $K_{n,3}^{l}(t,x,y)$ can be estimated by the same way, thus we only consider $K_{n, 1}^{l}(t,x,y)$ and $K_{n, 2}^{l}(t,x,y)$ here.
		
		{\bf Step 1: Estimate for $K_{n,1}^{l}(t,x,y)$.}
		
		$\bullet$ {\emph{Subcase 1 : $2m<n\le 4m-1$.}}
		
		Using \eqref{eq-res-low-nl2m-1-0}, $K_{n,1}^{\pm,l}(t,x,y)$ can be written as (up to constants)
		\begin{equation*}
			\begin{split}
				K_{n,1}^{l}(t,x,y)=&\int_{0}^{+\infty}{e}^{-{\rm i}t{\lambda }^{2m}+{\rm i}\lambda (|x|+|y|)}({F}_{+}^{\alpha ,l}({\lambda }^{2m})-{F}_{-}^{\alpha ,l}({\lambda }^{2m}))W_{1,1}^+(\lambda, x)W_{1,1}^+(\lambda, y)\lambda^{2m-1} {\rm d}\lambda.
			\end{split}
		\end{equation*}
		It follows from the estimate \eqref{eq3.11.48} and \eqref{eq-res-low-nl2m-1} that for $0\le l\le \left [ \frac{n}{2m}\right]+1$ and $0 \le \lambda \le 1$, we have
		\begin{align}\label{eq-symbol-n-1}
			\left | \partial_\lambda^l \left(\lambda^{2m-1} ({F}_{+}^{\alpha ,l}({\lambda }^{2m})-{F}_{-}^{\alpha ,l}({\lambda }^{2m}))W_{1,1}^+(\lambda, x)W_{1,1}^+(\lambda, y)\right)\right |\lesssim \lambda^{n-1-l}.
		\end{align}
		Note that $t^{-\frac{1}{2m}}|x-y|\le 1$ also holds, then it follows from \eqref{eq-osc-l-1} with $b=n-1$ that when $2m<n\le 4m-1$,
		\begin{align*}
			\left | K_{n,1}^{l}(t,x,y)\right|\lesssim t^{-\frac{n}{2m}} \lesssim t^{-\frac{n}{2m}}(1+t^{-\frac{1}{2m}}|x-y|)^{-\frac{n(m-1)}{2m-1}}.
		\end{align*}
		
		$\bullet$ {\emph{Subcase 2 : $n\ge 4m$.}}
		\begin{equation*}
			\left | K_{n,1}^{l}(t,x,y) \right | \lesssim \sum_{s,j=0}^{1} \left | K_{n,1,s,j}^{l}(t,x,y) \right |,
		\end{equation*}
		where
		\begin{equation*}
			\begin{split}
				K_{n,s,j}^{l}(t,x,y)=&\int_{0}^{+\infty}{e}^{-{\rm i}t{\lambda }^{2m}+{\rm i}\lambda (|x|+|y|)} ({F}_{+}^{\alpha ,l}({\lambda }^{2m})-{F}_{-}^{\alpha ,l}({\lambda }^{2m}))W_{0,s}^+(\lambda, x)W_{0,j}^+(\lambda, y)\lambda^{2m-1} {\rm d}\lambda.
			\end{split}
		\end{equation*}
		\eqref{eq-res-low-nl2m-2} and \eqref{eq-res-low-nl2m-3} indicate that $\lambda^{2m-1}({F}_{+}^{\alpha ,l}({\lambda }^{2m})-{F}_{-}^{\alpha ,l}({\lambda }^{2m}))W_{0,s}^+(\lambda, x)W_{0,j}^+(\lambda, y)$ still satisfies \eqref{eq-symbol-n-1} for all $s,j\in \{0,1\}$. Thus, we have that in the dimensions $n\ge 4m$,
		\begin{equation*}
			\left | K_{n,1}^{l}(t,x,y)\right|\lesssim t^{-\frac{n}{2m}}(1+t^{-\frac{1}{2m}}|x-y|)^{-\frac{n(m-1)}{2m-1}}.   
		\end{equation*}
		
		{\bf Step 2: Estimate for $K_{n,2}^{l}(t,x,y)$.}
		
		$\bullet$ {\emph{Subcase 1 : $2m<n\le 4m-1$.}}
		
		Using \eqref{eq-res-low-nl2m-1-0} and \eqref{eq-res-low-nl2m-4-0}, $ K_{n,2}^{l}(t,x,y)$ can be written as a linear combination of  
		\begin{equation}\label{eq-kn2-2m}
			K_{n,2}^{\pm,l}(t,x,y)=\int_{0}^{+\infty}{e}^{-{\rm i}t{\lambda }^{2m}-{\rm i}\lambda (|x|\mp|y|)} {F}_{-}^{\alpha ,l}({\lambda }^{2m})W_{1,1}^-(\lambda, x)U_{0,0}^\pm(\lambda, y)\lambda^{2m-1} {\rm d}\lambda. 
		\end{equation}
		
		By \eqref{eq-res-low-nl2m-1} and \eqref{eq-res-low-nl2m-4}, we have that $\lambda^{2m-1}{F}_{-}^{\alpha ,l}({\lambda }^{2m})W_{1,1}^-(\lambda, x)U_{0,0}^\pm(\lambda, y)$ satisfies \eqref{eq-symbol-n-1}. Note that $t^{-\frac{n}{2m}}||x|\mp|y||\le 1$, then using \eqref{eq-osc-l-1} in Lemma \ref{lm2.4}, it follows that  
		\begin{equation*}
			\left | K_{n,2}^{l}(t,x,y)\right|\lesssim t^{-\frac{n}{2m}}\lesssim t^{-\frac{n}{2m}}(1+t^{-\frac{1}{2m}}|x-y|)^{-\frac{n(m-1)}{2m-1}}.   
		\end{equation*}
		
		$\bullet$ {\emph{Subcase 2 : $n\ge 4m$.}}
		
		In this subcase, it suffices to replace $W_{1,1}^-(\lambda, x)$ by $W_{0,s}^-(\lambda, x)$ for $s=0,1$ in \eqref{eq-kn2-2m}. By  \eqref{eq-res-low-nl2m-2} and \eqref{eq-res-low-nl2m-3}, it follows that $\lambda^{2m-1}{F}_{-}^{\alpha ,l}({\lambda }^{2m})W_{0,s}^-(\lambda. x)U_{0,0}^\pm(\lambda, y)$ still satisfies the estimate \eqref{eq-symbol-n-1}, thus we have the desired bound.
		
		Therefore the proof for $n\ge 2m$ is finished.
		\qed
		
		\subsection{The proof of Theorem \ref{thm3.2} for $1\leq n<2m$}\label{sec4.3}\
		
		Compared to the case $n>2m$, the situation for $n<2m$ is much more complicated due to the expansion of the free  resolvent (see Lemma \ref{lm2.0}).
		To proceed, we need the following lemma which deals with functions exhibiting specific vanishing moments. The proof can be found in \cite[Lemma 4.2]{CHHZ}, we  have included  a proof in Appendix \ref{app-3} for the sake of completeness.
		
		\begin{lemma}\label{lm-4.2-CHHZ}
			Assume that $\varphi (x) \in {L}_{\sigma }^{2}(\mathbb{R}^n)$,  where $\sigma >\frac{n}{2}+\mathbf{k}_0$, $\mathbf{k}_0\in\N^{+}$ and
			\begin{equation*}
				\langle x^{\beta},\varphi(x)\rangle = 0,\,\,\,\mbox{for all}\,\, \,0\le |\beta| \le \mathbf{k}_0-1.	\end{equation*} 
			Then, we have
			\begin{equation}\label{eqA.1.1}
				\left | \int_{\R^n} |x-y|^r\varphi (y)\d y \right |\lesssim  \| \varphi \|_{{L}_{\sigma }^{2}(\mathbb{R}^n)}\left \langle x\right \rangle^{r-\mathbf{k}_0},\ \ \ \ \ x \in \mathbb{R}^n,
			\end{equation}
			provided $r \in \mathbb{Z}$ and $-\frac{n-1}{2} \le r \le \mathbf{k}_0-1$.
		\end{lemma}
		
		By invoking Lemma \ref{lm-4.2-CHHZ},  we establish the following proposition concerning the property of $R_0^{\pm}(\lambda^{2m}) \varphi$, which depends on the order of the vanishing  moment for $\varphi$ and  plays a key role in the proof. 
		
		\begin{proposition}\label{lm-n<2m-odd-vanish}
			Let $n<2m$ and let $\varphi$ satisfy Assumption \ref{assumption-1}. If  there exists $\mathbf{k}_0 \in \N_0$, $0\le \mathbf{k}_0 \le  m+1-[\frac{n+1}{2}]$ such that
			$$\int_{{\mathbb R}^{n}}^{}{x}^{\beta}\varphi (x)dx=0,\quad \text{for all} \,\,\, \ 0 \le |\beta|<\mathbf{k}_0.$$
			Then we have the following expression
			\begin{equation}\label{eq3.33.0}
				\int_{\mathbb{R}^n } R_0^{\pm}(\lambda^{2m})(x-y) \varphi(y)\d y=e^{ \pm \i\lambda|x|}W_{1,2}^\pm(\lambda, x),\,\,\, 0<\lambda<1.
			\end{equation}
			Moreover,  for $l=0,\, 1$, $W_{1,2}^\pm$ satisfies 
			\begin{equation}\label{eq3.33.1}
				\left| \partial_\lambda ^l W_{1,2}^\pm(\lambda, x) \right|\lesssim \min\left\{\lambda^{n-2m+\mathbf{k}_0-l},\, \lambda^{\frac{n+1}{2}-2m+\mathbf{k}_0-l}\langle x \rangle^{-\frac{n-1}{2}}\right\},\,\,\, 0<\lambda<1,
			\end{equation}
			except for  $n=2m-2$ and $\mathbf{k}_0=2$. While in this exceptional case, $W_{1,2}^\pm$ satisfies 
			\begin{equation}\label{eq3.33.1-1129}
				\left| \partial_\lambda ^l W_{1,2}^\pm(\lambda, x) \right|\lesssim \min\left\{\lambda^{-l}|\log{\lambda}|,\, \lambda^{-\frac{n-1}{2}-l}\langle x \rangle^{-\frac{n-1}{2}}\right\},\,\,\, 0<\lambda<1.
			\end{equation}
		\end{proposition}
		\begin{proof}
			First,  \eqref{eq3.33.0} follows if we set
			\begin{equation*}
				W_{1,2}^\pm(\lambda, x)=\langle e^{\mp \i \lambda|x|}{R}_{0}^{\pm }({\lambda }^{2m})(x-\cdotp ),\varphi (\cdotp )\rangle.
			\end{equation*}
			Given that the proof of \eqref{eq3.33.1}  is  different depending on whether $\lambda\langle x\rangle \le 1$  or $\lambda\langle x\rangle> 1$, we divide the proof into two steps.
			
			\noindent {\bf Step 1: We prove  \eqref{eq3.33.1}  and \eqref{eq3.33.1-1129} when  $\lambda\langle x\rangle > 1$.}
			
			By \eqref{eq-expansion-high-1}, we have \begin{align}\label{eq-lambda x>1}
				\langle e^{\mp \i \lambda|x|}{R}_{0}^{\pm }({\lambda }^{2m})(x-\cdotp ),\varphi (\cdotp )\rangle=\left\langle \frac{{\lambda}^{\frac{n+1}{2}-2m}}{{|x-\cdot|}^{\frac{n-1}{2}}}e^{\mp \i \lambda|x|}{e}^{\pm {\rm i}\lambda|x-\cdot|}U_1^{\pm}(\lambda |x-\cdot|),\,\varphi (\cdot) \right\rangle.
			\end{align}
			Notice that \eqref{eq-expansion-high-4} implies 
			\begin{equation}\label{eq-4.92'}
				\left|  \partial_\lambda^l\left(e^{\mp \i \lambda|x|}{e}^{\pm {\rm i}\lambda|x-\cdot|}U_1^{\pm}(\lambda |x-\cdotp|)\right)\right|\lesssim_l \langle\cdot \rangle^{l} \lambda^{-l}, \quad 0<\lambda<1.
			\end{equation}
			Thus,  by\eqref{eq-lambda x>1} and \eqref{eq2.20}, we  obtain the following estimates for $|x|\le 1$ and $l=0,\, 1$:
			\begin{equation}\label{eq-|x|-small}
				\begin{aligned}
					\left|  \partial_\lambda^l\langle e^{\mp \i \lambda|x|}{R}_{0}^{\pm }({\lambda }^{2m})(x-\cdotp ),\varphi (\cdotp )\rangle \right|&\lesssim {\lambda}^{\frac{n+1}{2}-2m-l} \left|\langle {|x-\cdot|}^{-\frac{n-1}{2}},\, \langle\cdot \rangle \varphi (\cdot)\rangle \right| \\
					&\lesssim {\lambda}^{\frac{n+1}{2}-2m-l}  \langle x \rangle^{-\frac{n-1}{2}-\k_0}\\
					&\lesssim \min\left\{\lambda^{n-2m+\k_0-l},\, \lambda^{\frac{n+1}{2}-2m-l}\langle x \rangle^{-\frac{n-1}{2}}\right\}, 
				\end{aligned}
			\end{equation} 
			where in the second inequality, we applied Lemma \ref{lm-4.2-CHHZ}; and in the last inequality we used the fact that $\lambda\langle x\rangle > 1$.
			
			Now we consider the case $|x|>1$. In this case, note that $1 <\lambda \langle x \rangle\lesssim\lambda|x|$.
			Then, applying Taylor expansion to ${e}^{\pm {\rm i} z}U_1^{\pm }(z)$ at $z_0$, we deduce that
			\begin{equation*}
				\begin{split}
					{e}^{\pm {\rm i} z}U_1^{\pm }(z)
					=\displaystyle\sum_{j=0}^{\k_0-1}\frac{{(e^{\pm {\rm i}z}U_1^{\pm }(z))}^{(j)}(z_0)}{j!}{(z-z_0)}^{j}+r_{\k_0}^{\pm }(z,z_0),
				\end{split}
			\end{equation*}
			where
			$${r}_{\k_0}^{\pm }(z,z_0)=\frac{(z-z_0)^{\k_0}}{(\k_0-1)!}\int_{0}^{1}(1-u)^{\k_0-1} {(e^{\pm {\rm i}z}U_1^{\pm }(z))}^{(\k_0)}(z_0+u(z-z_0))\d u.$$
			Thus,  choosing $z=\lambda|x-\cdot|$ and  $z_0=\lambda|x|$ in \eqref{eq-lambda x>1}, we see that 
			\begin{equation*}
				\mbox{RHS of \eqref{eq-lambda x>1}}=T_1(\lambda,x)+T_2(\lambda,x)
			\end{equation*}
			where
			\begin{equation*}
				T_1(\lambda, x)=\sum_{j=0}^{\k_0-1} \left\langle \frac{{\lambda}^{\frac{n+1}{2}-2m}}{{|x-\cdot|}^{\frac{n-1}{2}}}e^{\mp \i \lambda|x|}\displaystyle\frac{{(e^{\pm {\rm i}z}U_1^{\pm }(z)}^{(j)}(\lambda|x|)}{j!}{(\lambda|x-\cdot|-\lambda|x|)}^{j},\varphi (\cdot) \right\rangle,
			\end{equation*}
			and 
			\begin{equation*}
				\begin{split}
					T_2(\lambda, x)=\left\langle \frac{{\lambda}^{\frac{n+1}{2}-2m}}{{|x-\cdot|}^{\frac{n-1}{2}}}e^{\mp \i \lambda|x|}r_{\k_0}^{\pm }(\lambda|x-\cdot|,\,\lambda|x|),\varphi (\cdot) \right\rangle.
				\end{split}
			\end{equation*}
			It follows from  the Leibniz rule that
			$${(e^{\pm {\rm i}z}U_1^{\pm }(z)}^{(l)}(z_0)=\sum_{s=0}^l (\pm {\rm i})^s\frac{l!}{s!(l-s)!}e^{\pm {\rm i}z_0}(U_1^{\pm})^{(l-s)}(z_0), \quad l\in \N_0.$$
			Then, by \eqref{eq-expansion-high-4}, we derive that for $0\le j\le \k_0-1$ and $0\le l\le 1$
			\begin{equation}\label{eq-sum term}
				\left |  \partial_\lambda^l\left [ e^{\mp \i \lambda|x|}\frac{{(e^{\pm {\rm i}z}U_1^{\pm }(z))}^{(j)}(\lambda|x|)}{j!}\right ]\right |\lesssim \lambda^{-l}
				,\,\,\,0<\lambda<1.
			\end{equation}
			
			For the term $T_1$, by invoking Lemma \ref{lm-4.2-CHHZ},  the binomial theorem, \eqref{eq-sum term} and the assumption $\lambda\left \langle x\right \rangle >1$, we obtain that for $0\le l\le 1$
			\begin{align}\label{eq-T1}
				\left | \partial_\lambda^lT_1(\lambda, x)\right |&\lesssim_l \displaystyle\sum_{j=0}^{\k_0-1} {\lambda }^{\frac{n+1}{2}-2m+j-l}{\left \langle x\right \rangle}^{-\frac{n-1}{2}+j-\k_0}\nonumber\\
				&\lesssim_l\min\left\{\lambda^{n-2m+\k_0-l},\, \lambda^{\frac{n+1}{2}-2m-l}\langle x \rangle^{-\frac{n-1}{2}}\right\}, \quad\,0<\lambda<1.
			\end{align}
			
			Regarding the term $T_2$, 
			we observe that 
			$$e^{\mp \i z_0}r_{\k_0}^{\pm }(z,\,z_0)=\sum_{s=0}^{\k_0} (\pm {\rm i})^s\frac{\k_0 \, (z-z_0)^{\k_0}}{s!(\k_0-s)!}\int_{0}^{1}(1-u)^{\k_0-1} {e^{\pm {\rm i}u(z-z_0)}(U_1^{\pm }})^{(s)}(z_0+u(z-z_0))\d u.$$
			It follows from  \eqref{eq-expansion-high-4} that when $l=0, 1$, 
			\begin{equation*}
				\left |  \partial_\lambda^l(e^{\mp \i \lambda|x|}r_{\k_0}^{\pm }\big(\lambda|x-\cdot|,\,\lambda|x|)\big)\right|\lesssim  \lambda^{\k_0-\frac12-l} |x-\cdot|^{-\frac 12} \langle\cdot \rangle^{\k_0+1},\qquad0<\lambda<1,\,\, |x|>1.
			\end{equation*}
			This implies that  for $|x|>1$ and $l=0, 1$, 
			\begin{equation*}\label{eq-claim-2}
				\begin{split}
					\left | \partial_\lambda^l T_2(\lambda, x)\right | &\lesssim \lambda^{\frac{n}{2}-2m+\k_0-l} \int_{\R^n} |x-y|^{-\frac n2} \langle y \rangle^{\k_0+1}  \varphi(y) \d y\\
					&\lesssim \lambda^{\frac{n}{2}-2m+\k_0-l}{\left \langle x\right \rangle}^{-\frac{n}{2}}\\
					&\lesssim\min\left\{\lambda^{n-2m+\k_0-l},\, \lambda^{\frac{n+1}{2}-2m-l}\langle x \rangle^{-\frac{n-1}{2}}\right\},\,\,\,0<\lambda<1,
				\end{split}
			\end{equation*}
			where in the second inequality, we used \eqref{eq2.20}.
			The above estimate, together with \eqref{eq-|x|-small}, \eqref{eq-T1} and the fact that $\lambda^{n-2m+\k_0-l}\lesssim \lambda^{n-2m+\k_0-l} |\log \lambda|$, implies \eqref{eq3.33.1} and \eqref{eq3.33.1-1129} when $\lambda\left \langle x\right \rangle>1$.

			\noindent {\bf Step 2: We prove  \eqref{eq3.33.1} and \eqref{eq3.33.1-1129} when $\lambda\langle x\rangle\le 1$.}
			
			We shall employ an alternative representation of ${R}_{0}^{\pm }({\lambda }^{2m})(x-\cdotp )$ as  stated in Lemma \ref{lem-expansion-with-theta}. In particular, we shall take $\theta=\k_0$ in \eqref{eq.2-low-odd-1}. However, recall that $\k_0 \le m+1-[\frac{n+1}{2}]$,  when $n<2m$, except for the specific case $n=2m-2$ and $\k_0=2$  we can apply the conditions of Lemma \ref{lem-expansion-with-theta};
			while  if $n=2m-2$ and $\k_0=2$, Lemma \ref{lem-expansion-with-theta} is not applicable because  $m+1-[\frac{n+1}{2}]> 2m-n-1$. Thus we need to analyze this situation separately.
			
			{\it Case 1.}  When $n<2m$ with the exception of $n=2m-2$ and $\k_0=2$, We are allowed to choose $\theta=\k_0$ in \eqref{eq.2-low-odd-1}. Consequently, this leads to the expression
			\begin{equation*}
				{R}_{0}^{\pm }({\lambda }^{2m})(x-y)=\displaystyle\sum_{j=0}^{[\frac{\k_0-1}{2}]}{a}_{j}^{\pm} {\lambda }^{n-2m+2j} {\left | x-y\right |}^{2j}+\lambda^{n-2m}{r}_{\k_0}^{\pm}(\lambda|x-y|).
			\end{equation*}
			The vanishing moment condition of $\varphi$ yields that
			\begin{equation*}
				\sum_{2j<\k_0}^{}{a}_{j}^{\pm }{\lambda }^{n-2m+2j}	\langle | x-\cdot|^{2j}, \varphi (\cdot)\rangle=0,
			\end{equation*}
			which implies that
			\begin{equation}\label{eq-lamda x<1-1}
				\langle e^{\mp \i \lambda|x|}{R}_{0}^{\pm}({\lambda }^{2m})(x-\cdot),\varphi (\cdot)\rangle=\langle e^{\mp \i \lambda|x|}\lambda^{n-2m}{r}_{\k_0}^{\pm}(\lambda|x-\cdot|),\, \varphi (\cdotp)\rangle.
			\end{equation}
			Now we apply Taylor expansion to ${{r}}_{\k_0}^{\pm}(z)$ to obtain the following expression:
			\begin{equation*}
				{r}_{\k_0}^{\pm}(z)=\displaystyle\sum_{j=0}^{\k_0-1}\frac{{({r}_{\k_0}^{\pm})}^{(j)}(z_0)}{j!}{(z-z_0)}^{j}+\tilde{r}_{\k_0}^{\pm}(z,z_0),	
			\end{equation*}
			where
			$$\tilde{r}_{\k_0}^{\pm }(z,z_0)=\frac{(z-z_0)^{\k_0}}{(\k_0-1)!}\int_{0}^{1}(1-u)^{\k_0-1} {({r}_{\k_0}^{\pm})}^{(\k_0)}(z_0+u(z-z_0))\d u.$$
			Letting $z=\lambda|x-\cdot|$ and $z_0=\lambda|x|$ in \eqref{eq-lamda x<1-1}, we see that 
			\begin{equation*}
				\mbox{RHS of \eqref{eq-lamda x<1-1}}=T_3(\lambda,x)+T_4(\lambda,x)
			\end{equation*}
			where
			\begin{equation*}
				T_3(\lambda,x)=e^{\mp \i \lambda|x|}\sum_{j=0}^{\k_0-1} {\lambda}^{n-2m+j}{\frac{{({r}_{\k_0}^{\pm})}^{(j)}(\lambda|x|)}{j!}} \left\langle {(|x-\cdot|-|x|)}^{j},\varphi (\cdot) \right\rangle,  
			\end{equation*}
			and 
			\begin{equation*}
				\begin{split}
					T_4(\lambda,x)={\lambda}^{n-2m}e^{\mp \i \lambda|x|}\left\langle \tilde{r}_{\k_0}^{\pm }(\lambda|x-\cdot|,\,\lambda|x|),\varphi (\cdot) \right\rangle.
				\end{split}
			\end{equation*}
			
			For the term $T_3$, by \eqref{eq-est-rthe-1} we obtain that when 
			$ 0\le j\le \k_0-1$ and $l=0,1$, 
			\begin{equation*}
				\left |  \partial_\lambda^l\left [ e^{\mp \i \lambda|x|}{({r}_{\k_0}^{\pm})}^{(j)}(\lambda|x|)/{j!}\right ]\right |\lesssim \lambda^{\k_0-j-l}|x|^{\k_0-j}
				,\quad 0<\lambda<1.
			\end{equation*}
			Meanwhile, by  Lemma \ref{lm-4.2-CHHZ}, we have
			\begin{equation*}
				\left|\langle (|x-\cdot|-|x|)^{j},\varphi (\cdot) \rangle\right|\le \sum_{k=0}^{j}|C_k||x|^{j-k}||\langle |x-\cdot|^{k},\varphi (\cdot) \rangle|\lesssim\langle x \rangle^{j-\k_0}.
			\end{equation*}
			The above two inequalities, as well as the expression of $T_3(\lambda,x)$, yield
			\begin{equation}\label{eq-T3}
				\left | \partial_\lambda^l T_3(\lambda,x)\right |\lesssim{\lambda }^{n-2m+\k_0-l}\lesssim \min\left\{\lambda^{n-2m+\k_0-l},\, \lambda^{\frac{n+1}{2}-2m-l}\langle x \rangle^{-\frac{n-1}{2}}\right\},
			\end{equation}
			where we  used the fact that $\lambda\langle x\rangle\le1$.
			
			For the term  $T_4$, similar to \eqref{eq-claim-2}, by using \eqref{eq-est-rthe-1}, a direct computation shows that when $\lambda \langle x\rangle \le 1$ and $l=0,1$,
			\begin{equation*}
				\left |  \partial_\lambda^l\Big( e^{\mp \i \lambda|x|}\tilde{r}_{\k_0}^{\pm }(\lambda|x-\cdot|,\,\lambda|x|)\Big)\right |\lesssim_l \lambda^{-l} \langle \cdotp \rangle^{\k_0+1},
			\end{equation*}
			where we used fact  $\big||x-y|-|x|\big|\le |y|$.
			Furthermore, when $l=0,1$, we  have 
			\begin{equation*}
				\left | \partial_\lambda^lT_4(\lambda,x)\right |\lesssim{\lambda }^{n-2m+\k_0-l}\lesssim \min\left\{\lambda^{n-2m+\k_0-l},\, \lambda^{\frac{n+1}{2}-2m-l}\langle x \rangle^{-\frac{n-1}{2}}\right\},\quad .
			\end{equation*}
			This, together with \eqref{eq-T3}, implies \eqref{eq3.33.1} when $\lambda\left \langle x\right \rangle\lesssim 1$.
			
			{\it Case 2.} The exceptional case: $n=2m-2$ and $\k_0=2$. We can deduce from the expansion  \eqref{eq.2-low-even}  that  
			\begin{equation*}
				\begin{split}
					{R}_{0}^{\pm }({\lambda }^{2m})(x-y)=a_0^\pm \lambda^{-2}+a_1^\pm |x-y|^2+b_0|x-y|^2\log(\lambda |x-y|)+\lambda^{-2}r_{3}^{\pm}(\lambda |x-y|),
				\end{split}
			\end{equation*}
			moreover, it follows from  \eqref{eq3.5.0-rem} that one has
			\begin{equation*}
				\left|\frac{{d}^{l}}{d{z}^{l}}{r}_{3}^{\pm }(z) \right|\lesssim z^{3-l},\quad\quad\,\,0\le l\le 1.
			\end{equation*}
			Since $\k_0=2$, it implies that $$\left \langle e^{\mp \i \lambda|x|}a_0^\pm \lambda^{-2}, \varphi(y) \right \rangle=0.$$ Additionally, based on  Lemma \ref{lm-4.2-CHHZ} and the assumption $\lambda\left \langle x\right \rangle\lesssim 1$, a straightforward computation yields
			\begin{equation}\label{eq-n=2m-2}
				\left | \partial_\lambda^l\left \langle e^{\mp \i \lambda|x|}|x-y|^2, \varphi(y) \right \rangle \right |\lesssim{\lambda }^{-l}\lesssim \min\left\{\lambda^{-l}|\log{\lambda}|,\, \lambda^{-\frac{n-1}{2}-l}\langle x \rangle^{-\frac{n-1}{2}}\right\}
			\end{equation}
			and 
			\begin{equation}\label{eq-n=2m-2'}
				\left | \partial_\lambda^l\left \langle e^{\mp \i \lambda|x|}\lambda^{-2}r_{3}^{\pm}(\lambda |x-y|), \varphi(y) \right \rangle \right |\lesssim \min\left\{\lambda^{-l}|\log{\lambda}|,\, \lambda^{-\frac{n-1}{2}-l}\langle x \rangle^{-\frac{n-1}{2}}\right\}
			\end{equation}
			
			It remains to consider the term $\partial_\lambda^l \left \langle e^{\mp \i \lambda|x|}|x-y|^2\log(\lambda|x-y|), \varphi(y) \right \rangle$. By applying a Taylor expansion to $\log(\lambda|x-y|)$ around $y=0$, we derive the following
			\begin{equation*}
				\log(\lambda|x-y|)=\log(\lambda|x|)+\frac{1}{\lambda|x|}(\lambda|x-y|-\lambda|x|)+r_2(\lambda|x-y|,\lambda|x|),
			\end{equation*}
			where the remainder term satisfies 
			\begin{equation*}
				r_2(\lambda|x-y|,\lambda|x|)=o\left ( \frac{1}{(\lambda|x|)^2}(\lambda|x-y|-\lambda|x|)^2 \right ), 
			\end{equation*}
			and its $l$-th derivative satisfies
			\begin{equation*}
				\left | \partial_\lambda^l r_2(\lambda|x-y|,\lambda|x|) \right |\lesssim \lambda^{-l}|x|^{-2}|y|^2.
			\end{equation*}
			
			On the one hand, using the condition $\lambda\left \langle x\right \rangle\le 1$, $\k_0=2$, and applying Lemma \ref{lm-4.2-CHHZ} along with the triangle inequality $\left| |x-y|-|x|\right |\le |y|$, we have
			\begin{equation*}
				\begin{split}
					\left | \partial_\lambda^l \left \langle e^{\mp \i \lambda|x|}|x-y|^2\log(\lambda|x|), \varphi(y) \right \rangle \right |\lesssim \left | \partial_\lambda^l \left \langle e^{\mp \i \lambda|x|}|x-y|^2|\log\lambda|, \varphi(y) \right \rangle \right |\lesssim {\lambda }^{-l}|\log \lambda|,
				\end{split}
			\end{equation*}
			\begin{equation}\label{1891}
				\left | \partial_\lambda^l \left \langle e^{\mp \i \lambda|x|}|x-y|^2|x|^{-1}(|x-y|-|x|), \varphi(y) \right \rangle \right |\lesssim {\lambda }^{-l},
			\end{equation}
			and 
			\begin{equation}\label{1894}
				\left | \partial_\lambda^l \left \langle e^{\mp \i \lambda|x|}|x-y|^2r_2(\lambda|x-y|,\lambda|x|), \varphi(y) \right \rangle \right |\lesssim {\lambda }^{-l}.
			\end{equation}
			The above three inequalities  imply
			\begin{equation}\label{eq-n=2m-2''}
				\left | \partial_\lambda^l \left \langle e^{\mp \i \lambda|x|}|x-y|^2\log(\lambda|x-y|), \varphi(y) \right \rangle \right |\lesssim {\lambda }^{-l}|\log \lambda|.
			\end{equation}
			
			On the other hand, notice that  $\log(\lambda|x|)\le (\lambda\left \langle x\right \rangle)^{-\delta}$ holds when $\lambda\left \langle x\right \rangle\le 1$ and $\delta>0$. Letting $\delta=\frac{n-1}{2}$, we deduce that
			\begin{equation*}
				\left | \partial_\lambda^l \left \langle e^{\mp \i \lambda|x|}|x-y|^2\log(\lambda|x|), \varphi(y) \right \rangle \right |\lesssim \lambda^{-\frac{n-1}{2}-l}\langle x \rangle^{-\frac{n-1}{2}}.
			\end{equation*}
			This, together with \eqref{1891}, \eqref{1894} and the fact that  ${\lambda }^{-l}\lesssim \lambda^{-\frac{n-1}{2}-l}\langle x \rangle^{-\frac{n-1}{2}}$, yields 
			\begin{equation}\label{eq-n=2m-2'''}
				\left | \partial_\lambda^l \left \langle e^{\mp \i \lambda|x|}|x-y|^2\log(\lambda|x-y|), \varphi(y) \right \rangle \right |\lesssim \lambda^{-\frac{n-1}{2}-l}\langle x \rangle^{-\frac{n-1}{2}}.
			\end{equation}
			
			By combining  \eqref{eq-n=2m-2}, \eqref{eq-n=2m-2'}, \eqref{eq-n=2m-2''} and \eqref{eq-n=2m-2'''}, we prove \eqref{eq3.33.1-1129} when $\lambda\left \langle x\right \rangle\le 1$.
			
			Therefore  the proof of the lemma is complete.
		\end{proof}
		
		We are in the position to estimate ${K}_{n}^{\pm ,l}(t,x,y)$ given by \eqref{equ-2-low}. By Proposition \ref{lm-n<2m-odd-vanish}, we rewrite  \eqref{equ-2-low} as 
		\begin{equation}\label{eq-K_n^l}
			\begin{split}
				{K}_{n}^{\pm ,l}(t,x,y)=&-\frac{m\alpha}{\pi \i}\int_{0}^{+\infty }{e}^{-{\rm i}t{\lambda }^{2m}\pm {\rm i}\lambda (|x|+|y|)}{F}_{\pm }^{\alpha ,l}({\lambda }^{2m})W_{1,2}^\pm(\lambda, x)W_{1,2}^\pm(\lambda,y){\lambda }^{2m-1}\d \lambda.
			\end{split}
		\end{equation}
		In view of the vanishing moment conditions in Lemma \ref{lm3.4} and Lemma \ref{lm3.6},  we first assume that {\it there exists a multi-index $\kappa_0$  ($|\kappa_0|=\mathbf{k}_0$) with $0\le \mathbf{k}_0\le m-[\frac{n+1}{2}]-1$ such that $\int_{{\mathbb R}^{n}}^{}{x}^{\kappa_0}\varphi (x)\d x \neq 0,$ but}
		\begin{equation}\label{eq-vani-1123}
			\int_{{\mathbb R}^{n}}^{}{x}^{\beta}\varphi (x)\d x=0,\, \, 0 \le |\beta|<\k_0.    
		\end{equation}
		We set
		\begin{equation}\label{eq-def-psi-1}
			\psi(\lambda,x )={F}_{\pm }^{\alpha ,l}({\lambda }^{2m})W_{1,2}^\pm(\lambda, x)W_{1,2}^\pm(\lambda,y){\lambda }^{2m-1},
		\end{equation}
		by \eqref{equ4-F1}, \eqref{eq3.33.1} and \eqref{eq3.33.1-1129}, we obtain that for $0\le l\le 1$,
		\begin{align}\label{eq-psi-n<2m-sub1-1124}
			\left |\partial_{\lambda}^l \psi(\lambda,x)\right |\lesssim \lambda^{n-1-l},
		\end{align}
		and
		\begin{align}\label{eq-psi-n<2m-sub1}
			{{\langle x \rangle}^{\frac{n-1}{2}}}\left | \partial_\lambda^l \psi(\lambda,x)\right |\lesssim \lambda ^{\frac{n-1}{2}-l}.
		\end{align}
		where the constant is uniformly with respect to $x$ and $\lambda$.
		
		Now we examine the kernel ${K}_{n}^{\pm ,l}(t,x,y)$ within two distinct regions: ${t}^{-\frac{1}{2m}}(|x|+|y|)\le 1$ and  ${t}^{-\frac{1}{2m}}(|x|+|y|)>1$.
		
		When ${t}^{-\frac{1}{2m}}(|x|+|y|)\le 1$,  by \eqref{eq-psi-n<2m-sub1-1124}, we
		apply (\ref{eq-osc-l-1}) with $b=n-1$ to obtain that
		\begin{equation}\label{n<2m-odd-1}
			\left |{K}_{n}^{\pm ,l}(t,x,y)\right |\lesssim(1+{t}^{\frac{1}{2m}})^{-n}\lesssim {t}^{-\frac{n}{2m}}(1+{t}^{-\frac{1}{2m}}\left | x-y\right |)^{-\frac{n(m-1)}{2m-1}},
		\end{equation}
		where we used the fact $|x-y|\le |x|+|y|$ and the assumption ${t}^{-\frac{1}{2m}}(|x|+|y|)\le 1$.
		
		While when ${t}^{-\frac{1}{2m}}(|x|+|y|)\ge 1$, by  \eqref{eq-psi-n<2m-sub1}, we apply (\ref{eq-osc-g-1}) with $b=\frac{n-1}{2}$ and derive the following estimate:
		\begin{equation}\label{n<2m-odd-2}
			\begin{split}
				\left|{K}_{n}^{\pm ,l}(t,x,y)\right|&\lesssim {t}^{-\frac{n+1}{4m}}({t}^{-\frac{1}{2m}}(\left | x\right |+\left | y\right |))^{-\frac{m-1-\frac{n-1}{2}}{n-1}}{\left \langle x\right \rangle}^{-\frac{n-1}{2}}\\
				&\lesssim {t}^{-\frac{n}{2m}}(1+{t}^{-\frac{1}{2m}}(\left | x\right |+\left | y\right |))^{-\frac{n(m-1)}{2m-1}}\\
				&\lesssim {t}^{-\frac{n}{2m}}(1+{t}^{-\frac{1}{2m}}\left | x-y\right |)^{-\frac{n(m-1)}{2m-1}},
			\end{split}
		\end{equation}
		where in the second and third inequality above, we used the fact that, given the assumption $|x|\gtrsim|y|$, the following chain of inequalities holds: $|x-y|\le |x|+|y|\lesssim |x|\lesssim\langle x \rangle $.
		
		Combining \eqref{n<2m-odd-1} and \eqref{n<2m-odd-2}, we prove the desired estimate of ${K}_{n}^{\pm ,l}(t,x,y)$ in this case.
		
		To finish the proof, we note that {\it 
			if there exists a multi-index $\kappa_0$ with $|\kappa_0|=\mathbf{k}_0=m-[\frac{n+1}{2}]$ such that $\int_{{\mathbb R}^{n}}^{}{x}^{\kappa_0}\varphi (x)\d x \neq 0,$ but
			\eqref{eq-vani-1123} holds} or {\it $\int_{{\mathbb R}^{n}}^{}{x}^{\beta}\varphi (x)\d x=0 \text{ holds for all }\ 0\le |\beta|\le m-[\frac{n+1}{2}]$.} The arguments are quite similar. Indeed, we let ${\psi}(\lambda,x)$ be given by \eqref{eq-def-psi-1}, then \eqref{eq-psi-n<2m-sub1-1124} and \eqref{eq-psi-n<2m-sub1} hold by \eqref{eq3.33.1}, \eqref{equ4-F2} and 
		\eqref{eq3.6.2}.
		Therefore the proof is complete.
		\qed


		
		\subsection{The proof of Theorem \ref{thm3.2} for $n=2m$}\
		
		In analogy with  Proposition \ref{lm-n<2m-odd-vanish}, we establish the following result for $n=2m$.
		\begin{proposition}\label{lm-n=2m-odd-vanish}
			Let  $n=2m$ and let $\varphi$ satisfy Assumption \ref{assumption-1}. Then we have the representation:
			\begin{equation}\label{eq4.54}
				\int_{\mathbb{R}^n } R_0^{\pm}(\lambda^{2m})(x-y) \varphi(y)\d y=e^{\pm \i\lambda|x|}W_{1,3}^\pm(\lambda, x), \qquad 0<\lambda<\frac{1}{2},
			\end{equation}
			where $W_{1,3}^\pm$ satisfies the  following bounds  for $0\le l\le 2$,
			\begin{equation}\label{eq3.33.8}
				{\left \langle x\right \rangle}^{\frac{n-1}{2}}\left|  \partial_\lambda^lW_{1,3}^\pm(\lambda, x)\right |\lesssim \lambda^{-\frac{n-1}{2}-l}, \qquad 0<\lambda<\frac{1}{2}.
			\end{equation}
			Moreover, if $\int_{\mathbb{R}^n}^{} \varphi(x) \d x \neq 0$, we have
			\begin{equation}\label{eq3.33.5}
				\left | \partial_\lambda^lW_{1,3}^\pm(\lambda, x)\right |\lesssim |\log \lambda| \lambda^{-l}, \qquad 0<\lambda<\frac{1}{2}.
			\end{equation}
			If $\int_{\mathbb{R}^n}^{} \varphi(x) \d x = 0$,
			\begin{equation}\label{eq3.33.7}
				\left | \partial_\lambda^lW_{1,3}^\pm(\lambda, x)\right |\lesssim \lambda^{-l}, \qquad\qquad 0<\lambda<\frac{1}{2}. 
			\end{equation}
			
		\end{proposition}
		
		\begin{proof}
			The representation \eqref{eq4.54} follows immediately if we  set
			\begin{equation*}
				W_{1,3}^\pm(\lambda, x)=\langle e^{\mp \i \lambda|x|}{R}_{0}^{\pm }({\lambda }^{2m})(x-\cdot ),\ \varphi (\cdot )\rangle.
			\end{equation*}
			
			First, we consider the case $\lambda\langle x\rangle > 1$. Note that \eqref{eq-expansion-high-1} is still valid when $n=2m$, thus by \eqref{eq-lambda x>1} and \eqref{eq-4.92'}, we have
			\begin{equation*}
				\begin{split}
					\left|  \partial_\lambda^lW_{1,3}^\pm(\lambda, x)\right|&\lesssim {\lambda}^{\frac{n+1}{2}-2m-l} \left|\langle {|x-\cdot|}^{-\frac{n-1}{2}}, \left \langle \cdot \right \rangle\varphi (\cdot)\rangle \right| \\
					&\lesssim {\lambda}^{\frac{n+1}{2}-2m-l}  \langle x \rangle^{-\frac{n-1}{2}}\\
					&\lesssim \lambda^{-l},
				\end{split}
			\end{equation*}
			where in the last inequality, we used the assumption $\lambda\langle x\rangle > 1$.
			Moreover,  note that $\lambda^{-l}\lesssim |\log \lambda| \cdot \lambda^{-l}$ holds for $0<\lambda<\frac 12$. Thus \eqref{eq3.33.8}-\eqref{eq3.33.7} follows for $\lambda\langle x\rangle > 1$.
			
			Next, we consider the case $\lambda\langle x\rangle\le 1$. By \eqref{eq.2-low-even}, we have
			\begin{equation*}
				{R}_{0}^{\pm }({\lambda }^{2m})(x-y)=a_0^{\pm}+b_0 \log \lambda +b_0 \log |x-y|+r^{\pm}_{2m-n+1}(\lambda |x-y|),
			\end{equation*}
			where
			\begin{equation*}
				\left | \frac{d^l}{dz ^l}r^{\pm}_{2m-n+1}(z)\right |\lesssim |z|^{1-l},\ \ 0\le l\le 2, \quad z\in \R.
			\end{equation*}
			Now, we have 
			\begin{equation}\label{eq-exp-er}
				\begin{split}
					W_{1,3}^\pm(\lambda, x)&=e^{\mp \i \lambda|x|}\Big(a_0^{\pm} \int_{\mathbb{R}^n }^{} \varphi(y)\d y+b_0 \int_{\mathbb{R}^n }^{} \log (\lambda|x-y|)\varphi(y)\d y\\
					&+ \int_{\mathbb{R}^n }^{}  r^{\pm}_{2m-n+1}(\lambda |x-y|)\varphi(y)\d y\Big).
				\end{split}   
			\end{equation}
			For $l=0,1,2$, a direct computation yields 
			\begin{equation}\label{eq-est-a_0}
				\left | \partial_\lambda^l e^{\mp \i \lambda|x|}  \right|\lesssim |x|^l\lesssim  \lambda ^{-l},
			\end{equation}
			and
			\begin{equation}\label{eq-est-r_2m-n+1}
				\left |  \partial_\lambda^l \int_{\mathbb{R}^n }^{}  r^{\pm}_{2m-n+1}(\lambda |x-y|)\varphi(y)\d y\right|\lesssim  \lambda ^{-l}.
			\end{equation}
			
			
			
			We now proceed to evaluate the second term on the right hand side of \eqref{eq-exp-er}. The absolute value of the integral is bounded by
			\begin{equation*}
				\begin{split}
					\left |b_0 \int_{\mathbb{R}^n}^{} \log (\lambda |x-y|)\varphi(y) \d y\right|&\lesssim \left |\int_{\mathbb{R}^n}^{}(\lambda |x-y|)^{-\frac 12}\varphi(y) \d y\right|+\left |\int_{\mathbb{R}^n}^{} (\lambda |x-y|)^{\frac 12}\varphi(y) \d y\right|\\
					&\lesssim  \lambda^{-\frac 12}\left \langle x \right \rangle^{-\frac 12} + \lambda^{\frac 12}\left \langle x \right \rangle^{\frac 12} \\
					&\lesssim \lambda ^{-\frac{n-1}{2}}\left \langle x \right \rangle^{-\frac{n-1}{2} },
				\end{split}
			\end{equation*}
			where in the first inequality, we used the fact
			\begin{equation}\label{eq-log-in}
				|\log a|\lesssim |a|^{\frac12}+|a|^{-\frac12},\quad\,a>0,
			\end{equation}
			and in the second and the third inequality, we used \eqref{eq2.20} and the condition $\lambda\langle x\rangle\le 1$ respectively.
			Additionally, we observe that
			$$\sup_{x,y}|{\partial_\lambda ^l}\log (\lambda |x-y|)|\lesssim_l \lambda^{-l},\,\,\,l\ge1.$$ Thus, it follows that
			\begin{equation*}
				\left |  \partial_\lambda^l\left \langle b_0\log (\lambda |x-y|),\varphi(y) \right \rangle    \right|\lesssim  \lambda ^{-\frac{n-1}{2} -l}\left \langle x \right \rangle^{-\frac{n-1}{2} }
			\end{equation*}
			holds for $l=0,1,2$ and $\lambda\langle x\rangle\le 1$.
			This, together with \eqref{eq-exp-er}, \eqref{eq-est-a_0} and \eqref{eq-est-r_2m-n+1} 
			implies \eqref{eq3.33.8} for $\lambda\langle x\rangle\le 1$.
			
			Moreover, when $\int_{\mathbb{R}^n}^{} \varphi(x) \d x \neq 0$,  note that
			$$\left|\log |x-y|\right|\lesssim |x-y|^{-\frac{1}{2}}+\log(1+|x|)+\log(1+|y|),$$
			which indicates that
			\begin{equation*}
				\begin{split}
					\left | \int_{\mathbb{R}^n}^{} \log |x-y|\varphi(y) \d y\right|&\lesssim \left | \int_{\mathbb{R}^n} \left(|x-y|^{-\frac{1}{2}}+\log((1+|x|)(1+|y|))\right)|\varphi(y)| \d y\right|\\
					&\lesssim 1+\log(1+|x|) \lesssim |\log \lambda|.
				\end{split}
			\end{equation*}
			This, together with \eqref{eq-exp-er}, \eqref{eq-est-a_0} and \eqref{eq-est-r_2m-n+1}, implies \eqref{eq3.33.5} for $\lambda\langle x\rangle\le 1$.

			On the other hand, if $\int_{\mathbb{R}^n}^{} \varphi(x) \d x =0$, we have
			\begin{equation*}
				\begin{split}
					\left|\left \langle b_0\log |x-y|,\varphi (y) \right \rangle\right|&=\left|\left \langle b_0 \log\left(\frac{ |x-y|}{\langle x \rangle }\right),\varphi (y) \right \rangle\right|\\
					&\lesssim \int_{\mathbb{R}^n}^{}\left ( \frac{|x-y|}{\langle x \rangle } \right )^{\frac 12} \varphi (y) \d y + \int_{\mathbb{R}^n}^{}\left ( \frac{\langle x \rangle }{|x-y|} \right )^{\frac 12} \varphi (y) \d y \lesssim 1,
				\end{split}
			\end{equation*}
			where in the first inequality, we used \eqref{eq-log-in}
			; and in the second inequality, we used \eqref{eq2.20}. 
			Thus we obtain  \eqref{eq3.33.7} for $\lambda\langle x\rangle\le 1$.
			Therefore the proof is complete.
		\end{proof}
		
		Now we are ready to prove Theorem \ref{thm3.2} when $n=2m$. 
		Similar to the case $n<2m$, we analyze the kernel ${K}_{n}^{\pm ,l}(t,x,y)$ within two separate regions: ${t}^{-\frac{1}{2m}}(|x|+|y|)\le 1$ and  ${t}^{-\frac{1}{2m}}(|x|+|y|)>1$.
		
		When $t^{-\frac{1}{2m}}(|x|+|y|)>1$,
		we rewrite ${K}_{n}^{\pm ,l}(t,x,y)$  as  follows:
		\begin{equation}\label{eq-K_n^l'}
			\begin{split}
				{K}_{n}^{\pm ,l}(t,x,y)=&-\frac{m\alpha}{\pi \i}\int_{0}^{+\infty }{e}^{-{\rm i}t{\lambda }^{2m}\pm {\rm i}\lambda (|x|+|y|)}{F}_{\pm }^{\alpha ,l}({\lambda }^{2m})W_{1,3}^\pm(\lambda, x)W_{1,3}^\pm(\lambda,y){\lambda }^{2m-1}\d \lambda.
			\end{split}
		\end{equation}
		Set 
		\begin{equation}
			\psi(\lambda,x)={F}_{\pm }^{\alpha ,l}({\lambda }^{2m})W_{1,3}^\pm(\lambda, x)W_{1,3}^\pm(\lambda,y){\lambda }^{2m-1},
		\end{equation}
		we deduce that for $0\le l\le 2$,
		\begin{equation*}
			{{\langle x \rangle}^{\frac{n-1}{2}}}\left | \partial_\lambda^l \psi(\lambda,x )\right |\lesssim \lambda ^{\frac{n-1}{2}-l},
		\end{equation*}
		which is derived  from \eqref{eq-n=2m-F-log}, \eqref{eq3.33.8}, \eqref{eq3.33.5} if $\int_{\mathbb{R}^n}^{} \varphi(x) \d x \neq 0$ and from \eqref{eq-n=2m-F-}, \eqref{eq3.33.8}, \eqref{eq3.33.7} if $\int_{\mathbb{R}^n}^{} \varphi(x) \d x = 0$.
		Observe that this result coincides with the second inequality in \eqref{eq-psi-n<2m-sub1}. Consequently, following the same arguments as in \eqref{n<2m-odd-2},  it follows that
		\begin{equation*}
			\left| K_n^{\pm,l}(t,x,y) \right| \lesssim {t}^{-\frac{n}{2m}}(1+{t}^{-\frac{1}{2m}}|x-y|)^{-\frac{n(m-1)}{2m-1}}.
		\end{equation*}
		
		When $t^{-\frac{1}{2m}}(|x|+|y|)\le 1$,
		as per the expression \eqref{eq.l7}, it is sufficient to consider the kernels $K_{n,1}^{l}(t,x,y)$ and $K_{n,2}^{l}(t,x,y)$. Using \eqref{eq-res-low-nl2m-4-0}, they can be expressed as follows (up to constants):
		\begin{equation*}
			\begin{split}
				K_{n,1}^{l}(t,x,y)=\int_{0}^{+\infty }{e}^{-{\rm i}t{\lambda }^{2m}+\i \lambda(|x|+|y|)}({F}_{+}^{\alpha ,l}({\lambda }^{2m})-{F}_{-}^{\alpha ,l}({\lambda }^{2m})){\lambda }^{2m-1}W_{1,3}^+(\lambda, x)W_{1,3}^+(\lambda, y) {\rm d} \lambda,
			\end{split}
		\end{equation*}
		and
		\begin{equation*}
			\begin{split}
				K_{n, 2}^{l}(t,x,y)=\sum_{\pm}\mp \int_{0}^{+\infty }{e}^{-{\rm i}t{\lambda }^{2m}-\i \lambda(|x|\mp|y|)}{F}_{-}^{\alpha ,l}({\lambda }^{2m}){\lambda }^{2m-1}W_{1,3}^-(\lambda, x)U_{0,0}^\pm(\lambda,y){\rm d} \lambda .
			\end{split}
		\end{equation*}
		
		We first estimate $K_{n,1}^{l}(t,x,y)$.
		In light of  Proposition \ref{lm-n=2m-odd-vanish}, we shall examine two distinct situations:  the case where $\int_{\mathbb{R}^n}^{} \varphi(x) \d x\ne0$ and the case where it is zero.  We define
		\begin{equation*}
			\psi_1(\lambda,x)=({F}_{+}^{\alpha ,l}({\lambda }^{2m})-{F}_{-}^{\alpha ,l}({\lambda }^{2m})){\lambda }^{2m-1}W_{1,3}^+(\lambda, x)W_{1,3}^+(\lambda, y).
		\end{equation*}
		If $\int_{\mathbb{R}^n}^{} \varphi(x) \d x \neq 0$, then by \eqref{eq-n=2m-F+F-_log}, \eqref{eq3.33.5}, we deduce that
		\begin{equation}\label{eq-4.105}
			\left | \partial_\lambda^l \psi_1(\lambda,x)\right |\lesssim \lambda ^{n-1-l}.
		\end{equation}
		While if $\int_{\mathbb{R}^n}^{} \varphi(x) \d x = 0$, the estimate \eqref{eq-4.105} remains valid due to \eqref{eq-n=2m-F+_F-} and  \eqref{eq3.33.7}.
		Therefore, applying \eqref{eq-osc-l-1} with $b=n-1$, we obtain
		\begin{equation*}
			\left| K_{n,1}^{l}(t,x,y) \right|\lesssim {t}^{-\frac{n}{2m}}(1+{t}^{-\frac{1}{2m}}|x-y|)^{-\frac{n(m-1)}{2m-1}}.
		\end{equation*}
		
		Next, we estimate  $K_{n,2}^{l}(t,x,y)$.
		We define
		\begin{equation*}
			\psi^\pm_2(\lambda,x)={F}_{-}^{\alpha ,l}({\lambda }^{2m}){\lambda }^{2m-1}W_{1,3}^-(\lambda, x)U_{0,0}^\pm(\lambda,y).
		\end{equation*}
		It follows from \eqref{eq-res-low-nl2m-4-0} and Proposition \ref{lm-n=2m-odd-vanish} that $\psi^\pm_2(\lambda)$ satisfies \eqref{eq-4.105}, regardless of whether the integral $\int_{\mathbb{R}^n} \varphi(x)\d x$ vanishes or not. Consequently, we conclude that  $ K_{n,2}^{l}(t,x,y)$ satisfies \eqref{eq3.22.11} as desired.
		
		Therefore the proof when $n=2m$ is finished.
		\qed
		
		\begin{remark}\label{rmk3.4}
			We mention that during the proof of the low energy estimate \eqref{eq3.22.11}, all properties that we need for $F^{\alpha, l}_{\pm}(\lambda^{2m})$ are contained in Lemma \ref{lm3.3}-\ref{lm3.7}. These properties  will be used in the subsequent analysis of finite rank perturbations.
		\end{remark}
		
		\section{Proof of Theorem \ref{cor-finite-rank}}\label{sec5}

		In this section,  we consider the finite rank case \eqref{eq-high-per} and prove Theorem \ref{cor-finite-rank}. 
		By limiting absorption principle and the decay assumption (\ref{eq1.2}) of $\varphi_j\ (j=1,\cdots,N)$, we are allowed to  define
		\begin{equation}\label{def-fij}
			F_{N\times N}^{\pm}(\lambda^{2m})=\left(f_{ij}^{\pm}(\lambda^{2m})\right)_{N\times N},\,\,\,\text{where}   \,\,f_{ij}^{\pm}(\lambda^{2m}):=\langle R_0^{\pm}(\lambda^{2m})\varphi_j, \varphi_i\rangle,\,\,\,\,\,1\leq i, j\leq N.   
		\end{equation}
		Moreover, each $f_{ij}^{\pm}(\lambda^{2m})$ is continuous on $(0, \infty)$. Denote the boundary value of the matrix by
		\begin{equation}\label{eq-CHY4.11}
			A^{\pm}(\lambda^{2m}):=A(\lambda^{2m}\pm \i 0)=I_{N\times N}+F_{N\times N}^{\pm}(\lambda^{2m}),\qquad\,\,\,\lambda>0.
		\end{equation}
		By assumption \eqref{eq-low-spec}, it follows that $H$ only has the absolutely continuous spectrum  (see \cite[Lemma B.1]{CHZ}) .  
		Let $P=\displaystyle\sum_{j=1}^{N}P_j$, where $P_j=\langle \cdot,\,  \varphi_j\rangle \varphi_j$.
		By resolvent identity, we have the following Aronszajn-Krein type formula
		\begin{equation}\label{eq4.13}
			\begin{aligned}
				R^{\pm}(\lambda^{2m})&=R^{\pm}_{0}(\lambda^{2m})- R^{\pm}_{0}(\lambda^{2m})P\Big(I+PR^{\pm}_{0}(\lambda^{2m})P\Big)^{-1}PR^{\pm}_{0}(\lambda^{2m})\\
				&=R^{\pm}_{0}(\lambda^{2m})-\sum_{i,j=1}^N{g^{\pm}_{ij}(\lambda^{2m})}\cdot R^{\pm}_{0}(\lambda^{2m})\varphi_i\langle R^{\pm}_{0}(\lambda^{2m})\cdot,  \varphi_j\rangle, 
			\end{aligned}
		\end{equation}
		where
		\begin{equation}\label{eq4.13.1}
			R^{\pm}(\lambda^{2m})=(H_N-(\lambda^{2m} \pm \i 0))^{-1}, \quad  g_{i, j}^{\pm}(\lambda^{2m})=\Big\langle \Big(I+PR^{\pm}_{0}(\lambda^{2m})P\Big)^{-1}\varphi_j, \, \varphi_i \Big\rangle.
		\end{equation}
		Set $ G^{\pm}(\lambda^{2m}):=\left(g_{ij}^{\pm}(\lambda^{2m})\right)_{N\times N}$, then by \eqref{eq-CHY4.11}, we have 
		\begin{equation}\label{eq-gij}
			G^{\pm}(\lambda^{2m})=A^{\pm}(\lambda^{2m})^{-1}.
		\end{equation}

		Now we are ready to give a spectral representation  formula for the evolution $e^{-\i tH_N}$.  Based on the proof of \cite[Proposition 4.1]{CHZ}, we have 
		\begin{align}\label{eq-evolution-1}
			e^{-\i tH_N}-e^{-\i tH_{0}}&=-\frac{m}{\pi \i }\int_0^{\infty}e^{-\i t\lambda^{2m}}R^{+}_{0}(\lambda^{2m})P(I+PR^{+}_{0}(\lambda^{2m})P)^{-1}PR^{+}_{0}(\lambda^{2m}) \lambda^{2m-1}\,d\lambda\nonumber\\
			&+ \frac{m}{\pi \i }\int_0^{\infty}e^{-\i t\lambda^{2m}}R^{-}_{0}(\lambda^{2m})P(I+PR^{-}_{0}(\lambda^{2m})P)^{-1}PR^{-}_{0}(\lambda^{2m}) \lambda^{2m-1}\,d\lambda\\
			&:=\left(K^+_{n,N}-K^-_{n,N}\right)\nonumber.
		\end{align}
		By Stone's formula and \eqref{eq4.13}, we also have
		\begin{align}\label{eq4.13.2}
			e^{-\i tH_N}-e^{-\i tH_{0}}&=-\frac{m}{\pi \i }\sum_{i,j=1}^N\int_0^{\infty}e^{-\i t\lambda^{2m}}g^+_{ij}(\lambda^{2m})R^{+}_{0}(\lambda^{2m})\varphi_i\langle R^{+}_{0}(\lambda^{2m})\, \cdot, \varphi_j\rangle \lambda^{2m-1}\,\d\lambda\nonumber\\
			&+\frac{m}{\pi \i }\sum_{i,j=1}^N\int_0^{\infty}e^{-\i t\lambda^{2m}}g^-_{ij}(\lambda^{2m})R^{-}_{0}(\lambda^{2m})\varphi_i\langle R^{-}_{0}(\lambda^{2m})\, \cdot, \varphi_j\rangle \lambda^{2m-1}\,\d\lambda\nonumber\\
			&:=\sum_{i,j=1}^N{\left(K^+_{n, ij}-K^-_{n, ij}\right)}.
		\end{align}
		We remark that \eqref{eq4.13.2} plays the same role as \eqref{equ-rank1-st} in the  rank one case (with $\alpha=1$).
		
		As in the rank one case, we  first  choose a smooth cut-off function $\chi(\lambda)$ satisfying \eqref{equ-chi} where $\lambda_0$ is a small constant to be chosen later. Then we split \eqref{eq4.13.2} into the  high and low energy part respectively:
		\begin{align*}
			K^{\pm, h}_{n, ij}&:=-\frac{m}{\pi \i }\int_0^{\infty}e^{-\i t\lambda^{2m}}(1-\chi(\lambda))g^{\pm}_{ij}(\lambda^{2m})R^{\pm}_{0}(\lambda^{2m})\varphi_i\langle R^{\pm}_{0}(\lambda^{2m})\, \cdot, \varphi_j\rangle\lambda^{2m-1}\,\d\lambda,
		\end{align*}
		\begin{equation*}
			K^{\pm, l}_{n, ij}:=-\frac{m}{\pi \i }\int_0^{\infty}e^{-\i t\lambda^{2m}}\chi(\lambda)g^{\pm}_{ij}(\lambda^{2m})R^{\pm}_{0}(\lambda^{2m})\varphi_i\langle R^{\pm}_{0}(\lambda^{2m})\, \cdot, \varphi_j\rangle\lambda^{2m-1}\,\d\lambda.
		\end{equation*}
		
		We mention that by comparing \eqref{eq4.13.2} with \eqref{equ-rank1-st}, the main difference is that $\frac{1}{1+ F^{\pm}(\lambda^{2m})}$ is replaced by $g^{\pm}_{ij}(\lambda^{2m})$. Should we establish that $g^{\pm}_{ij}(\lambda^{2m})\ (1\le i,j \le N)$ shares the same properties with $\frac{1}{1+ F^{\pm}(\lambda^{2m})}$, the analysis could be reduced to the rank one case. However, Lemma \ref{lm3.4}-\ref{lm3.7} indicate that the vanishing moment conditions of $\varphi_j$ will affect the analysis of $g^{\pm}_{ij}(\lambda^{2m})$ in the low energy part when $n\le 2m$, so we will treat them separately.
		
		To this end, we divide the proof into the following four parts.
		\subsection{High energy estimate.}\label{sec5.1.1}\
		
		By Remark \ref{rmk3.1}, it's enough to prove that for any $\lambda_0>0$, there exists an absolute constant $C=C(N, \lambda_0, \varphi_1,\cdots, \varphi_N)>0$ such that for all $\lambda>\lambda_0$, and $1\leq i, j\leq N$, one has
		\begin{equation}\label{eq4.14}
			\begin{split}
				\left|\frac{d^l}{d\lambda^l}g_{ij}^{\pm}(\lambda^{2m})\right|\leq
				\begin{cases}C,\,\,\,\,\,\,\,\quad\quad  \mathrm{if}\,\,\,l=0,\\
					C\lambda^{1-2m}, \,\, \,\,\,\quad \mathrm{if}\,\,\,~l=1,\cdots,[\frac{n}{2m}]+1.
				\end{cases}
			\end{split}
		\end{equation}
		Indeed, we observe that similar to the proof of \eqref{equ2.1.1} in Lemma \ref{lm2.1}, it follows from the decay assumptions on $\varphi_j$ that for all $\lambda>\lambda_0$,
		\begin{equation}\label{eq4.14.1}
			\left|\frac{d^l}{d\lambda^l}f^{\pm}_{i,j}(\lambda^{2m})\right|\leq C\lambda^{1-2m}, \,\,\,\quad \mathrm{if}\,\,\,~l=0. 1,\cdots,[\frac{n}{2m}]+1.
		\end{equation}
		Recall that by \eqref{eq4.13.1},
		\begin{equation}\label{eq4.14.2}
			\left(g_{ij}^{\pm}(\lambda^{2m})\right)_{N\times N}=A^{\pm}(\lambda^{2m})^{-1}=\frac{1}{\det{(A^{\pm}(\lambda^{2m}))}}\text {adj}(A^{\pm}(\lambda^{2m})),
		\end{equation}
		where $\text{adj}(A^{\pm}(\lambda^{2m}))$ is the adjugate matrix of $A^{\pm}(\lambda^{2m})$ (given in \eqref{eq-CHY4.11}).
		Note that by assumption \eqref{eq-low-spec} one has $|\det{A^{\pm}(\lambda^{2m})}|\ge c_0>0$, thus by \eqref{eq4.14.1} and \eqref{eq4.14.2},  there are some positive constants depending on $N, \lambda_0, \varphi_1,\cdots, \varphi_N$ such that
		\begin{equation*}
			|g_{ij}^{\pm}(\lambda^{2m})|\leq C(N, \lambda_0, \varphi_1,\cdots, \varphi_N),
		\end{equation*}
		which proves \eqref{eq4.14} for $l=0$. Further, we notice that
		\begin{equation}\label{eq4.14.3}
				\frac{d^l}{d\lambda^l}A^{\pm}(\lambda^{2m})^{-1} = A^{\pm}(\lambda^{2m})^{-1}\sum_{s=1}^{l}\sum_{l_{m,1}+\cdots+l_{m,s}=l}C_{l_{m,s}}\prod_{m=1}^s\left[\frac{d^{l_{m,s}}}{d\lambda^{l_{m,s}}}(A^{\pm}(\lambda^{2m})) \cdot A^{\pm}(\lambda^{2m})^{-1}\right],
			\end{equation}
			where $C_{l_{m,s}}=C(l_{m,1},\cdots,l_{m,s})$ depends on $l_{m,1},\cdots,l_{m,s}$.
			This, together with \eqref{eq4.14.1}, yields \eqref{eq4.14} for $l=1,\cdots,[\frac{n}{2m}]+1$.
			
			\subsection{Low energy estimate for $n>2m$.}\label{sec5.1.2}\
			
			Now we turn to the low energy estimates. Before proceeding, we highlight a key fact that will be applicable for all $n\ge 1:$
			
			\textbf{Observation.} All the results about $F^{\pm}(\lambda^{2m})=\langle R^{\pm}_{0}(\lambda^{2m})\varphi,\varphi\rangle$ in Lemma \ref{lm3.3}-\ref{lm3.7} are still valid for $f^{\pm}_{i,j}=\langle R^{\pm}_{0}(\lambda^{2m})\varphi_i, \varphi_j\rangle$.
			
			In fact, this follows immediately since each $\varphi_j$ ($j=1,\cdots, N$)  has the same decay assumption as $\varphi$ in the rank one case.
			
			Let $\lambda_0$ be given in Lemma \ref{lm3.3} here and in the rest of this section. By Remark \ref{rmk3.4}, it suffices to prove that for $0\le l\le [\frac{n}{2m}]+1$,
			\begin{equation}\label{eq4.15}
				\left | \frac{d^l}{d\lambda^l}g^{\pm}_{i,j}(\lambda^{2m})\right|\lesssim\lambda ^{-l},\quad 0<\lambda<\lambda_0,
			\end{equation}
			and
			\begin{equation}\label{eq4.16}
				\left | \frac{d^l}{d\lambda^l}\left [ g^{+}_{i,j}(\lambda^{2m})-g^-_{i,j}(\lambda^{2m}) \right ] \right|\lesssim\lambda ^{n-2m-l},\quad 0<\lambda<\lambda_0.
			\end{equation}
			In order to prove these, we note that by \eqref{eq-low-spec}, \eqref{equ-n>2m-F+-} and the above  \textbf{Observation}, we have
			\begin{equation}\label{eq4.18.1}
				\left | \frac{d^l}{d\lambda^l}f^{\pm}_{i,j}(\lambda^{2m})\right|\lesssim\lambda ^{-l}, \quad 0<\lambda<\lambda_0,
			\end{equation}
			and
			\begin{equation}\label{eq4.18.2}
				\left | \frac{d^l}{d\lambda^l}\left [ f^{+}_{i,j}(\lambda^{2m})-f^-_{i,j}(\lambda^{2m}) \right ] \right|\lesssim\lambda ^{n-2m-l},\quad 0<\lambda<\lambda_0.
			\end{equation}
			Therefore the property \eqref{eq4.18.1}, together with \eqref{eq4.14.2}, \eqref{eq4.14.3} and the assumption \eqref{eq-low-spec}, yields \eqref{eq4.15}.
			And
			\eqref{eq4.16} follows from \eqref{eq4.18.2} and the following identity
			\begin{equation*}
				A^+(\lambda^{2m})^{-1}-  A^-(\lambda^{2m})^{-1}=  A^+(\lambda^{2m})^{-1}\left(A^{-}(\lambda^{2m})-A^{+}(\lambda^{2m})\right) A^-(\lambda^{2m})^{-1}.
			\end{equation*}
			This proves the low energy part for $n>2m$.

			\subsection{Low energy estimate for 
				$n<2m$ and odd $n$.}\ \label{sec-fin-rank-odd-low}
			
			By \eqref{eq.2-low-odd}, the Hilbert-Smith property and the decay assumption of $\varphi$(see \eqref{eq1.2}),
			$I+P{R}_{0}^{\pm}({\lambda }^{2m})P$ has the following expansion on $PL^2$,	\begin{equation}\label{expansion-G2j}
				\left(I+P{R}_{0}^{\pm}({\lambda }^{2m})P\right)f=\displaystyle\sum_{j=0}^{m-\frac{n+1}{2}}{a}_{j}^{\pm} {\lambda }^{n-2m+2j}P{G}_{2j}Pf+T_0f+{\mathcal{R}}_{2m-n+1}^{\pm}(\lambda)f, \quad f\in \mathcal{S}(\mathbb{R}^n),
			\end{equation}
			where
			\begin{align*}
				\begin{split}
					&{G}_{2j}f=\int_{{\mathbb{R}}^{n}}^{}{|x-y|}^{2j}f(y)dy,\ {G}_{2m-n}f=\int_{{\mathbb{R}}^{n}}^{}{|x-y|}^{2m-n}f(y)dy,\\ 
					&{\mathcal{R}}_{2m-n+1}^{\pm}(\lambda)f=\lambda^{n-2m} \int_{{\mathbb{R}}^{n}}P{r}_{2m-n+1}^{\pm}(\lambda|x-y|)Pf(y)dy,
				\end{split}
			\end{align*}
			and
			\begin{equation*}	
				T_0=I+b_0PG_{2m-n}P.
			\end{equation*}
			
			In contrast to case $n>2m$, the estimates \eqref{eq4.15} and \eqref{eq4.16} are not valid in general, and it is intricately linked to the vanishing moment condition of $\varphi_j$ for $j=1,\cdots, N$. To be more precise, the validity of these estimates depends on the fulfillment of the following vanishing condition: 
			\begin{equation}\label{eq-van-1129}
				P(x^\beta)=0, \,\,\,\mbox{for all}\,\,  0\le |\beta| \le m-\frac{n+1}{2}.  
			\end{equation}
			Thus we further divide the proof into two cases.
			
			\noindent{\it Case 1: The vanishing condition \eqref{eq-van-1129} holds.}
			In this case, we have 
			\begin{equation}\label{eq-vanish-condition-2}
				I+P{R}_{0}^{\pm}({\lambda }^{2m})P=T_0+{\mathcal{R}}_{2m-n+1}^{\pm}(\lambda) 
			\end{equation}
			holds on $PL^2$.
			Further, by \eqref{eq3.5.0-rem} and the decay assumption \eqref{eq2}, we have 
			\begin{equation}\label{eq-est-for-rinl2}
				\left\|\frac{d^l}{dz^l}{\mathcal{R}}_{2m-n+1}^{\pm}(\lambda)\right\|_{L^2-L^2}\lesssim \lambda^{1-l},\quad \text{for}\,\ l=0,\,1,\,\cdots\frac{n+1}{2}.
			\end{equation}
			Note that $T_0$ coincides with $I+P(-\Delta)^{-m}P$ on $PL^2$. Since $I+P(-\Delta)^{-m}P$ is strictly positive on $PL^2$, $T_0$ is invertible, moreover,
			it follows from \eqref{eq-est-for-rinl2} and the Neumann series that there exists a small $\lambda_0$ such that $I+P{R}_{0}^{\pm}({\lambda }^{2m})P$ is also invertible on $PL^2$ and 
			\begin{equation}\label{eq-vanish-condition-2'}
				\left\| \frac{d^l}{d\lambda^l}\left( I+P{R}_{0}^{\pm}({\lambda }^{2m})P\right)^{-1}\right\|_{L^2} \lesssim \lambda^{-l}, \quad 0<\lambda<\lambda_0.  
			\end{equation}
			Therefore, by \eqref{eq-gij}, the function $g_{i,j}^\pm(\lambda^{2m})$ satisfies 
			\begin{equation}\label{eq-finite-odd-low-gij}
				\left| \frac{d^l}{d\lambda^l}g_{i,j}^\pm(\lambda^{2m})\right|_{L^2-L^2} \lesssim \lambda^{-l}, \quad 0<\lambda<\lambda_0, 
			\end{equation}
			which is consistent with the estimate \eqref{equ4-F2} in rank one perturbation case. This shows that the proof of the estimate for low energy part $K^{\pm, l}_{n, ij}(t,x,y)$ can be reduced to the rank one case (see the arguments in Section \ref{sec4.3}).
			
			\noindent{\it Case 2: The vanishing condition \eqref{eq-van-1129} does not hold.}
			In this case, $ I+P{R}_{0}^{\pm}({\lambda }^{2m})P$ does not have a simple form as \eqref{eq-vanish-condition-2}.
			The main difficulty lies in the asymptotic expansion  of $(I+P{R}_{0}^{\pm}({\lambda }^{2m})P)^{-1}$ on $PL^2(\mathbb{R}^n)$ near the zero energy threshold. Our approach to obtain the inverse is based on a suitable decomposition of space $PL^2( \mathbb{R}^n)$.
			To this end, we introduce some notations that will be used throughout this section. First, denote
			$$J=\{0,\, 1,...,m-\frac{n+1}{2},m-\frac{n}{2}\}.$$
			Now we define some subspaces of $P{L}^{2}$. More precisely, let $Q_j: L^2\to Q_jL^2$ be the orthogonal projections with
			\begin{equation}\label{equ-5.27}
				{Q}_{0}{L}^{2}:=\mbox{span}\left\{\varphi_i, 1\le i\le N: \int{\varphi_i\d x}\ne 0\right\},
			\end{equation}
			and for $1\le j\le m-\frac{n+1}{2}$,
			\begin{align}\label{equ-5.27-1129}
				{Q}_{j}{L}^{2}:=\mbox{span}\Big\{\varphi_i, 1\le i\le N: \mbox{there exists $\kappa$ with $|\kappa|=j$}\nonumber\\
				\mbox{such that $\int{x^{\kappa}\varphi_i\d x}\ne 0$ and $\int{x^{\beta}\varphi_i\d x}=0$ for $|\beta|<j$}\Big\},
			\end{align}  
			finally,
			\begin{equation}\label{eq-projection-Qj}
				{Q}_{m-\frac n2}L^2:=\mbox{span}\left\{\varphi_i, 1\le i\le N: \int{x^{\beta}\varphi_i\d x}=0 \,\,\mbox{for all\,\, $|\beta|\le m-\frac{n+1}{2}$}\right\}.
			\end{equation}
			
			From the above definitions, we have the following decomposition \begin{equation}\label{eq-QjL2-1}
				PL^2={\bigoplus }_{j\in J}Q_jL^2,
			\end{equation}
			as well as the relation
			\begin{equation}\label{eq-QjP=Qj}
				Q_jP=Q_j.
			\end{equation}	
			We point out that the vanishing moment condition of $\varphi_j$ may vary from one to one.
			The following orthogonal property associated with $Q_j$ is needed.
			
			\begin{lemma}\label{lm-finiterank-5.1}
				For $i,j\in J$ and $l\in {\mathbb{N}}_{0}$, we have 		\begin{equation}\label{eq.5.30}
					Q_i{G}_{2l}Q_j=0,
				\end{equation}
				provided $2l\le [i+\frac{1}{2}]+[j+\frac{1}{2}]-1$.
			\end{lemma}
			\begin{proof}
				By the definitions of $Q_j$ above, we deduce that
				\begin{equation}\label{eq-1106-h1}
					Q_j(x^\alpha)=0,\,\,\,\mbox{if}\,\,\, |\alpha|\le [j+\frac{1}{2}]-1.   
				\end{equation}
				Since
				\begin{equation*}\label{eq-expansion-xy}
					|x-y|^{2l}=\sum_{|\alpha|+|\beta|=2l}C_{\alpha,\beta}(-1)^{|\beta|} x^{\alpha}y^{\beta}, \quad l\in \N_0,
				\end{equation*}
				we have
				\begin{equation*}
					Q_i{G}_{2l}Q_jf =  \sum_{|\alpha|+|\beta|=2l}C_{\alpha,\beta}(-1)^{|\beta|} Q_i(x^\alpha)\int_{\mathbb{R}^n}\big(Q_j(x^\beta)\big)(y)f(y)\d y.
				\end{equation*}
				Given that $|\alpha|+|\beta|=2l$, if $2l\le [i+\frac{1}{2}]+[j+\frac{1}{2}]-1$, then either $ |\alpha|\le [j+\frac{1}{2}]-1$ or $ |\beta|\le [j+\frac{1}{2}]-1$. Therefore \eqref{eq.5.30} follows from \eqref{eq-1106-h1}.
			\end{proof}
			
			
			\begin{proposition}\label{Thm-I+PRP}
				There exists a sufficiently small $\lambda_0\in (0, 1)$ such that for all $0<\lambda<\lambda_0$, the operator $I+P{R}_{0}^{\pm}({\lambda }^{2m})P$ is invertible on $PL^2$ and 	\begin{align}\label{eq-expansion-for-M-odd}
					(I+P{R}_{0}^{\pm}({\lambda }^{2m})P)^{-1}=\displaystyle\sum_{i,j\in J}^{}\lambda ^{2m-n-i-j}Q_i\Big(M_{i,j}^{\pm}+{\Gamma}_{i,j}^{\pm}(\lambda )\Big)Q_j,
				\end{align}
				where $Q_j$ ($j\in J$) is given by \eqref{equ-5.27}-\eqref{eq-projection-Qj}, and $M_{i,j}^{\pm}$, ${\Gamma}_{i,j}^{\pm}(\lambda)$ are bounded operators on $L^2$. Moreover, 
				\begin{align}\label{eq5.36}
					\left\|\frac{d^l}{d\lambda^l} {\Gamma}_{i,j}^{\pm}(\lambda)\right\|_{L^2-L^2}\lesssim\lambda^{\frac12-l}, \quad l=0,\,1,\,\cdots\frac{n+1}{2}.
				\end{align}
			\end{proposition}
			\begin{proof}
				The proof is based on the decomposition of the space $PL^2$ in \eqref{eq-QjL2-1}.
				More precisely, we define $B=(\lambda ^{-j}Q_j)_{j\in J}\ $: ${\bigoplus }_{j\in J}Q_jL^2\ \rightarrow L^2$ by 
				$$Bf=\displaystyle\sum_{j\in J}^{}\lambda ^{-j}Q_jf_j,$$
				where $f=(f_j)_{j\in J}\in {\bigoplus }_{j\in J}Q_jL^2$.  Without loss of generality, we assume that  $Q_j\ne0$ for all $j\in J$. Otherwise, let $J'$ be the set of indices $j$ for which $Q_j\ne 0$ and it suffices to replace $J$ by $J'$ in the following proof. 
				
				Let $B^*$ be the dual operator of $B$ and define $\mathbb{A}^{\pm}$ on ${\bigoplus }_{j\in J}Q_jL^2$ by
				\begin{align}\label{eq-A}
					\mathbb{A}^{\pm}=\lambda ^{2m-n}B^*(I+P{R}_{0}^{\pm}({\lambda }^{2m})P)B.
				\end{align}
				Notice that $B$ is a bijection between ${\bigoplus }_{j\in J}Q_jL^2$ and $PL^2$. Thus $I+P{R}_{0}^{\pm}({\lambda }^{2m})P$ is invertible on $PL^2$  if and only if $(\mathbb{A}^{\pm})^{-1}$ exists. Moreover, we have
				\begin{equation}\label{equ-I+PRP}
					(I+P{R}_{0}^{\pm}({\lambda }^{2m})P)^{-1}=\lambda ^{2m-n}B(\mathbb{A}^{\pm})^{-1}B^*,
				\end{equation}
				provided that the inverse exists.	
				
				Now we prove the existence of $(\mathbb{A}^{\pm})^{-1}$. By \eqref{expansion-G2j} and  the assumption in {\it Case 2}, we obtain
				\begin{equation}\label{eq-5.27}
					I+P{R}_{0}^{\pm}({\lambda }^{2m})P=\displaystyle\sum_{j=0}^{m-\frac{n+1}{2}}{a}_{j}^{\pm} {\lambda }^{n-2m+2j}P{G}_{2j}P+T_0+P{\mathcal{R}}_{2m-n+1}^{\pm}(\lambda)P.
				\end{equation}
				By Lemma \ref{lm-finiterank-5.1}, we have
				\begin{equation*}
					\begin{aligned}
						&\lambda^{-i}Q_i(I+P{R}_{0}^{\pm}({\lambda }^{2m})P)\lambda^{-j}Q_j\\
						=& \displaystyle\sum_{l=O_{i,j}}^{m-\frac{n+1}{2}}{a}_{l}^{\pm} {\lambda }^{n-2m+2l-i-j}Q_i{G}_{2j}Q_j+\lambda^{-i-j}Q_iT_0Q_j+\lambda^{-i-j}Q_i{\mathcal{R}}_{2m-n+1}^{\pm}(\lambda)Q_j,    
					\end{aligned}
				\end{equation*}
				where $O_{i,j}=[\frac{[i+\frac{1}{2}]+[j+\frac{1}{2}]}{2}]+1$ .
				Combining this and \eqref{eq-A}, we have
				\begin{equation}\label{eq-5.28}	\mathbb{A}^{\pm}=D^{\pm}+{({r}_{i,j}^{\pm}(\lambda))}_{i,j\in J},
				\end{equation}
				where
				$D^{\pm}$ is given by 
				\begin{equation}\label{equ3.64}
					d_{i,j}^\pm=\begin{cases}
						a_{\frac{i+j}{2}}^\pm Q_iG_{i+j}Q_j,\quad&\mbox{if}\,\,i+j\,\,\mbox{is even},\\
						Q_{m-\frac{n}{2}}T_0Q_{m-\frac{n}{2}},&\mbox{if}\,\,i=j=m-\frac{n}{2},\\
						0,&\mbox{else},
					\end{cases}
				\end{equation}
				and by \eqref{eq-est-for-rinl2} and \eqref{eq-5.27}-\eqref{eq-5.28}, for $0\le l\le \frac{n+1}{2}$, we have 
				\begin{align}\label{eq-rem-D}
					\left\|\frac{d^l}{d\lambda^l} {r}_{i,j}^{\pm}(\lambda)\right\|_{L^2-L^2}\lesssim \lambda^{\frac{1}{2}-l}, \quad \text{for}\ \lambda\in (0,\, 1).
				\end{align}
				Observe that $D^{\pm}$ can be written as the following matrix:
				\begin{equation*}
					D^{\pm}=\begin{pmatrix}
						{D}_{0}^\pm  & 0\\[0.2cm]
						0 & Q_{m-\frac{n}{2}}T_0Q_{m-\frac{n}{2}}
					\end{pmatrix}.
				\end{equation*}
				
				To proceed, we need the following 
				
				\underline{\it Claim 1}: $D^{\pm}$ is invertible on ${\bigoplus }_{j\in J}Q_jL^2$.
				
				\noindent {\it Proof of \underline{Claim 1}:} 
				First, we  note that only $d_{m-\frac{n}{2}, m-\frac{n}{2}}^\pm \neq 0$ exists in the last row or last column of $D^{\pm}$. Similar to \eqref{eq-equal-laplacian-m}, $Q_{m-\frac{n}{2}}T_0Q_{m-\frac{n}{2}} =I+ Q_{m-\frac n2}(-\Delta)^{-m}Q_{m-\frac n2}$  is invertible on $Q_{m-\frac n2}L^2$. Thus, to obtain the invertibility of $D^{\pm}$, it suffices to show the invertibility of $D_0^\pm ={({d}_{i,j}^{\pm})}_{i,j\in J\setminus\{m-\frac n2\}}$ on ${\bigoplus }_{j\in J\setminus\{m-\frac n2\}}Q_jL^2$.
				
				Next, we denote by $\sigma_j=\mbox{dim} Q_jL^2$ and let $\{\varphi^j_{i};\ i=0,1, \cdots, \sigma_j\}$ be the orthonormal basis of $Q_jL^2$, where $j\in\{0,1,\cdots,m-\frac{n+1}{2}\}$. From the definition of $Q_j$, we have 
				$$ \sigma_j\le \min\left\{N,\, \#\{\alpha \in \N_0^n;\, |\alpha|=j\}\right\}.$$
				Clearly, the invertibility of matrix operator $D_0^\pm$ is equivalent to the non-singularity of the matrix $$\widetilde{D}_0^\pm=\big(\, \widetilde{d}_{i,j}^{\pm}\, \big)_{i,j\in J\setminus\{m-\frac n2\}},$$ where the blocks $\widetilde{d}_{i,j}^{\pm}$ are defined as 
				$$\widetilde{d}_{i,j}^{\pm}=\Big(\big\langle d_{i,j}\varphi^j_{s}, \varphi^i_{t} \big\rangle\Big)_{\substack{0\le s \le \sigma_i \\ 0\le t \le  \sigma_j}}\, , \quad \text{for}\ i, \, j\in J\setminus\{m-\frac n2\}.  $$ 
				From \eqref{eq-expansion-xy} and the definition of ${d}_{i,j}^{\pm}$, if $i+j$ is even, we have that
				\begin{equation}
					\begin{aligned}
						\big\langle d_{i,j}\varphi^j_{s}, \varphi^i_{t} \big\rangle&=a_{\frac{i+j}{2}}^\pm \sum_{|\alpha|+|\beta|{i+j}}C_{\alpha,\beta}(-1)^{|\beta|}\int_{\mathbb{R}^n} x^{\alpha} \varphi^i_{s}(x)\d x \int_{\mathbb{R}^n} x^{\beta} \varphi^j_{t}(x)\d x\\
						&= \sum_{|\alpha|=i, \, |\beta|=j} A_{\alpha,\beta}^{\pm}\int_{\mathbb{R}^n} x^{\alpha} \varphi^i_{s}(x)\d x \int_{\mathbb{R}^n} x^{\beta} \varphi^j_{t}(x)\d x. \\   \end{aligned}
				\end{equation}
				Thus, the matrices $\widetilde{D}_0^\pm$ now can be expressed as
				$$
				\widetilde{D}_0^\pm=\text{diag}\big( B_i\big)_{i\in J\setminus\{m-\frac n2\}}\Big(A_{\alpha, \beta}^\pm\Big)_{|\alpha|, |\beta|\in J\setminus\{m-\frac n2\}}\text{diag}( B_i\big)_{i\in J\setminus\{m-\frac n2\}}^*,
				$$
				where the block $B_i$ is defined as
				$$B_i= \Big(\int_{\mathbb{R}^n} x^{\alpha} \varphi^i_{s}(x)\d x \Big)_{s=0,\cdots, \sigma_i,\,|\alpha|=i}$$
				which is a $(\sigma_i+1)\times \#\{\alpha \in \N_0^n;\, |\alpha|=i\}$ submatrix.
				We claim that  $\mbox{rank} (B_i)=\sigma_i+1$.
				In fact, it suffices to show that the vectors $\Big(\int_{\mathbb{R}^n} x^{\alpha} \varphi^i_{s}(x)\d x \Big)_{|\alpha|=i}$, $s=0,1, \cdots, \sigma_i$ are linearly independent. If there exist $t_0, \, t_1, \, \cdots, t_{\sigma_i}\in \mathbb{R}$ such that
				$$\sum_{s=0}^{\sigma_i}t_s\int_{\mathbb{R}^n} x^{\alpha} \varphi^i_{s}(x)\d x= \int_{\mathbb{R}^n} x^{\alpha} \left(\sum_{s=0}^{\sigma_i}t_s\varphi^i_s(x)\right)\d x=0, \quad \text{for all} \ |\alpha|=i,$$
				this, together with the definition of $Q_i$, implies that $\sum_{s=0}^{\sigma_i}t_s\varphi^i_s(x)\in (Q_iL^2)^{\perp}$ in $PL^2$. On the other hand, since $\varphi^i_s(x)\in Q_iL^2$,  we have $\sum_{s=0}^{\sigma_i}t_s\varphi^i_s(x)\in Q_iL^2$. Thus  it follows that $\sum_{s=0}^{\sigma_i}t_s\varphi^i_s(x)=0$. Moreover, since $\{\varphi^i_{s};\ s=0,1, \cdots, \sigma_i\}$ is the orthonormal basis of $Q_jL^2$, we derive that $t_0=t_1=\cdots=t_{\sigma_i}=0$. Hence the claim is proved.
				
				The above claim shows that $\text{diag}( B_i\big)_{i\in J\setminus\{m-\frac n2\}}$ is a row full rank matrix. Thus if the matrices $\Big(A_{\alpha, \beta}^\pm\Big)_{|\alpha|, |\beta|\in J\setminus\{m-\frac n2\}}$ are positive definite, then the  matrices $\widetilde{D}_0^\pm$ are  nonsingular. Note that by definition \eqref{eq-def-A} of $A_{\alpha, \beta}$, we  have 
				$$A_{\alpha, \beta}^\pm=A_{\alpha, \beta}  e^{\mp \frac{\i\pi(n-2m+|\alpha|+|\beta|)}{m}}.$$
				This yields that
				$$\widetilde{D}_0^\pm=E_0^\pm \big(A_{\alpha, \beta}\big)_{0\le |\alpha|, |\beta|\le m-\frac{n+1}{2}}E_0^\pm,$$
				where $E_0^\pm=\text{diag}\big( e^{\mp \frac{\i\pi(\frac n2-m+|\alpha|)}{m}}\big)_{0\le |\alpha|\le m-\frac{n+1}{2}}$. Thus, $\widetilde{D}_0^\pm$ is congruent to $\big(A_{\alpha, \beta}\big)_{|0\le |\alpha|, |\beta|\le m-\frac{n+1}{2}}$. Meanwhile, it follows from \eqref{eq5.48} that $\big(A_{\alpha, \beta}\big)_{0\le |\alpha|, |\beta|\le m-\frac{n+1}{2}}$ is the Gram matrix with respect to  the following vectors,
				$$\left\{\frac{\xi^\alpha}{(2\pi)^{\frac{n}{2}}\alpha!(1+|\xi|^{2m})^\frac12}\right\}_{0\le |\alpha|\le m-\frac{n+1}{2}}.$$
				This implies that the matrix $\big(A_{\alpha, \beta}\big)_{0\le |\alpha|, |\beta|\le m-\frac{n+1}{2}}$ is strictly positive definite, which shows that $\widetilde{D}_0^\pm$  is non-singular. 
				
				By \eqref{eq-5.28}, \eqref{eq-rem-D} and the above \underline{{\it Claim 1}},  we can apply the Neumann series to deduce that 
				there exists a small $\lambda_0>0$ such that $\mathbb{A}^{\pm}$ is invertible on ${\bigoplus}_{j\in J}Q_jL^2$ for $\lambda \in (0,\, \lambda_0)$ and 
				\begin{equation*}
					\big(\mathbb{A}^{\pm}\big)^{-1}=\left(I+(D^{\pm})^{-1}{({r}_{i,j}^{\pm}(\lambda))}_{i,j\in J}\right)^{-1} (D^{\pm})^{-1}=\sum_{l=0}^{\infty}\left(-(D^{\pm})^{-1}{({r}_{i,j}^{\pm}(\lambda))}_{i,j\in J}\right)^l (D^{\pm})^{-1}.   
				\end{equation*}
				This, together with \eqref{equ-I+PRP}, implies \eqref{eq-expansion-for-M-odd} and the  estimate \eqref{eq5.36}. Thus we complete the proof of the proposition.
			\end{proof}

			Thanks to Proposition \ref{Thm-I+PRP}, we are able to reduce the analysis to the rank one case. Indeed, since $P=\sum_{j\in J}Q_j$, then by \eqref{eq-expansion-for-M-odd}, there exists a small $\lambda_0$ such that for $0<\lambda<\lambda_0$, 
			\begin{equation}\label{eq-reso-expansion-new}
				\begin{aligned}
					R^{\pm}(\lambda^{2m})&= R^{\pm}_{0}(\lambda^{2m})- R^{\pm}_{0}(\lambda^{2m})P\Big(I+PR^{\pm}_{0}(\lambda^{2m})P\Big)^{-1} PR^{\pm}_{0}(\lambda^{2m}) \\  
					&= R^{\pm}_{0}(\lambda^{2m})-\sum_{i,j\in J}^{}\lambda ^{2m-n-i-j} R^{\pm}_{0}(\lambda^{2m})Q_i\Big(M_{i,j}^{\pm}+{\Gamma}_{i,j}^{\pm}(\lambda )\Big)Q_jR^{\pm}_{0}(\lambda^{2m}) \\
					&=R^{\pm}_{0}(\lambda^{2m})-\sum_{i,j\in J}\sum_{s=0}^{\sigma_i}\sum_{l=0}^{\sigma_j}g_{i,j}^{s,l,\pm}(\lambda) R^{\pm}_{0}(\lambda^{2m})\varphi^{i}_{s} \big\langle R^{\pm}_{0}(\lambda^{2m}) \cdot, \, \varphi^{j}_{l}\big\rangle,
				\end{aligned}
			\end{equation}
			where $\{\varphi^{j}_{l}\}_{0\le l\le \sigma_j}$  is the orthonormal basis of $Q_jL^2$ for each $j\in J$, and 
			$$g_{i,j}^{s,l,\pm}(\lambda)=\lambda ^{2m-n-i-j} \big\langle \Big(M_{i,j}^{\pm}+{\Gamma}_{i,j}^{\pm}(\lambda )\Big)\varphi_{i}^{s}, \, \varphi_{j}^{l}\big\rangle.$$
			\eqref{eq-reso-expansion-new} plays the same role as \eqref{A-K-for} in the  rank one case.
			Furthermore, 	by \eqref{eq5.36}, we have
			$$\left|\frac{d^k}{d\lambda^k} g_{i,j}^{s,l,\pm}(\lambda)\right|_{L^2}\lesssim\lambda^{2m-n-i-j-k}, \quad 0\le k\le \left [ \frac{n}{2m}\right ] +1\le\frac{n+1}{2},$$
			which is parallel to \eqref{equ4-F1}. Besides, observe that
			we have $$\int_{\R^n}x^{\alpha}\varphi^{j}_{l}(x) \d x=0,\qquad\mbox{for}\,\,|\alpha|<j,
			$$
			then by Proposition \ref{lm-n<2m-odd-vanish},  \eqref{eq3.33.0} and  \eqref{eq3.33.1} hold with $\varphi$ replaced by $\varphi^{j}_{l}$ and $\k_0$ replaced by $j$.
			Therefore, plugging \eqref{eq-reso-expansion-new} into  \eqref{eq-evolution-1}, we can reduce analysis of the estimate for low energy part $K^\pm_{n,N}(t,x,y)$ to the rank one case.
			
			
			\subsection{Low energy estimate for $n\le2m$ and  even $n$} \label{sec-fin-rank-odd-even}\
			
			Our approach is similar to that used in the previous subsection. However, the expression of $I+P{R}_{0}^{\pm}({\lambda }^{2m})P$ is more complicated due to the $\log$ factor in the resolvent. Indeed, by \eqref{eq.2-low-even}, it follows that
			\begin{equation}\label{expansion-G2j-even}
				I+P{R}_{0}^{\pm}({\lambda }^{2m})P=\displaystyle\sum_{j=0}^{m-\frac{n}{2}}{a}_{j}^{\pm} {\lambda }^{n-2m+2j}P{G}_{2j}P+b_0\log\lambda PG_{2m-n}P+T_0+{\mathcal{R}}_{2m-n+1}^{\pm}(\lambda)
			\end{equation}
			holds in  $PL^2$, where
			\begin{align*}
				{\mathcal{R}}_{2m-n+1}^{\pm}(\lambda)f=\lambda^{n-2m} \int_{{\mathbb{R}}^{n}}P{r}_{2m-n+1}^{\pm}(\lambda|x-y|)Pf(y)dy, \quad f\in \mathcal{S}(\mathbb{R}^n),
			\end{align*}
			and
			\begin{equation*}
				T_0f=I+b_0P\int_{{\mathbb{R}}^{n}}{|x-y|}^{2m-n}\log|x-y|Pf(y)dy, \quad f\in \mathcal{S}(\mathbb{R}^n).
			\end{equation*}
			Moreover, the remainder term ${\mathcal{R}}_{2m-n+1}^{\pm}(\lambda)$ also satisfies \eqref{eq-est-for-rinl2}.
			
			Now, similar to \eqref{eq-van-1129}, we consider the following vanishing condition:
			\begin{equation}\label{vanish-finiterank-even}
				P(x^\beta)=0\ \mbox{holds for all}\ 0\le |\beta| \le m-\frac{n}{2}.
			\end{equation}
			
			Now we divide the proof into two cases:
			
			{\it Case 1: The vanishing condition \eqref{vanish-finiterank-even} holds}.
			By repeating the arguments in \eqref{eq-vanish-condition-2}-\eqref{eq-vanish-condition-2'}, we deduce that 
			\begin{equation*}
				\left| \frac{d^l}{d\lambda^l}g_{i,j}^\pm(\lambda^{2m})\right|_{L^2-L^2} \lesssim \lambda^{-l}, \quad 0<\lambda<\lambda_0, 
			\end{equation*}
			which is consistent with \eqref{eq3.6.3} in rank one perturbation case, so we can reduce the proof of the estimate for low energy part $K^{\pm, l}_{n, ij}(t,x,y)$ to the rank one case (see the argument in Section \ref{sec4.3}). 
			
			{\it Case 2: The vanishing condition \eqref{vanish-finiterank-even} does not hold}. In order to obtain a suitable decomposition of $P{L}^{2}$ in this case, we denote
			$$\widetilde{J}=\{j: 0\le j\le m-\frac{n}{2}+1\}.$$
			Then we define ${Q}_{j}$ by
			\begin{equation}\label{equ-2626}
				{Q}_{0}{L}^{2}:=\mbox{span}\left\{\varphi_i: \int{\varphi_i\d x}\ne 0, 1\le i\le N\right\},
			\end{equation}
			and for $1\le j\le m-\frac{n}{2}$,
			\begin{align}\label{equ-2626'}
				{Q}_{j}{L}^{2}:=\mbox{span}\left\{\varphi_i: \mbox{there exists $\kappa$ with $|\kappa|=j$ such that } \int{x^{\kappa}\varphi_i\d x}\ne 0\right.\nonumber\\
				\left. \mbox{ and} \int{x^{\beta}\varphi_i\d x}=0\,\, \mbox{for all} \,\,|\beta|<j, 1\le i\le N\right\},
			\end{align}  
			finally,
			\begin{equation}\label{equ-2626''}
				{Q}_{m-\frac n2 +1}L^2:=\mbox{span}\left\{\varphi_i: \int{x^{\beta}\varphi_i\d x}=0 \,\,\mbox{for all\,\, $|\beta|\le m-\frac{n}{2}$}, 1\le i\le N\right\}.
			\end{equation}
			It immediately follows that 
			\begin{equation}\label{eq-QjL2-1-even}
				PL^2={\bigoplus }_{j\in \widetilde{J}}Q_jL^2,
			\end{equation}
			and
			\begin{equation}\label{eq-QjP=Qj-even}
				Q_jP=Q_j.
			\end{equation}
			
			We establish the following result which is analogous to Proposition \ref{Thm-I+PRP}.
			\begin{proposition}\label{Thm-I+PRP-even}
				There exists a sufficiently small $\lambda_0>0$ such that for all $0<\lambda<\lambda_0$, the operator $I+P{R}_{0}^{\pm}({\lambda }^{2m})P$ is invertible on $PL^2$ and 
				\begin{align}\label{expansion-even}
					(I+P{R}_{0}^{\pm}({\lambda }^{2m})P)^{-1}=\displaystyle\sum_{i,j\in \widetilde{J}}^{}\lambda ^{2m-n}\eta_i(\lambda)\eta_j(\lambda)Q_i\Big(M_{i,j}^{\pm}+{\Gamma}_{i,j}^{\pm}(\lambda )\Big)Q_j,
				\end{align}
				where $Q_j$ is defined in \eqref{eq-projection-Qj},
				\begin{equation}\label{eq-etaj-even}
					\eta_i(\lambda)=
					\begin{cases}
						\lambda^{-i}, &\quad \mbox{if} \,\,~ \k_0\le i<m-\frac n2,\\ 
						\lambda^{\frac n2-m} \left(a_{m-\frac n2}^\pm+b_0\log\lambda \right)^{-\frac 12}, &\quad \mbox{if} \,\,~  i=m-\frac n2,\\ 
						\lambda^{\frac n2-m}, &\quad \mbox{if}\,\, ~  i=m-\frac n2+1,
					\end{cases}
				\end{equation}
				$a_{m-\frac n2}^\pm,\, b_0$ are given in \eqref{expansion-G2j-even}, $M_{i,j}^{\pm}$, ${\Gamma}_{i,j}^{\pm}(\lambda)$ are bounded operators on $L^2$. Furthermore,  for $0\le l\le\frac{n}{2}+1$, we have 
				\begin{align}\label{eq5.36-even}
					\left\|\frac{d^l}{d\lambda^l} {\Gamma}_{i,j}^{\pm}(\lambda)\right\|_{L^2-L^2}\lesssim |\log\lambda|^{-\frac12} \lambda^{-l}, \qquad 0<\lambda<\lambda_0.
				\end{align}
			\end{proposition}
			\begin{proof}
				In the same spirit of  the proof for Theorem \ref{Thm-I+PRP},  we assume that  $Q_j\ne0$ for all $j\in \widetilde{J}$ and  define $B=(\lambda ^{-j}Q_j)_{j\in \widetilde{J}}\ $: ${\bigoplus }_{j\in \widetilde{J}}Q_jL^2\ \rightarrow L^2$ by 
				$$Bf=\displaystyle\sum_{j\in \widetilde{J}}\eta_j(\lambda) Q_jf_j,$$
				where $f=(f_j)_{j\in \widetilde{J}}\in {\bigoplus }_{j\in \widetilde{J}}Q_jL^2$. 
				Then, denote 
				\begin{align*}
					\mathbb{A}^{\pm}=\lambda ^{2m-n}B^*(I+P{R}_{0}^{\pm}({\lambda }^{2m})P)B.
				\end{align*}
				From which, we derive that $I+P{R}_{0}^{\pm}({\lambda }^{2m})P$ is invertible on $PL^2$ and
				$$(I+P{R}_{0}^{\pm}({\lambda }^{2m})P)^{-1}=\lambda ^{2m-n} B(\mathbb{A}^{\pm})^{-1}B^*,$$
				provided $(\mathbb{A}^{\pm})^{-1}$ exists. Similar to the arguments in \eqref{eq-5.27}-\eqref{eq-rem-D}, we have
				$$\mathbb{A}^{\pm}=D^{\pm}+{({r}_{i,j}^{\pm}(\lambda))}_{i,j\in \widetilde{J}}.$$
				where $D^{\pm}={({d}_{i,j}^{\pm})}_{i,j\in \widetilde{J}}$ and ${({r}_{i,j}^{\pm}(\lambda))}_{i,j\in \widetilde{J}}$ are  operators on ${\bigoplus }_{j\in \widetilde{J}}Q_jL^2$, and $D^{\pm}$ is given by 
				\begin{equation*}
					d_{i,j}^\pm=\begin{cases}
						a_{\frac{i+j}{2}}^\pm Q_iG_{i+j}Q_j,\quad&\mbox{if}\,\,i+j\,\,\mbox{is even and $i,\, j< m-\frac n2$},\\
						Q_{m-\frac{n}{2}}G_{2m-n}Q_{m-\frac{n}{2}},&\mbox{if}\,\,i=j=m-\frac n2,\\
						Q_{m-\frac{n}{2}+1}T_0Q_{m-\frac{n}{2}+1},&\mbox{if}\,\,i=j=m-\frac n2+1,\\
						0,&\mbox{else},
					\end{cases}
				\end{equation*}
				furthermore, ${r}_{i,j}^{\pm}(\lambda)$ satisfies that for $0\le l\le\frac{n}{2}+1$,
				\begin{align*}
					\left\|\frac{d^l}{d\lambda^l} {r}_{i,j}^{\pm}(\lambda)\right\|_{L^2}\lesssim |\log \lambda|^{-\frac 12} \lambda^{-l}, \quad  0<\lambda<\frac12.
				\end{align*}
				Thus, the problem is also reduced to  the invertibility of $D^{\pm}$, which can be expressed as
				\begin{equation*}
					D^{\pm}=\begin{pmatrix}
						{D}_{0}^\pm  & 0 & 0\\[0.2cm]
						0 & Q_{m-\frac{n}{2}}G_{2m-n}Q_{m-\frac{n}{2}} & 0\\[0.2cm]
						0 & 0 & Q_{m-\frac{n}{2}+1}T_0Q_{m-\frac{n}{2}+1}
					\end{pmatrix}.
				\end{equation*}
				Note that $Q_{m-\frac{n}{2}+1}T_0Q_{m-\frac{n}{2}+1}=Q_{m-\frac{n}{2}+1}(I+(-\Delta)^{-m})Q_{m-\frac{n}{2}+1}$ is invertible on $Q_{m-\frac{n}{2}+1}L^2$. For the operator $Q_{m-\frac{n}{2}}G_{2m-n}Q_{m-\frac{n}{2}}$,  since $Q_{m-\frac{n}{2}}(x^\alpha)=0$ for $|\alpha|<m-\frac n2$, it follows that for all $f\in Q_{m-\frac{n}{2}}L^2$, 
				\begin{equation*}
					\begin{aligned}
						\left\langle Q_{m-\frac{n}{2}}G_{2m-n}Q_{m-\frac{n}{2}}f,\,f  \right\rangle&= \left\langle\int_{\R^n}|\cdot-y|^{2m-n}f(y)\d y,\, f(\cdot)\right\rangle\\
						&=\left\langle\int_{\R^n}\sum_{l=0}^{[\frac m2-\frac n4]}C_{l}(-2xy)^{m-\frac n2-2l}|x|^{2l}|y|^{2l}f(y)\d y,\, f(x)\right\rangle\\
						&=\sum_{l=0}^{[\frac m2-\frac n4]}C_{l}(-2)^{m-\frac n2-2l}\sum_{|\alpha|\le m-\frac n2-2l} C_{\alpha} \left|\int_{\R^n} x^\alpha|x|^{2l}f(x)\d x \right|^2,
					\end{aligned}
				\end{equation*}
				where $C_l,\,C_{\alpha}>0$. Note that $(-2)^{m-\frac n2-2l}$ remains the same sign for all $l$. Thus if $\left\langle Q_{m-\frac{n}{2}}G_{2m-n}Q_{m-\frac{n}{2}}f,\,f  \right\rangle=0$, then 
				$$\left|\int_{\R^n} x^\alpha f(x)\d x \right|=0, \quad \text{for all} \,\, |\alpha|\le m-\frac n2,$$
				which indicates that $f=0$. Thus, $Q_{m-\frac{n}{2}}G_{2m-n}Q_{m-\frac{n}{2}}$ is invertible on $Q_{m-\frac{n}{2}}L^2$. Therefore it suffices to prove the invertibility of ${D}_{0}^\pm$. Following the same arguments as in Theorem \ref{Thm-I+PRP},  we obtain that ${D}_{0}^\pm$ is also invertible. Hence, the proof is complete. 
			\end{proof}
			Thanks to Proposition \ref{Thm-I+PRP-even}, we can proceed in the same manner as in Section \ref{sec-fin-rank-odd-low} and omit the detail.
			Therefore the proof is complete.
			
		
		\section{Proof of Corollary \ref{cor-Lp}}\label{sec6}
		
		We first consider the case that $({\frac 1p},{\frac 1q})$ lies in the line segment BC but $(p,q)\ne(1, \tau_m)$, i.e., $p=1$ and $q\in(\tau_m, +\infty]$ with $\tau_m=\frac{2m-1}{m-1}$. In this case one checks that $\frac{n(m-1)}{2 m-1}q>n$.
		If we set
		\begin{equation*}
			I(t,x-y)=|t|^{-\frac{n}{2 m}}\left(1+|t|^{-\frac{1}{2 m}}|x-y|\right)^{-\frac{n(m-1)}{2 m-1}},
		\end{equation*}
		then by Corollary \ref{cor-finite-rank}  and Young's inequality we have
		\begin{align*}
			\|e^{-\i tH_N}P_{ac}(H_N)\|_{L^1-L^q}&\lesssim\|I(t,\cdot)\|_{L^q}\\
			&\lesssim Ct^{-\frac{n}{2m}}\left(\int_{\mathbb{R}^n}{(1+|t|^{-1/{2m}}|x|)^{-\frac{n(m-1)}{2 m-1}q}\,dx}\right)^{\frac 1q}\lesssim |t|^{-\frac{n}{2m}(1-\frac{1}{q})}.
		\end{align*}
		At the endpoint $(p,q)=(1,\tau_m)$, the above estimate with $L^{\tau_m}$ replaced by $L^{\tau_m,\infty}$ (or $L^1$ replaced by $H^1$) follows from the weak Young's inequality (see e.g. \cite[p.22]{gra}) and the boundedness of the Riesz potential $f\mapsto|\cdot|^{-\frac{n(m-1)}{2 m-1}}*f$. 
		
		Next, if $(\frac1p,\frac1q)$ lies in the interior of triangle ABC, we deduce \eqref{equ-Lp-Lq} from the Riesz-Thorin interpolation and the fact that
		$\| e^{-\i tH_N}\|_{L^2-L^2}=1$.
		On the other hand, when $(\frac1p,\frac1q)$ lies in the edge AB, the above estimate follows from the Marcinkiewicz interpolation theorem (see e.g. \cite[p.31]{gra}). Therefore estimate \eqref{equ-Lp-Lq} is valid if $(\frac1p,\frac1q)$ lies in the triangle $ABC$, while the case of triangle $ADC$ follows by duality.

		\begin{appendix}\label{app-1}
			
			\section{Proof of the lemmas in Section \ref{sec2.2}}\label{app-01}
			
			We first establish the following asymptotic expansions for ${R}_0^{\pm}(\lambda^{2m})(x-y)$ when $\lambda|x-y|$ is small. 
			\begin{lemma}\label{lem-expan-R- small}
				Assume that  $0<\lambda|x-y|<\lambda_0$ for some small constant $\lambda_0>0$.
				
				\noindent(i) When $n$ is odd, we have
				\begin{equation}\label{eq-expansion-A-odd}
					{R}_0^{\pm}(\lambda^{2m})(x-y)=\lambda^{n-2m}\sum_{l=1}^{[\frac{n-1}{2m}]}b_{l-1}(\lambda |x-y|)^{2ml-n}+\lambda^{n-2m}\sum_{j=0}^{\infty}G_j^{1,\pm}(\lambda|x-y|),
				\end{equation}
				where
				\begin{align*}
					G_j^{1,\pm}=
					\begin{cases}
						a_{j/2}^\pm{(\lambda|x-y|)}^{j},&\text{if $j$ is even},\\[4pt]
						b_{\frac{j+n}{2m}-1}(\lambda|x-y|)^{j},&\text{if}\,\,  j+n\in 2m\Z,\\[4pt]
						0,&\text{else}.\\[4pt]
					\end{cases}
				\end{align*}
				(ii) When $n$ is even, we have
				\begin{equation}\label{eq-expansion-A-2}
					{R}_0^{\pm}(\lambda^{2m})(x-y)=\lambda^{n-2m}\sum_{l=1}^{[\frac{n-1}{2m}]}b_{l-1}(\lambda |x-y|)^{2ml-n}+\lambda^{n-2m}\sum_{j=0}^{\infty}G_j^{1,\pm}(\lambda|x-y|),
				\end{equation}
				where
				\begin{align*}
					G_j^{1,\pm}=
					\begin{cases}
						a_{j/2}^\pm{(\lambda|x-y|)}^{j},&\text{if $j$ is even}\  \text{and} \  j+n\in \Z \setminus 2m\Z ,\\[4pt]
						a_{j/2}^\pm(\lambda|x-y|)^{j}+b_{\frac{j+n}{2m}-1}(\lambda|x-y|)^{j}\log(\lambda|x-y|),&\text{if}\,\,  j+n\in 2m\Z,\\[4pt]
						0,&\text{else}.\\[4pt]
					\end{cases}
				\end{align*}
				Furthermore, the coefficients in  \eqref{eq-expansion-A-odd} and \eqref{eq-expansion-A-2} satisfy $a_j^\pm \in \mathbb{C}\setminus\R$ and $b_j\in \R$.
			\end{lemma}
			\begin{proof}
				We only prove the expansions for ${R}_0^{+}(\lambda_k^{2m})$ since the proof for ${R}_0^{-}(\lambda_k^{2m})$ follows from the same arguments.  
				
				When $n\ge 2$, there exists a $\lambda_0>0$ small enough, such that for each $k=0,1,\cdots,m-1$, we have (see e.g. in \cite[Lemma 2.1]{GW17}, \cite{AT} and \cite{J80})
				\begin{equation}\label{eq-epansion-2od}
					\mathfrak{R}_0^{+}(\lambda_k^2)(x-y)=|x-y|^{2-n}\sum_{j=0}^{\infty}G_j^0(\lambda_k|x-y|),\quad  \ \ 0<\lambda|x-y|<\lambda_0,
				\end{equation}
				where ${\lambda }_{k}=\lambda{e}^{\frac{{\rm i}k\pi}{m}}$, and we have $\mu_j\in \mathbb{R}$,  $\upsilon_j\in \mathbb{C}\setminus\R$.
				Furthermore, if $n \ge 3$ is odd,
				\begin{align*}
					G_j^0=
					\begin{cases}
						\mu_j{(\lambda_k|x-y|)}^{2j},&\text{if}\,\, 0\le j\le \frac{n}{2}-2,\\[4pt]
						\upsilon_j(\lambda_k|x-y|)^{2j}+\mu_j(\lambda_k|x-y|)^{2j+1},&\text{if}\,\,j> \frac{n}{2}-2;\\[4pt]
					\end{cases}
				\end{align*}
				while if $n$ is even,
				\begin{align*}
					G_j^0=
					\begin{cases}
						\mu_j{(\lambda_k|x-y|)}^{2j},&\text{if}\,\, 0\le j\le \frac{n}{2}-2,\\[4pt]
						\upsilon_j(\lambda_k|x-y|)^{2j}+\mu_j(\lambda_k|x-y|)^{2j}\log(\lambda_k|x-y|),&\text{if}\,\,j\ge \frac{n}{2}-1.\\[4pt]
					\end{cases}
				\end{align*}
				When $n=1$, there exists a $\lambda_0>0$  small enough, such that for each $k=0,1,\cdots,m-1$,
				\begin{equation}\label{eq-app-1dim}
					\mathfrak{R}_0^{+}(\lambda_k^2)(x-y)=\lambda_k^{-1}\sum_{j=0}^{\infty}G_j^0(\lambda_k|x-y|),\quad  \ \ 0<\lambda|x-y|<\lambda_0,
				\end{equation}
				where  ${\lambda }_{k}=\lambda{e}^{\frac{{\rm i}k\pi}{m}}$, and
				\begin{align*}
					G_j^0=\upsilon_j(\lambda_k|x-y|)^{2j}+\mu_j(\lambda_k|x-y|)^{2j+1},\ \ j\ge 0.
				\end{align*}
				
				Recall that
				\begin{equation}\label{eq-recall-1}
					{R}_{0}^{\pm }(\lambda^{2m})=\frac{1}{m{\lambda }^{2m}}\displaystyle\sum_{k\in {I}^{\pm }}{{\lambda }_{k}}^{2}\mathfrak{R}_0^{\pm}({{\lambda }_{k}}^{2}).	
				\end{equation}
				Insert \eqref{eq-epansion-2od} and \eqref{eq-app-1dim} into \eqref{eq-recall-1}, we derive that 
				\begin{equation}\label{eq-expansion-A-1}
					{R}_0^{+}(\lambda^{2m})(x-y)=
					\begin{cases}
						\frac{1}{m|x-y|^{n-2}\lambda^{2m}}\sum_{j=0}^{\infty}\sum_{k=0}^{m-1}{{\lambda }_{k}}^{2}G_j^0(\lambda_k|x-y|),&\text{if}\,\, n\ge 2,\\[4pt]
						\frac{1}{m\lambda^{2m}}\sum_{j=0}^{\infty}\sum_{k=0}^{m-1}{{\lambda }_{k}}G_j^0(\lambda_k|x-y|),&\text{if}\,\,n=1.\\[4pt]
					\end{cases}
				\end{equation}
				Note that $$\sum_{k=0}^{m-1}\lambda_k^{2j}=0, \quad \text{for}\, j\in \Z \setminus (m\Z), $$
				and  $$\log(\lambda_k|x-y|)=\log\lambda+\frac{\i k\pi}{m}+\log\lambda|x-y|.$$
				Combining the above three identities, we obtain \eqref{eq-expansion-A-odd} and \eqref{eq-expansion-A-2} as well as the expressions of $G_j^{1,\pm}$, moreover, we have $a_j^\pm \in \mathbb{C}\setminus\R$ and $b_j\in \R$. Therefore the proof of the lemma is complete. 
			\end{proof}

			\begin{proof}[Proof of Lemma \ref{lm2.0}.]
				We first prove \eqref{eq.2-low-odd}, \eqref{eq3.5.0-rem} and \eqref{eq.2-low-even} when $0<\lambda|x-y|<\lambda_0$, where $\lambda_0>0$ is the same as the one in  Lemma \ref{app-01}.
				Set
				\begin{equation}\label{eq-remind-app}
					{r}_{2m-n+1}^{\pm}(\lambda|x-y|)=\sum_{j>m-\frac n2}^{\infty}G_j^{1,\pm}(\lambda|x-y|),    
				\end{equation}
				where $G_j^{1,\pm}(\lambda|x-y|)$ is defined in \eqref{eq-expansion-A-odd} and \eqref{eq-expansion-A-2}.
				Note that the first term on the right hand side of in \eqref{eq-expansion-A-odd} and \eqref{eq-expansion-A-2} does not exist since  $n\le 2m$.
				Thus, the expansion \eqref{eq.2-low-odd} and \eqref{eq.2-low-even} follow from  \eqref{eq-expansion-A-odd} and \eqref{eq-expansion-A-2}  respectively, provided that $0<\lambda|x-y|<\lambda_0$. The estimate \eqref{eq3.5.0-rem} follows from the definition $G_j^{1,\pm}$ in Lemma \ref{app-01}.
				
				It remains to prove that \eqref{eq.2-low-odd},  \eqref{eq3.5.0-rem} and \eqref{eq.2-low-even} hold when $\lambda|x-y|\ge \lambda_0$. By \eqref{eq-asmptotic-hankel} and \eqref{eq-recall-1}, we have  for all $m,n$, 
				\begin{equation}\label{eq-est-res-h-0}
					{R}_{0}^{\pm}({\lambda }^{2m})(x-y)=\frac{\lambda ^{\frac{n+1}{2}-2m}}{|x-y|^{\frac{n-1}{2}}}e^{\pm \i \lambda |x-y|}\tilde\psi_>^{\pm}(\lambda |x-y|),	\qquad \lambda|x-y|\ge \lambda_0,
				\end{equation} 
				where $\mbox{supp}\,\tilde\psi_>^{\pm}\subset[\lambda_0, \infty)$, and for $l\in \N_0$,
				\begin{equation}\label{eq-est-res-1}
					\left| \frac{d^l}{dz^l}\tilde{\psi}_>^\pm (z)\right| \lesssim_l \langle z\rangle^{-l}, \quad z \ge \lambda_0.
				\end{equation}
				If $n$ is odd, we define 
				$${r}_{2m-n+1}^{\pm}(\lambda|x-y|)=\lambda^{2m-n}{R}_{0}^{\pm }(\lambda^{2m})(x-y)  
				-\sum_{j=0}^{m-\frac{n+1}{2}}{a}_{j}^{\pm } (\lambda\left| x-y\right|)^{2j}-{b}_{0}(\lambda\left| x-y\right |)^{2m-n},$$
				and if $n$ is even, we define
				$${r}_{2m-n+1}^{\pm}(\lambda|x-y|)=\lambda^{2m-n}{R}_{0}^{\pm }(\lambda^{2m})(x-y)  
				-\sum_{j=0}^{m-\frac{n}{2}}{a}_{j}^{\pm } (\lambda\left| x-y\right|)^{2j}-{b}_{0}(\lambda\left| x-y\right |)^{2m-n}\log{\left ( \lambda \left | x-y\right |\right )}.$$
				Clearly, \eqref{eq.2-low-odd} and \eqref{eq.2-low-even} hold immediately. Furthermore, by \eqref{eq-est-res-h-0} and \eqref{eq-est-res-1}, it follows  that
				\begin{equation*}
					\lambda^{2m-n}{R}_{0}^{\pm }(\lambda^{2m})(x-y)=\frac{e^{\pm \i \lambda |x-y|}}{(\lambda|x-y|)^{\frac{n-1}{2}}}\,\tilde\psi_>^{\pm}(\lambda |x-y|),\qquad \lambda|x-y|\ge \lambda_0,	
				\end{equation*}
				and $z^{-\frac{n-1}{2}}e^{\pm \i z}\tilde\psi_>^{\pm}(z)$ satisfy that for each $l\in \N_0$,
				\begin{equation*}
					\left| \frac{d^l}{dz^l}\left( z^{-\frac{n-1}{2}}e^{\pm \i z}\,\tilde\psi_>^{\pm}(z)\right) \right|\lesssim_l z^{-\frac{n-1}{2}}, \quad \text{for}\ z\ge \lambda_0.	
				\end{equation*}
				Meanwhile, for each $l\in \N_0$,  if $n$ is odd, we have 
				$$\left| \frac{d^l}{dz^l}\left(\sum_{j=0}^{m-\frac{n+1}{2}}{a}_{j}^{\pm } z^{2j}+{b}_{0}z^{2m-n}\right) \right|\lesssim_l {z^{2m-n+1-l}}, \quad \text{for}\ z\ge \lambda_0, $$
				and if $n$ is even, we deduce that 
				$$\left| \frac{d^l}{dz^l}\left(\sum_{j=0}^{m-\frac{n}{2}}{a}_{j}^{\pm } z^{2j}+{b}_{0}z^{2m-n}\log{z} \right) \right|\lesssim_l {z^{2m-n+1-l}}, \quad \text{for}\ z\ge \lambda_0.$$
				Hence,  ${r}_{2m-n+1}^{\pm}(\lambda|x-y|)$ satisfy the estimate \eqref{eq3.5.0-rem}  when $0<\lambda|x-y|<\lambda_0$. Therefore the proof is complete.
			\end{proof}

			\begin{proof}[{Proof of Lemma \ref{lm-finiterank-5.5}.}]
				We only prove the even dimensional case, the arguments for odd dimensions $n$ follows analogously.
				By \eqref{equ2.1.1.1}, we have
				\begin{equation*}
					R_0(-\lambda^{2m})(x)=\frac{1}{-m\lambda^{2m}}\sum_{k\in I^{\pm}}(e^{\frac{\pm\pi\i}{m}}\lambda^2_k) \mathfrak{R}_0^{\pm}\left(e^{\frac{\pm\pi\i}{m}}\lambda_k^2\right)(x), \qquad \lambda>0.
				\end{equation*}
				Following the proof of Lemma \ref{lm2.0} and setting $\lambda=1$ in \eqref{eq.2-low-even}, we derive
				\begin{equation*}
					R_0(-1)(x-y)=\sum_{j=0}^{m-\frac{n}{2}-1}a_j|x-y|^{2j}+{a}_{m-\frac{n}{2}}{|x-y|}^{2m-n}+b_0\log(|x-y|){|x-y|}^{2m-n}
					+\tilde{r}_{2m-n+1}(|x-y|),
			\end{equation*}
			where the coefficients  $a_j^\pm$ are in \eqref{eq.2-low-even}, the above $a_j$ have the relation
			\begin{equation}\label{eq-A-8}
				a_j=e^{\frac{\pm\pi \i}{m}(n-2m+2j)} a_j^\pm,
			\end{equation}
			and the remainder term $\tilde{r}_{2m-n+1}$ also satisfies \eqref{eq3.5.0-rem}. Consequently, the expansion \eqref{eq5.47} follows directly from the identity
			\begin{equation*}
				|x-  y|^{2j}=\sum_{|\alpha|+|\beta|=2j}C_{\alpha,\beta}(-1)^{|\beta|}x^{\alpha}y^{\beta},
			\end{equation*}
			by defining
			\begin{equation}\label{equp3.18.alpha}
				A_{\alpha,\beta}=a_{\frac{|\alpha|+|\beta|}{2}}(-1)^{|\beta|} C_{\alpha,\beta}.
			\end{equation}
			
			Now  we prove \eqref{eq5.48}, which follows from a series of identities as follows:
			\begin{equation}\label{equ3.35}
				\begin{split}
					(-1)^{|\beta|} A_{\alpha,\beta}\alpha!\beta!
					&=\lim\limits_{x,y\to0}\partial_x^\alpha\partial_y^\beta R_0(-1)(x-y)
					\\&=\lim\limits_{x,y\to0}\frac{1}{(2\pi)^n}\partial_x^\alpha\partial_y^\beta\left(\int_{\R^n}\frac{e^{\i(x-y)\cdot\xi}}{|\xi|^{2m}+1}\,\d\xi\right)
					\\&=\frac{\i^{|\alpha|+|\beta|}(-1)^{|\beta|}}{(2\pi)^n}\int_{\R^n}\frac{\xi^{\alpha+\beta}}{|\xi|^{2m}+1}\, \d\xi,
				\end{split}
			\end{equation}
			where in the first equality presented above, we use  \eqref{eq5.47}; the second equality follows from the integral expression  of $ R_0(-1)$;  and the third equality is justified by the dominated convergence theorem, as well as the fact that $\xi^{\alpha+\beta}/{(|\xi|^{2m}+1)}\in L^1$, since $0\le |\alpha|, |\beta|\le m-\frac{n}{2}-1$.    
		\end{proof}	
		
		\begin{proof}[Proof of Lemma \ref{lem-expansion-with-theta}.]
			The proof closely parallels that of Lemma \ref{lm2.0}  except that we set
			\begin{equation*}
				{r}_{\theta}^{\pm}(\lambda|x-y|)=\sum_{j>[\frac{\theta-1}{2}]}^{\infty}G_j^{1,\pm}(\lambda|x-y|).
			\end{equation*}
			With this definition in place, the remainder of the proof proceeds along the same lines as the argument presented in Lemma \ref{lm2.0}. Therefore, we omit the details for the sake of brevity.
		\end{proof}
		
		\begin{proof}[{Proof of Lemma \ref{lem-expasion-R+-R-1}.}]
			We define 
			\begin{equation*}
				U_0^\pm(\lambda|x-y|)=e^{\mp \i \lambda|x-y|}\lambda^{2m-n}{R}_{0}^{\pm}(\lambda^{2m})(x-y).
			\end{equation*}
			Thus \eqref{eq-R+-R-} holds immediately. On the one hand, when  $\lambda|x-y|<\lambda_0$ (the same constant appeared in Lemma \ref{lem-expan-R- small}), 
			by \eqref{eq-expansion-A-odd} and \eqref{eq-expansion-A-2}, we obtain that 
			\begin{equation*}
				U_0^\pm(\lambda|x-y|)=e^{\mp \i \lambda|x-y|}\left( \sum_{l=1}^{[\frac{n-1}{2m}]}b_{l-1}(\lambda |x-y|)^{2ml-n}+\sum_{j=1}^{\infty}G_j^{1,\pm}(\lambda|x-y|) \right),
			\end{equation*}
			this implies that  \eqref{eq-est-psi0} holds when $\lambda|x-y|<\lambda_0$. On the other hand, when $\lambda|x-y|\ge \lambda_0$, it follows from \eqref{eq-est-res-h-0} that
			\begin{equation*}
				U_0^\pm(\lambda|x-y|)={(\lambda|x-y|)}^{-\frac{n-1}{2}}\tilde\psi_>^{\pm}(\lambda |x-y|),
			\end{equation*}
			this, together with \eqref{eq-est-res-1}, implies  that \eqref{eq-est-psi0} holds when $\lambda|x-y|\ge \lambda_0$.
			Therefore, the proof is complete.
		\end{proof}
		
		\begin{proof}[Proof of Lemma \ref{lm-expansion-high}.]
			First we prove \eqref{eq-expansion-high-1} and \eqref{eq-expansion-high-4}. Set
			\begin{align*}
				U_1^\pm(\lambda|x-y|):=e^{\mp \i \lambda|x-y|}\frac{|x-y|^\frac{n-1}{2}}{\lambda^{\frac{n+1}{2}-2m}}{R}_{0}^{\pm}(\lambda^{2m})(x-y).
			\end{align*}
			Then by Lemma \ref{lem-expan-R- small}, there is a $\lambda_0>0$ such that
			\begin{align*}
				U_1^\pm(\lambda|x-y|)
				=e^{\mp \i \lambda|x-y|}{(\lambda|x-y|)^\frac{n-1}{2}}\left( \sum_{l=1}^{[\frac{n-1}{2m}]}b_{l-1}(\lambda |x-y|)^{2ml-n}+\sum_{j=1}^{\infty}G_j^{1,\pm}(\lambda|x-y|) \right)
			\end{align*}
			holds for $\lambda|x-y|\le \lambda_0$. Moreover, when $1\le n\le 4m-1$, we have $2m-\frac{n+1}{2}>0$. Then
			a direct calculation implies
			\begin{align*}
				\left | \frac{d^l}{dz^l}U_{1}^{\pm}(z)\right | \lesssim_l\  |z|^{-l},\ \ 0<z\le \lambda_0.
			\end{align*}
			This, combined with \eqref{eq-est-res-h-0} and \eqref{eq-est-res-1}, shows \eqref{eq-expansion-high-1} and \eqref{eq-expansion-high-4}.
			
			Next, we prove \eqref{eq-decom-r-nl2m} and \eqref{eq-expansion-high-6}.  Write 
			$${R}_{0}^{\pm}(\lambda^{2m})(x-y)=\phi(\lambda|x-y|){R}_{0}^{\pm}(\lambda^{2m})(x-y)+(1-\phi)(\lambda|x-y|){R}_{0}^{\pm}(\lambda^{2m})(x-y),$$
			where $\phi\in C_c^\infty(\R)$ such that $\phi(t)=1$ for $|t|<\frac 12$ and $\phi(t)=0$ for $|t|>1$.
			Similar to the previous proof, applying Lemma \ref{lem-expan-R- small} to $\phi(\lambda|x-y|){R}_{0}^{\pm}(\lambda^{2m})(x-y)$ and applying \eqref{eq-est-res-h-0}, \eqref{eq-est-res-1} to $(1-\phi)(\lambda|x-y|){R}_{0}^{\pm}(\lambda^{2m})(x-y)$, we then have \eqref{eq-decom-r-nl2m} and \eqref{eq-expansion-high-6}.
			Thus we finish the proof.
		\end{proof}
		
		\section{Proof of Lemma \ref{lm2.4}}\label{app-4}
		Let $\Omega=(0,{\lambda }_{0})$ and $a(\lambda)= \lambda^{2m}$, we write $\Omega=\Omega_1 \cup \Omega_2$ where
		\begin{equation*}
			\Omega_1=\{\lambda\in\Omega;~\mbox{$|\frac{\partial }{\partial \lambda } a(\lambda)+\frac xt|<\frac12|\frac xt|$}\}=\{\lambda\in\Omega;~\mbox{$|2m\lambda^{2m-1}+\frac xt|<\frac12|\frac xt|$}\},
		\end{equation*}
		and
		\begin{equation*}
			\Omega_2=\{\lambda\in\Omega;~\mbox{$|\frac{\partial }{\partial \lambda } a(\lambda)+\frac xt|>\frac14|\frac xt|$}\}=\{\lambda\in\Omega;~\mbox{$|2m\lambda^{2m-1}+\frac xt|>\frac14|\frac xt|$}\}.
		\end{equation*}
		And we define
		\begin{equation*}
			\Phi_1(\lambda)=\Phi\left((2m\lambda^{2m-1}+\mbox{$\frac xt$})/\mbox{$\frac12|\frac xt|$}\right)~\mathrm{and}~\Phi_2(\lambda)=1-\Phi_1(\lambda),
		\end{equation*}
		where $\Phi\in C_0^\infty(\mathbb{R})$ such that $\Phi(\lambda)=1$ when $|\lambda|\leq\frac12$ and $\Phi(\lambda)=0$ when $|\lambda|\geq 1$.
		Now $I(t,x)$ is split into
		\begin{equation*}
			I_j(t,x)=\int_{\Omega_j}e^{\i (t\lambda^{2m}+ \lambda x)}\Phi_j(\lambda)\psi(\lambda,x) \d\lambda,~~j=1, 2.
		\end{equation*}
		We shall prove \eqref{eq-osc-g-1} and \eqref{eq-osc-l-1} with $I$ replaced by $I_j$, $j=1,2$. 
		
		{\bf Step 1. Estimates for $I_2$.}
		We denote two more cut-off functions as 
		$$\eta_1^r(\lambda)=\Phi(\lambda/r),\quad \eta_2^r(\lambda)=1-\Phi(\lambda/r),\quad r>0.$$
		We shall decompose the integral $I_2$ properly and use the integration-by-parts arguments. To proceed, we further split $I_{2}$ into
		\begin{equation*}
			\begin{split}
				I_{2}(t,x)
				&=\int_{\Omega_2}e^{\i(t\lambda^{2m}+\lambda x)}\eta_1^r(\lambda)\Phi_2(\lambda)\psi(\lambda,x) \d\lambda\\
				&+\int_{\Omega_2}e^{\i(t\lambda^{2m}+\lambda x)}\eta_2^r(\lambda)\Phi_2(\lambda)\psi(\lambda,x) \d\lambda\\
				&:= I_{21}(t,x)+I_{22}(t,x).
			\end{split}
		\end{equation*}
		We remark that $I_{22}(t,x)=0$ if $r>2\lambda_0$.
		Obviously, for the term $I_{21}(t,x)$, using the fact that $\sup\limits_x\left | \psi(\lambda,x)\right |\lesssim{\lambda}^{b}$ and $b>-1$,  we have
		\begin{equation}\label{eqC-5}
			|I_{21}|\leq C\int_{|\lambda|\leq r}\lambda^b \d\lambda\leq Cr^{1+b}.
		\end{equation}
		In order to estimate the term $I_{22}(t,x)$, we denote by
		\begin{equation*}
			Df:=\frac{1}{2mt\lambda^{2m-1}+x}\cdot\frac{d}{d\lambda}f,\qquad\quad D_{*}f:=-\frac{d}{d\lambda}(\frac{1}{2mt\lambda^{2m-1}+x}\cdot f).
		\end{equation*}
		Integration by parts $k$ times leads to
		\begin{equation}\label{eqC-3}
			I_{22}=\int_{\Omega_2}e^{\i(t\lambda^{2m}+\lambda x)}D_*^k\left(\eta_2^r(\lambda)\Phi_2(\lambda)\psi(\lambda,x)\right) \d\lambda.
		\end{equation}
		Set $h(\lambda)=t\cdot\frac{\partial }{\partial \lambda }a(\lambda)+x $, and note that $\partial^j (h^{-1}) \ (j \le k)$ can be written as (up to constants):
		\begin{equation}\label{C-2}
			\sum_{l=1}^j\sum_{\substack{\sigma_s\ge 1,\ s=1,\cdots,l \\\sigma_1+\cdots+\sigma_l=j}}\frac{\prod_{s=1}^l \partial^{\sigma_s} h(\lambda)}{(h(\lambda))^{1+l}}.
		\end{equation}
		Moreover, in $\Omega_2$,  we have
		\begin{equation}\label{C-1}
			|\partial a(\lambda)|<5|\partial a(\lambda)+\frac{x}{t}|,
		\end{equation}
		this, together with \eqref{C-2} and $|\frac{\partial }{\partial \lambda } a(\lambda)+\frac xt|>\frac14|\frac xt|$, shows that
		$$|\partial^j  (h^{-1}) |\lesssim |t|^{-1} |\lambda|^{1-2m-j},\ |\partial^j (h^{-1}) |\lesssim |x|^{-1}|\lambda|^{-j},$$
		and therefore
		$$|\partial^j (h^{-1})|\lesssim |t|^{-(1-\theta)} |x|^{-\theta} |\lambda|^{-(1-\theta)(2m-1)-j},$$
		holds for all $\theta \in [0,1]$. Then, it follows that 
		\begin{equation}\label{C-8}
			D_*^k\left(\eta_2^r(\lambda)\Phi_2(\lambda)\psi(\lambda,x)\right)\lesssim |t|^{-k(1-\theta)} |x|^{-k\theta} |\lambda|^{b-k(1-\theta)(2m-1)-k}.
		\end{equation}
		
		First, for \eqref{eq-osc-g-1} with $\mu_b\le 0$ and \eqref{eq-osc-l-1}, plugging \eqref{C-8} into \eqref{eqC-3}, and choose $\theta=0$, we obtain
		\begin{equation}\label{C-4}
			|I_{22}|\lesssim |t|^{-k}\int_{\{r/2<\lambda<\lambda_0\}} \lambda^{b-2mk} \d\lambda \lesssim |t|^{-k}r^{b+1-2mk},
		\end{equation}
		moreover, apply \eqref{eqC-5} and \eqref{C-4} with $r=|t|^{-\frac{1}{2m}}$, a direct computation yields
		\begin{equation}\label{C-10}
			|I_2| \lesssim |I_{21}| +|I_{22}| \lesssim|t|^{-\frac{1+b}{2m}},
		\end{equation}
		meanwhile, note that $|I_2|\lesssim 1$, the above estimate implies \eqref{eq-osc-g-1} with $\mu_b\le 0$ and \eqref{eq-osc-l-1}. 
		
		Next, we will prove that $I_2$ satisfies  \eqref{eq-osc-g-1} with $\mu_b> 0$. 
		Suppose ${t}^{-\frac{1}{2m}}\left | x\right |\ge 1$. If $b+1<k$, we set $\theta=1$, $r=|x|^{-1}$, thus it follows that 
		$$|I_{2}|\lesssim |x|^{-(1+b)}\lesssim |t|^{-\frac{1+b}{2m}}({t}^{-\frac{1}{2m}}\left | x\right |)^{-\frac{m(1+2b)+m-1-b}{2m-1}}\lesssim t^{-\frac{1+b}{2m}}({t}^{-\frac{1}{2m}}\left | x\right |)^{-\mu_b}.$$
		If $b+1\ge k$, we choose $r=(|t|^{-(1-\theta)}|x|^{-\theta})^{\frac{1}{(2m-1)(1-\theta)+1}}$ and $\theta>0$ in \eqref{C-8} such that $$b+1- (2m-1)k(1-\theta)-k=-\varepsilon,\quad 0<\varepsilon<<1.$$ 
		Now, it follows that 
		$$|I_{21}|\lesssim r^{1+b}.$$
		Meanwhile, we have 
		$$|I_{22}|\lesssim r^{-k}\int_{\{r/2<\lambda<\lambda_0\}} \lambda^{b-(2m-1)k(1-\theta)-k} \d\lambda\lesssim r^{1+b}.$$
		Thanks to $|t|^{-\frac {1}{2m}}|x|\le 1$ and the above two equalities, we have 
		\begin{align*}
			|I_{2}|&\lesssim r^{1+b}\lesssim (|t|^{-(1-\theta)}|x|^{-\theta})^{\frac{1+b}{(2m-1)(1-\theta)+1}}
			=\left(|t|^{-(1-\theta-\frac{\theta}{2m})}(|t|^{-\frac {1}{2m}}|x|)^{-\theta}\right)^{\frac{1+b}{(2m-1)(1-\theta)+1}}\\
			&=|t|^{-\frac{1+b}{2m}} (|t|^{-\frac {1}{2m}}|x|)^{-\theta\frac{1+b}{(2m-1)(1-\theta)+1}}=|t|^{-\frac{1+b}{2m}} (|t|^{-\frac {1}{2m}}|x|)^{-\frac{1+b}{b+1+\varepsilon}\frac{2mk-b-1}{2m-1}}\\
			&\lesssim |t|^{-\frac{1+b}{2m}}({t}^{-\frac{1}{2m}}\left | x\right |)^{-\mu_b}.
		\end{align*}
		
		{\bf Step 2. Estimates for $I_1$.}
		First, we have
		\begin{equation}\label{C-Omega_1}
			\Omega_1 \subset  \{\lambda \in \mathbb{R};\  (\frac{1}{4m})^{\frac{1}{2m-1}}r\le |\lambda|  \le (\frac{3}{4m})^{\frac{1}{2m-1}}r\  \} := \{\lambda \in \mathbb{R};\ c_1 r\le |\lambda|  \le c_2 r \},
		\end{equation}
		where $r=|\frac xt|^{\frac{1}{2m-1}}$. Thus, we assume $|t|^{-1}|x|\lesssim 1$ in this step, since otherwise $I_1 =0$.

		If ${t}^{-\frac{1}{2m}}\left | x\right |< 1$,  for $I_{1}$, we have
		\begin{equation}
			|I_{1}|\le \int_{\{ c_1 r\le |\lambda | \le c_2 r\} } |\lambda |^b \d\lambda \lesssim |t|^{-\frac{1+b}{2m-1}}|x|^{\frac{1+b}{2m-1}} \lesssim |t|^{-\frac{1+b}{2m}}({t}^{-\frac{1}{2m}}\left | x\right |)^{\frac{1+b}{2m-1}},
		\end{equation}
		which proves \eqref{eq-osc-l-1} in $\Omega_1$. 
		
		Next, we consider \eqref{eq-osc-g-1} in $\Omega_1$. The change of variables $\lambda \to \lambda r$ in $I_1$ leads to
		\begin{equation*}
			\begin{split}
				|I_{1}|&=\left|r \int_{\left \{ c_1\le |\lambda| \le c_2 \right \} }^{} e^{\i \left (tr^{2m}\lambda^{2m}+\lambda rx \right )}\Phi_1(\lambda r)\psi(\lambda r) \d\lambda \right|\\
				&=\left|r \int_{\left \{ c_1\le |\lambda| \le c_2 \right \} }^{} e^{\i tr^{2m} \left (\lambda^{2m}+\lambda\cdot sgn (x) \right )}\Phi_1(\lambda r)\psi(\lambda r) \d\lambda\right|\\
				&\lesssim r^{1+b}(r^{2m}t)^{-\frac 12}\lesssim |t|^{-\frac{1+b}{2m}}({t}^{-\frac{1}{2m}}\left | x\right |)^{\frac{1+b}{2m-1}},
			\end{split}
		\end{equation*}
		where in the second line, we have used $rx=tr^{2m}\cdot sgn (x)$;  and in the last line, given that $ c_1\le |\lambda| \le c_2$, thus  $\left | \partial _\lambda ^2(\lambda^{2m}+\lambda\cdot sgn (x) ) \right | \gtrsim 1$, then the inequality follows by applying  the van der Corput lemma and by using the fact $\sup\limits_x\left | \psi(\lambda,x)\right |\lesssim{\lambda}^{b}$.
		
		To sum up, we finish the proof of Lemma \ref{lm2.4}.

		\section{Proof of Lemma \ref{lm-4.2-CHHZ}}\label{app-3}
		When  $|x|\le 1$, applying the Schwartz inequality, in combination with the assumption $r-\sigma<-\frac{n}{2}$ and $\k_0\ge r+1$, we obtain the estimate \eqref{eqA.1.1}. Therefore it suffices to consider the case when $|x|\ge 1$ and we divide the proof into three cases.
		
		\noindent\emph{Case 1: $r\geq0$ is even.}\
		We write
		\begin{equation*}\label{eqA.2}
			|y-x|^{r}=\sum_{|\alpha|+|\beta|= r}(-1)^{|\beta|}C_{\alpha,\beta}x^{\alpha}y^{\beta},
		\end{equation*}
		and note that by our assumption,
		\begin{equation*}
			\begin{cases}
				\left| \langle x^{\alpha} y^{\beta}, \ \varphi(y) \rangle\right|
				\lesssim\left\| \varphi \right\|_{L_{\sigma}^2}\left\langle x\right\rangle ^{|\alpha|}\leq\left\|\varphi\right\|_{L_{\sigma}^2}\left\langle x\right\rangle ^{r-\k_0},\quad&\text{if}\,\,|\beta|>\k_0-1,\\
				\langle x^{\alpha}y^{\beta}, \ \varphi(y) \rangle=0,&\text{if}\,\,|\beta|\leq \k_0-1.
			\end{cases}
		\end{equation*}
		Thus we obtain \eqref{eqA.1.1} in this case.
		
		\noindent\emph{Case 2: $r \geq 0$ is odd.}\
		We express $|y-x|^{r}$ as 
		$$|x-y|^{r}=(|x-y|^{r}-|x||x-y|^{r-1})+|x||x-y|^{r-1}.$$ 
		For the term $|x||x-y|^{r-1}$, 
		since $r-1$ is even, it follows from \emph{Case 1} that
		\begin{equation}\label{equ4.707}
			\left| \langle |x||x-\cdot|^{r-1}, \ \varphi(\cdot) \rangle\right| \lesssim \left\langle x\right\rangle ^{r-\k_0}.
		\end{equation}
		For the first term, we further decompose it as
		\begin{align}\label{claimA.1-1202}
			|x-y|^{r}-|x||x-y|^{r-1}
			=\sum_{|\alpha|+|\beta|=r-1}(-1)^{|\beta|}C_{\alpha,\beta} \sum_{i=1}^{n} \left(\frac{2x^{\alpha}x_i}{|x|^{|\alpha|+1}}\frac{|x|^{|\alpha|+1}y^{\beta}{y_i}}{|x|+|x-y|}-\frac{x^{\alpha}}{|x|^{|\alpha|}}\frac{|x|^{|\alpha|}y^{\beta}{y_i}^2}{|x|+|x-y|} \right).
		\end{align}
		Therefore, it suffices to consider  the general form $\frac{|x|^ly^{\gamma}}{(|x|+|x-y|)}$ with  fixed $l \in \N_0$, $\gamma\in {\N_0}^n$ satisfying $r=|\gamma|+l-1\ge 0$, we verify by induction that
		\begin{equation}\label{claimA.1}
			\begin{aligned}
				\frac{|x|^ly^{\gamma}}{(|x|+|x-y|)} = \sum_{s=1}^{1+\k_0-|\gamma|} \sum_{\substack{|\alpha_1|\in\{\k_0, \k_0+1\} \\ |\alpha_1|+h-s=r}} L_{s,\alpha_1}(x) \frac{|x|^hy^{\alpha_1}}{(|x|+|x-y|)^s} 
				+ \sum_{\substack{|\alpha_2|\leq \k_0-1 \\ i+|\alpha_2|=r}} L_{\alpha_2}(x) |x|^{i}y^{\alpha_2},
			\end{aligned}
		\end{equation}		
		where $|L_{s,\alpha_1}(x)|,\ |L_{\alpha_2}(x)|\lesssim 1$.
		
		For the term $\frac{|x|^{|\alpha|+1}y^{\beta}{y_i}}{|x|+|x-y|}$, we apply \eqref{claimA.1} with $l=|\alpha|+1$, $\gamma=\beta+e_i$ and $q=1$;
		for the term $\frac{|x|^{|\alpha|}y^{\beta}{y_i}^2}{|x|+|x-y|}$, we apply \eqref{claimA.1} with $l=|\alpha|$, $\gamma=\beta+2e_i$ and $q=1$.  Note that from \eqref{claimA.1}, we have  $h-s=r-|\alpha_1|$. Moreover, since $|\alpha_1|\in\{\k_0, \k_0+1\}$, we deduce that
		\begin{align}\label{eq-first term}
			\left| \frac{|x|^hy^{\alpha_1}}{(|x|+|x-y|)^s}\right|\le |x|^{h-s}\langle y\rangle^{\alpha_1} \le \langle x\rangle^{r-\k_0}\langle y\rangle^{\k_0+1}.
		\end{align}
		By the assumption on $\varphi$,  we have
		\begin{equation*}
			\left \langle \frac{|x|^hy^{\alpha_1}}{(|x|+|x-y|)^s},\varphi(y) \right \rangle 
			\lesssim \left\| \varphi \right\|_{L_{\sigma}^2}\|\left\langle y\right\rangle^{-\sigma}\left\langle y\right\rangle^{\k_0+1} \|_{L^2} \left\langle x\right\rangle^{r-\k_0}
			\lesssim \left\| \varphi \right\|_{L_{\sigma}^2} \left\langle x\right\rangle^{r-\k_0},
		\end{equation*}
		and 
		\begin{equation*}
			\left \langle |x|^{i}y^{\alpha_2},\varphi(y) \right \rangle = 0.
		\end{equation*}
		The above two imply that
		\begin{equation*}
			\left| \left\langle |x-\cdot|^r-|x||x-\cdot|^{r-1}, \ \varphi(\cdot) \right\rangle\right| \lesssim  \left\| \varphi \right\|_{L_{\sigma}^2} \left\langle x\right\rangle^{r-\k_0}.
		\end{equation*}
		This, together with \eqref{equ4.707}, yields \eqref{eqA.1.1}.

		\noindent\emph{Case 3:} $\frac{1-n}{2}\le r < 0$.\
		In this case, we denote by $k=-r$, thus $0< k \le \frac{n-1}{2}$. For any fixed $ \k_0-1 \in\N_0$, we have the following identity
		\begin{equation}\label{eqA.3}
			|x-y|^{-k}=\sum_{l=0}^{k-1}C_{l}\frac{(|x|-|x-y|)^{\k_0}}{|x|^{\k_0+l}|x-y|^{k-l}}+\sum_{l=0}^{\k_0-1}C'_{l}\frac{(|x|-|x-y|)^{l}}{|x|^{k+l}},
		\end{equation}
		for some $C_{l}, C_{l}^{'}>0$.
		On the one hand, note that by Lemma \ref{lm2.8} and the assumption on $\sigma$, we have
		\begin{equation*}\label{eqA.333}
			\begin{aligned}
				\left| \left\langle \frac{(|x|-|x-\cdot|)^{\k_0}}{|x|^{\k_0+l}|x-\cdot|^{k-l}}, \ \varphi (\cdot) \right\rangle\right|
				\lesssim& \left\| \varphi \right\|_{L_{\sigma}^2}\left\langle x\right\rangle^{-\k_0 -l}\left(\int_{\mathbb{R}^n}{\langle y\rangle^{-2(\sigma-\k_0)}|x-y|^{-2(k-l)}\d y}\right)^{\frac12}\\
				\lesssim& \left\| \varphi \right\|_{L_{\sigma}^2} \left\langle x\right\rangle^{r-\k_0}.
			\end{aligned}
		\end{equation*}
		On the other hand, the binomial theorem gives
		$$\frac{(|x|-|x-y|)^{l}}{|x|^{k+l}}=\sum_{s=0}^{l}\frac{(-1)^{s}l!}{s!(l-s)!}\frac{|x|^{s}|x-y|^{l-s}}{|x|^{k+l}}.$$
		Then by \emph{Case 1} and \emph{Case 2}, we have for each $s\in\{0,\cdots,l\}$ that
		\begin{equation*}\label{eqA.3333}
			\begin{aligned}
				|x|^{s-k-l}\left| \left\langle |x-\cdot|^{l-s}, \ \varphi (\cdot) \right\rangle\right|
				\lesssim& \left\| \varphi \right\|_{L_{\sigma}^2} \left\langle x\right\rangle^{r-\k_0},
			\end{aligned}
		\end{equation*}
		In view of \eqref{eqA.3}, \eqref{eqA.1.1} follows by combining the above two estimates.

	\end{appendix}

	\section*{Acknowledgements}
	S. Huang was supported by the National Natural Science Foundation of China under the grants 12171178 and 12171442.


\begin{thebibliography}{plain}
		
		
		\bibitem{BHM}  A. Barton, S. Hofmann,  S. Mayboroda, The Neumann problem for higher order elliptic equations with symmetric
		coefficients. \textit{Math. Ann.} 371 (1-2) (2018) 297-336.
		
		\bibitem{CLM} R. Carles, W. Lucha, E. Moulay, Higher-order Schr\"{o}dinger and Hartree–Fock equations.\textit{ J. Math. Phys.}  56 (12) (2015) 122301.
		
		\bibitem{CM}   R.  Carles,  E. Moulay,  Higher order Schr\"{o}dinger equations. \textit{J. Phys. A: Math. Theor}. 45 (39) (2012) 395304.
		
		\bibitem{CHZ} H. Cheng, S. Huang, Q. Zheng, Dispersive estimates for the Schr\"{o}dinger equation with finite rank perturbations. \textit{ Adv. Math. } 426 (2023) 109105, 91pp.
		
		\bibitem{CHHZ} H. Cheng, S. Huang, T. Huang, Q. Zheng, Pointwise estimates for the fundamental solution of higher order Schr\"{o}dinger equations in odd dimensions. Preprint, \textit{arxiv.} 2401.04969.
		
		%
		%
		%
		%
		%
		%
		%
		\bibitem{EG10}  M. B. Erdo\v{g}an, W. R. Green, Dispersive estimates for the Schr\"{o}dinger equation for $C^{\frac{n-3}{2}}$ potentials in odd dimensions. \textit{Int. Math. Res. Not. IMRN}. 13 (2010) 2532-2565.
		
		\bibitem{EGG23} M. B. Erdo\v{g}an, M. Goldberg, and W. R. Green, Counterexamples to $L^p$ boundedness of wave operators for classical
		and higher order Schr\"{o}dinger operators. \textit{J. Funct. Anal.} 285 (5) (2023) 110008.
		
		\bibitem{EGG822} M. B. Erdo\v{g}an, M. Goldberg, W. R. Green,  Dispersive estimates for higher order Schr\"{o}dinger operators with scaling-critical potentials. Preprint, \textit{arXiv.} 2308.11745.
		
		\bibitem{EGT}  M. B. Erdo\v{g}an, W. R. Green,  E. Toprak, On the fourth order Schr\"{o}dinger equation in three dimensions: dispersive estimates and zero energy resonances. \textit{J. Differential Equations}. 271 (2021) 152-185.
		
		\bibitem{EG22} M. B. Erdo\v{g}an, W. R. Green,  The $L^p$-continuity of wave operators for higher order Schr\"{o}dinger operators.  \textit{Adv. Math.} 404 (2022) 108450.
		
		\bibitem{EG23}  M. B. Erdo\v{g}an, W. R. Green, A note on endpoint $L^p$-continuity of wave operators for classical and higher
		order Schr\"{o}dinger operators. \textit{J. Differential Equations}. 355 (2023) 144-161.
		
		
		\bibitem{FSY} H. Feng,  A. Soffer,  X. Yao, Decay estimates and Strichartz estimates of fourth-order Schr\"{o}dinger operator. \textit{J. Funct. Anal.} 274 (2) (2018) 605-658.
		
		\bibitem{FSWY} H. Feng, A. Soffer, Z. Wu,  X. Yao, Decay estimates for higher-order elliptic operators. \textit{Trans. Amer. Math. Soc.} 373 (4) (2020) 2805-2859.
		
		\bibitem{gra} L. Grafakos, Classical Fourier Analysis.  Graduate Texts in Mathematics, second ed., vol. 249, New York: Springer, 2008.
		
		
		%
		\bibitem{GV}  M. Goldberg, M. Visan, A counterexample to dispersive estimates for Schr\"{o}dinger operators in higher dimensions. \textit{Comm. Math. Phys.} 266 (2006) 211-238.
		
		\bibitem{GW15} M. Goldberg, W. R. Green, Dispersive estimates for higher dimensional Schr\"{o}dinger operators with threshold eigenvalues \uppercase\expandafter{\romannumeral1}: The odd dimensional case.  \textit{J. Funct. Anal.} 269 (3)  (2015) 633-682.
		
		\bibitem{GW17} M. Goldberg, W. R. Green, Dispersive estimates for higher dimensional Schr\"{o}dinger operators with threshold eigenvalues \uppercase\expandafter{\romannumeral2}. The even dimensional case. \textit{J. Spectr. Theory.} 7 (1) 
		(2017) 33-86.
		
		\bibitem{GT}  W. Green, E. Toprak, On the Fourth order Schr\"odinger equation in four dimensions: dispersive estimates and
		zero energy resonances. \textit{J. Differential Equations}. 267 (3) (2019) 1899-1954.
		
		
		\bibitem{H-Y-Z} S. Huang, X. Yao, Q. Zheng,  Remarks on $L^p$-limiting absorption principle of Schr\"{o}dinger operators and applications to spectral multiplier theorems. \textit{Forum Math.} 30 (1) (2018) 43-55.
		
		\bibitem{HHZ} T. Huang, S. Huang, Q. Zheng, Inhomogeneous oscillatory integrals and global smoothing effects for dispersive equations. \textit{J. Differential Equations}. 263 (12) (2017)   8606-8629.
		
		\bibitem{J80} A. Jensen, Spectral properties of Schr\"{o}dinger operators and time-decay of the wave functions results in $L^2(\mathbb{R}^m)$, $m\ge 5$. \textit{Duke Math. J.} 47 (1) (1980) 57-80.
		
		\bibitem{AT} A. Jensen, T. Kato, Spectral properties of Schr\"{o}dinger operators and timedecay of the wave functions. \textit{Duke Math. J.} 46 (3) (1979) 583-611.
		
		
		\bibitem{JSS} J.-L. Journ\'{e}, A. Soffer, C. D. Sogge, Decay estimates for Schr\"{o}dinger operators. \textit{Comm. Pure Appl. Math.} 44 (5) (1991) 573-604.
		
		
		\bibitem{KS}   V. Karpman, A.  Shagalov, Stability of solitons described by nonlinear Schr\"{o}dinger-type equations with higher-order dispersion. \textit{Phys. D} 144 (2000) 194-210.
		
		\bibitem{KT}  M. Keel,  T. Tao, Endpoint Strichartz estimates. \textit{Amer. J. Math.} 120 (5) (1998) 955-980.
		
		\bibitem{KPV} C.E. Kenig, G. Ponce, L. Vega,  Oscillatory integrals and regularity of dispersive equations. \textit{Indiana Univ. Math. J.} 40 (1991) 33–69. 
		
		\bibitem{KK}  A. Komech, E. Kopylova, Dispersion decay and scattering theory. John Wiley \& Sons, Inc., Hoboken, NJ, 2012.
		
		\bibitem{LSY}    P. Li, A. Soffer, X. Yao, Decay estimates for fourth-order Schr\"{o}dinger operators in dimension two. \textit{J. Funct. Anal.} 284 (2023) 109816.
		
		\bibitem{Liaw}  C. Liaw,  Rank one and finite rank perturbations-survey and open problems. Preprint, \textit{arXiv:} 1205.4376v1.
		
		\bibitem{MM}  S. Mayboroda,  V. Maz'ya, Regularity of solutions to the polyharmonic equation in general domains. \textit{Invent. Math.} 196 (2014) 1-68.
		
		\bibitem{Mi} A. Miyachi, On some estimates for the wave equation in $L^p$ and $H^p$.  \textit{J. Fac. Sci. Univ. Tokyo.} 27 (1980) 331-354.
		
		\bibitem{MWY}   H. Mizutani, Z. Wan, X. Yao, $L^p$-boundedness of wave operators for fourth-order Schr\"{o}dinger operators on
		the line. \textit{ Adv. Math.} 451 (2024) 109806, 68 pp.
		
		\bibitem{NS}   F. Nier, A. Soffer, Dispersion and Strichartz estimates for some finite rank perturbations of the Laplace operator.  \textit{J. Funct. Anal.} 198 (2) (2003) 511-535.
		
		
		\bibitem{RS} I. Rodnianski, W. Schlag, Time decay for solutions of Schr\"{o}dinger equations with rough and time-dependent potentials. \textit{Invent. Math.} 155 
		(3) (2004) 451-513.
		
		
		\bibitem{Sch} W. Schlag, Dispersive estimates for Schr\"{o}dinger operators in dimension two. \textit{Comm. Math. Phys.} 257 (1) (2005) 87-117.
		
		\bibitem{Sch07}  W. Schlag, Dispersive estimates for Schr\"{o}dinger operators: a survey. Mathematical aspects of nonlinear dispersive equations. \textit{Ann. of Math. Stud.} 163 (2007), 255-285.  Princeton Univ. Press, Princeton, NJ.
		
		\bibitem{Sch21} W. Schlag, On pointwise decay of waves. \textit{J. Math. Phys.} 62 (2021) 061509.
		
		\bibitem{SW} B. Simon, T. Wolff, Singular continuous spectrum under rank one perturbations and localization
		for random Hamiltonians. \textit{Comm. Pure Appl. Math.}  39 (1) (1986) 75-90.
		
		\bibitem{Simon}  B. Simon, Spectral analysis of rank one perturbations and applications, in: Mathematical Quantum
		Theory. \uppercase\expandafter{\romannumeral2}. Schr\"{o}dinger operators, Vancouver, BC, 1993, American Mathematical Society, Providence, RI, 1995, 109-149
		
		\bibitem{SWY} A. Soffer, Z. Wu, X. Yao, Decay estimates for bi-Schr\"{o}dinger operators in dimension one. \textit{Ann. Henri Poincar\'{e}}.
		23 (8) (2022) 2683-2744.
		
		
		
		\bibitem{Wey} H. Weyl, \"{U}ber gew\"{o}hnliche Differentialgleichungen mit Singularit\"{a}ten und die zugeh\"{o}rigen Entwicklungen willk\"{u}rlicher Funktionen. \textit{Math. Ann.}  68 (2) (1910) 220-269.
		
		\bibitem{Ya} K. Yajima, The $W^{k,p}$ continuity of wave operators for Schr\"{o}dinger operators. \textit{J. Math. Soc. Japan}. 47 (3) (1995) 551-581.
		
		
		
		
	\end{thebibliography}
\end{document}